\documentclass[11pt]{article}
\usepackage[T1]{fontenc}
\usepackage{lmodern}
\usepackage[margin=1in]{geometry}
\usepackage{amsmath,amssymb,amsthm,mathtools}
\usepackage{booktabs,graphicx}
\usepackage{float}
\usepackage{xcolor}
\usepackage[colorlinks=true,allcolors=blue!55!black]{hyperref}
\usepackage{enumitem}
\setlist{nosep}
\newtheorem{theorem}{Theorem}[section]
\newtheorem{proposition}[theorem]{Proposition}
\newtheorem{lemma}[theorem]{Lemma}
\newtheorem{corollary}[theorem]{Corollary}
\theoremstyle{definition}
\newtheorem{assumption}[theorem]{Assumption}
\newtheorem{remark}[theorem]{Remark}
\newcommand{\R}{\mathbb R}
\newcommand{\E}{\mathbb E}
\newcommand{\dd}{\mathrm d}
\newcommand{\KL}{\operatorname{KL}}
\newcommand{\Id}{I}
\newcommand{\cH}{\mathcal H}
\newcommand{\cI}{\mathcal I}
\newcommand{\eps}{\varepsilon}

\title{Global Convergence of Third-Order Langevin Dynamics for Non-Convex Optimization via Simulated Annealing}
\author{
Yingli Wang\thanks{School of Mathematical Sciences, Fudan University, Shanghai, People's Republic of China; \texttt{yingliwang@fudan.edu.cn}}
\and
Lingjiong Zhu\thanks{Department of Mathematics, Florida State University, Tallahassee, Florida, United States of America; \texttt{zhu@math.fsu.edu}}}
\date{\today}

\begin{document}
\maketitle

\begin{abstract}
We study global convergence guarantees 
of third-order Langevin dynamics for non-convex optimization via simulated annealing
with fixed friction and decreasing noise.  An explicit three-block distorted entropy
transfers dissipation from the noisy auxiliary variable to the full state.
Under dissipativity, regularity, and low-temperature functional-inequality
assumptions, logarithmic cooling drives the objective values to the global
minimum in probability at the barrier-controlled kinetic rate.  For the
exact-force-integral and midpoint three-stage discretizations, polynomially
decreasing steps preserve this rate on the physical time scale.  The cubic
local endpoint estimate gives a less restrictive sufficient step-size condition
than the available frozen-force kinetic result.  A comparison with the
one-gradient UBU integrator shows how its centered stochastic local error leads,
under the same strong-coupling analysis, to a smaller sufficient iteration
exponent. Numerical experiments
are conducted to illustrate our theory. For a double well objective, third-order Langevin terminal-success point
estimates are higher than UBU at both a common horizon and an equal gradient
budget. For a high-dimensional nonconvex neural-network objective using synthetic data,
independently tuned UBU and third-order Langevin schemes both outperform overdamped
Langevin dynamics; the third-order Langevin point estimate is higher.  
For the same neural-network objective on real data, we show the same point-estimate
ordering for best-basin probability and post-quench test accuracy.\footnote{Numerical
code and associated experiment results are publicly available at
\url{https://github.com/gagawjbytw/simulated-annealing-third-order-langevin}.}
\end{abstract}

\section{Introduction}
Consider the non-convex optimization problem:
\begin{equation}\label{opt:problem}
\min_{x\in\mathbb{R}^{d}}U(x),
\end{equation}
where $U:\mathbb{R}^{d}\rightarrow\mathbb{R}$ is the given objective. 
The optimization problem \eqref{opt:problem} appears
in many applications in machine learning, including
empirical risk minimizations in deep learning where $U$ is typically non-convex.
Many algorithms have been used to solve the optimization problem \eqref{opt:problem}. 
Among these, gradient descent, stochastic gradient and their variance-reduced or momentum-based variants come with guarantees for finding a local minimizer or a stationary point for non-convex problems. 
While convergence to a local minimum can be satisfactory in some settings 
\cite{ge2017learning,du2017gradient}, 
methods with global convergence guarantees are also desirable and preferable in many other settings \cite{hazan2016graduated,simsekli-async-MCMC-18}.

\paragraph{Simulated annealing.}
Simulated annealing replaces direct descent in a nonconvex landscape by a
time-inhomogeneous stochastic dynamics whose instantaneous equilibrium is a Gibbs
law. The most classical one is based on the \textit{overdamped Langevin dynamics} (OLD): 
\begin{align}\label{eqn:OLD}
\dd X_{t}=-\nabla U(X_{t})\dd t+\sqrt{2\varepsilon_{t}}\dd B_{t},
\end{align}
where $B_{t}$ is a standard Brownian motion and $\varepsilon_{t}$ is a temperature parameter
that depends on $t$.
The simulated annealing method was introduced as a computational principle in
\cite{KirkpatrickGelattVecchi1983}. Its diffusion formulation was developed in
\cite{GemanHwang1986,ChiangHwangSheu1987,Royer1989}: at a fixed temperature
$\eps_{t}\equiv\eps$, the stationary position law is proportional to $e^{-U/\eps}$, while a decreasing
temperature must retain enough thermal energy for the process to cross the deepest
relevant barriers. This competition leads to the logarithmic cooling scale. Spectral
gap and Sobolev methods identify the corresponding low-temperature obstruction
\cite{HolleyStroock1988,HolleyKusuokaStroock1989}, and Miclo's free-energy analysis
\cite{Miclo1992} gives quantitative convergence. Later work extended the theory to
weaker confinement and slowly growing potentials
\cite{Zitt2008,FournierTardif2021}. Tang and Zhou \cite{TangZhou2023} obtained
convergence rates for continuous-time overdamped annealing using
Eyring--Kramers estimates. Tang, Wu, and Zhou \cite{TangWuZhou2024} established
corresponding rates and step-size conditions for its discrete-time counterpart.
These results identify the landscape-dependent critical cooling constant and
give discrete-time convergence guarantees that control accumulated
discretization error on the same physical-time scale.

\paragraph{Kinetic annealing and hypocoercivity.}
There have been many variants of overdamped Langevin dynamics, 
and among these, the momentum-based kinetic or \textit{underdamped Langevin dynamics} (ULD) is
one of the most popular:
\begin{align}\label{eqn:ULD}
 \dd X_t&=V_t\dd t,\qquad
 \dd V_t=(-\nabla U(X_t)-\gamma V_t)\dd t
                 +\sqrt{2\gamma\eps_t}\dd B_t,
\end{align}
where $\gamma>0$ is the friction coefficient, $\varepsilon_{t}$ is a temperature parameter
that depends on $t$,
and $B_{t}$ is a standard $d$-dimensional Brownian motion.
At a fixed temperature
$\eps_{t}\equiv\eps$, the stationary law is proportional to $e^{-\frac{1}{\eps}(U(x)+\frac{1}{2}|v|^{2})}$
whose position marginal law is proportional to $e^{-U/\eps}$, 
the same as in the overdamped case.
Adding momentum can improve mixing at a fixed temperature, but it makes the noise
degenerate: entropy is dissipated directly only in the velocity variable. The
hypocoercive framework transfers this dissipation to the full state through the
transport structure; see
\cite{ArnoldEtAl2001,Herau2007,Villani2009}. Lyapunov and coupling methods give
ergodicity and quantitative contraction for kinetic Langevin dynamics under broad
conditions \cite{MattinglyStuartHigham2002,EberleGuillinZimmer2019}, while
nonasymptotic sampling bounds and discretizations are developed in
\cite{ChengEtAl2018,MaEtAl2021}.

For simulated annealing, Monmarch\'e \cite{Monmarche2018} introduced a distorted
entropy adapted to a kinetic process in which the friction grows as the temperature
decreases. Journel and Monmarch\'e \cite{JournelMonmarche2022} treated the
fixed-friction, vanishing-noise formulation under substantially weaker potential
conditions and also proved failure results for cooling schedules that are too fast.
He, Tan, and Wu \cite{HeTanWu2024} obtained quantitative convergence rates for that
formulation and for a particular time discretization. Their calculation adapts the
matrix-valued entropic method of Ma et al.\ \cite{MaEtAl2021}; related discrete
Hamiltonian entropy estimates are given in \cite{Monmarche2024}. At low temperature,
the rate is controlled by a logarithmic Sobolev constant. Landscape decompositions
such as those of Menz and Schlichting \cite{MenzSchlichting2014} provide its
exponential scale, while the
Holley--Stroock perturbation principle \cite{HolleyStroock1987} provides a simple
explicit bound for bounded perturbations of a quadratic potential. Time-dependent
Fisher-information techniques offer another perspective on inhomogeneous diffusions
\cite{FengZuoLi2024}.

\paragraph{Third-order and generalized Langevin dynamics.}
Generalized Langevin equations introduce memory or auxiliary variables while
preserving the desired Gibbs marginal. Their Markovian embeddings and scaling
limits have been studied, for example, in
\cite{OttobrePavliotis2011,PavliotisStoltzVaes2021}. The \textit{third-order Langevin dynamics} of
Mou et al. \cite{MouEtAl2021}, equation (8), uses two auxiliary variables and places
noise only in the last component:  
\begin{equation}\label{eqn:third:order}
\begin{aligned}
 \dd X_t&=V_t\dd t,\\
 \dd V_t&=(-\nabla U(X_t)+\lambda Z_t)\dd t,\\
 \dd Z_t&=(-\lambda V_t-\gamma Z_t)\dd t
                 +\sqrt{2\gamma\eps_t}\,\dd B_t,
\end{aligned}
\end{equation}
where $\lambda,\gamma>0$, $\varepsilon_{t}$ is a temperature parameter
that depends on $t$, and 
$B_{t}$ is a standard $d$-dimensional Brownian motion.
The resulting smoother force integration supports
accelerated sampling guarantees for strongly log-concave targets and a specific
three-stage numerical scheme. Those guarantees apply under strong convexity and a
positive lower Hessian bound. The third-order Langevin dynamics \eqref{eqn:third:order} 
and its generalizations have been previously studied in the context
of sampling \cite{MouEtAl2021,DangGurbuzbalabanIslamYaoZhu2025,high-order-Liu-2025}. 
The nonconvex global-minimization setting \eqref{opt:problem} considered
in this paper using simulated annealing analysis to the best of our knowledge is novel
for the third-order Langevin dynamics \eqref{eqn:third:order}.

\paragraph{Global optimization.}
Several related lines of work pursue global optimization with Langevin-type
processes through fixed-temperature, discretized, and stochastic gradient
formulations.
Nonasymptotic guarantees for fixed-temperature or discretized diffusions appear in
\cite{RaginskyRakhlinTelgarsky2017,ErdogduMackeyShamir2018,xu2018global,Chau2019,Zhang2019}; momentum-based
stochastic-gradient Hamiltonian methods are analyzed in
\cite{GaoGurbuzbalabanZhu2022,Chau2022}. Temperature selection has also been formulated as
a state-dependent stochastic control problem \cite{GaoXuZhou2022}. More recently,
Herty and Zanella \cite{HertyZanella2026} proposed a feedback-controlled temperature
for an interacting kinetic particle model and proved entropy decay for its
Boltzmann-type mean-field description. Together, these approaches clarify the
broader role of thermal exploration. The present model follows a logarithmically
cooled diffusion with its associated physical time scale.

\paragraph{Eyring--Kramers laws and metastability.}
At a fixed low temperature, the Arrhenius exponent determines the dominant
transition scale, whereas the Eyring--Kramers law also identifies the sharp
prefactor. For reversible overdamped Langevin dynamics, the potential-theoretic
approach of Bovier et al.\ \cite{BovierEtAl2004} gives the classical formula
for capacities and mean exit times. Bouchet and Reygner
\cite{BouchetReygner2016} treat non-reversible Langevin dynamics, while the
non-selfadjoint capacity principles of Landim, Mariani, and Seo
\cite{LandimMarianiSeo2019} and the Eyring--Kramers formula of Lee and Seo
\cite{LeeSeo2022} develop the corresponding nonreversible Gibbs theory.

On the
spectral side, Bony, Le Peutrec, and Michel \cite{BonyLePeutrecMichel2025}
establish Eyring--Kramers formulas for a broad class of Fokker--Planck type
operators with Gibbs stationary measures.
The hypoelliptic underdamped process requires additional ideas because the
elliptic potential theory does not apply directly. Lee, Ramil, and Seo
\cite{LeeRamilSeo2026} recently derived the mean-transition-time
Eyring--Kramers law for underdamped Langevin dynamics in a double-well
landscape. This result provides the direct second-order benchmark for the
third-order metastability theory described next.

Recent work develops a fixed-temperature metastability theory for the same
hypoelliptic chain. He et al.\ \cite{HeLiWangZhu2026} construct weak
equilibrium measures and capacities and prove capacity--hitting identities. Using
this framework, Wang and Zhu \cite{WangZhu2026} prove an Eyring--Kramers law for a
double-well potential with a unique index-one saddle. The usual Arrhenius exponent
is retained, but under matched kinetic normalizations the sharp transition-time
prefactor is strictly smaller than its underdamped counterpart. Their argument
extends the underdamped strategy of Lee, Ramil, and Seo \cite{LeeRamilSeo2026} to the step-three chain.
Together, these results provide fixed-temperature metastable transition
asymptotics. We develop the
corresponding time-inhomogeneous simulated annealing convergence theory below.

Diffusion annealing lowers the temperature by reducing noise while keeping
the drift fixed \cite{GemanHwang1986,ChiangHwangSheu1987,Royer1989}. For kinetic
Langevin annealing, relative entropy directly dissipates only the velocity component
of the Fisher information \cite[Section~2.2]{HeTanWu2024}. The same degeneracy is
present in the higher-order chain of Mou et al.\ \cite{MouEtAl2021}, where noise acts
only on the last component; when its amplitude is cooled, the coefficient of this
direct dissipation vanishes.

Chak, Kantas, and Pavliotis
\cite{ChakKantasPavliotis2023}, equation (1.3) and Theorem 2.7, instead analyze
\begin{equation}\label{eq:ckp}
 \dd X_t=Y_t\dd t,\qquad
 \dd Y_t=\left(-\nabla U(X_t)+\Lambda^\top Z_t\right)\dd t,\qquad
 \dd Z_t=\left(-\Lambda Y_t-T_t^{-1}A Z_t\right)\dd t+\Sigma\dd B_t,
\end{equation}
with $\Sigma\Sigma^\top=2A$, where $B_{t}$ is a standard $d$-dimensional Brownian motion.
This parametrization keeps the noise matrix fixed and increases the effective
friction $T_t^{-1}A$ as the temperature decreases. At each frozen temperature it
has the desired Gibbs law. The authors adopt this parametrization to follow the
hypocoercive framework of Monmarch\'e \cite{Monmarche2018}; the additional degeneracy
of the third-order chain is handled by a tailored distorted entropy and approximation
by nondegenerate diffusions. They prove a continuous-time annealing rate and present
numerical experiments.

We combine the fixed-friction annealing strategy of He, Tan, and Wu
\cite{HeTanWu2024} with the third-order diffusion and three-stage scheme of Mou et
al.\ \cite{MouEtAl2021}. 
He, Tan, and Wu \cite{HeTanWu2024} return to the fixed-friction formulation with
vanishing noise. They replace relative entropy by a temperature-dependent distorted
entropy whose entropy weight is of order $\eps^{-1}$. A matrix hypocoercivity
calculation adapted from Ma et al.\ \cite{MaEtAl2021} transfers the
velocity-direction dissipation to the position variable, while the low-temperature
\textit{logarithmic Sobolev inequality} (LSI) converts the resulting Fisher-information bound into frozen-temperature
contraction. Uniform
moment bounds control the temperature-variation term, so logarithmic cooling
preserves the contraction needed for annealing. We apply this strategy to the
additional variable $Z$: an explicit
three-block Fisher-information matrix transfers dissipation from $Z$ through $V$
to $X$ while $\lambda$ and $\gamma$ remain fixed and the noise amplitude decreases.
We also establish convergence for the associated time-discrete schemes.
Higher-order annealed Langevin methods have been used for posterior sampling in
linear inverse problems \cite{ZilbersteinEtAl2024}; here we study the nonconvex
zero-temperature limit.

\paragraph{Contribution.}
Our main contributions can be summarized as follows.

\begin{enumerate}[label=(\roman*)]
\item An explicit three-block dissipation certificate
(Lemma~\ref{lem:matrix}) yields the continuous-time annealing rate in
Theorem~\ref{thm:main}.  Lemma~\ref{lem:closure} extends the entropy argument to
rougher initial laws, and Proposition~\ref{prop:restart} provides the uniform
restart estimate needed for discrete comparison.

\item Local coupling, moment control, grid-point regularization, and
exact-transition smoothing yield Theorem~\ref{thm:discrete} for the exact-force
and midpoint schemes.  Corollary~\ref{cor:iterationgain} gives the resulting
rate-preserving iteration gain.  Applying the same recent-window argument to
UBU splitting scheme in Theorem~\ref{thm:ubu} shows how the additional auxiliary level delays
noise propagation to position and gives the third-order midpoint scheme a
sharper sufficient discretization bound.

\item Three reproducible annealing experiments compare overdamped, underdamped,
and third-order Langevin dynamics after validation-set tuning and independent
evaluation.  A tilted double well uses theory-driven cooling and step grids; an
$88$-dimensional regularized two-layer tanh objective on synthetic data uses independent two-stage
parameter tuning; and a $248$-dimensional example evaluates the same
nonconvex objective on real data and a fixed held-out split.  Both kinetic
schemes outperform overdamped Langevin dynamics in the synthetic and real-data
neural-network examples.  The third-order Langevin best-basin point estimates are higher
than UBU in both examples; see Section~\ref{sec:numerics}.
\end{enumerate}

The resulting exponents retain the same landscape barrier $D$ that appears in the
low-temperature LSI. At the continuous-time level, Theorem~\ref{thm:main} matches
the polynomial exponent of the kinetic annealing rate. The distinct quantitative
gain proved here appears in the discretization: for step size $\eta$, the
$\mathcal{O}(\eta^3)$ local endpoint comparison permits a less restrictive sufficient step-size threshold than the
kinetic result of He, Tan, and Wu \cite{HeTanWu2024}. Sharp metastable prefactors offer a
complementary way to distinguish the dynamics beyond the common barrier exponent.
Table~\ref{tab:introdiscretecomparison} summarizes the comparison for discretized algorithms.
Write $\vartheta:=D/E<1$ and
\[
 r_{\rm c}(\delta):=\min\left\{\frac{\delta}{E},
                         \frac{1-\vartheta}{2}\right\},
\]
the common continuous-time exponent. For polynomial grids
$\eta_k\asymp s_k^{-a}$, where $s_k=\tau_k+t_0$ is shifted accumulated physical
time, the currently available sufficient conditions give the following
rate-preserving thresholds
and iteration exponents.
The third column in Table~\ref{tab:introdiscretecomparison} follows from
$\tau_k=\sum_{j=0}^{k-1}\eta_j$ and $\eta_j\asymp s_j^{-a}$, which imply
$s_k\asymp k^{1/(1+a)}$. Hence a physical-time estimate
$\mathbb P(U(X_k)>\delta)\lesssim s_k^{-r_{\rm c}(\delta)+\alpha}$
becomes, after an arbitrarily small adjustment of $\alpha$,
\[
 \mathbb P(U(X_k)>\delta)
 \lesssim k^{-r_{\rm c}(\delta)/(1+a)+\alpha},
\]
where $X_{k}$ is the $k$-th iterate of the corresponding algorithm and without loss of generality we assume that
$U$ has minimum zero.
The fourth column in Table~\ref{tab:introdiscretecomparison} uses the smallest value of $a$ allowed by the
rate-preserving condition in the second column, since this maximizes the
resulting iteration exponent.

\begin{table}[ht]
\centering
\small
\renewcommand{\arraystretch}{1.25}
\begin{tabular}{@{}p{0.27\textwidth}p{0.27\textwidth}p{0.18\textwidth}p{0.18\textwidth}@{}}
\toprule
Method and source & Sufficient rate-preserving grid condition
& Physical time at\newline the threshold & Proved iteration\newline exponent \\
\midrule
OLD (Euler)\newline
\cite{TangWuZhou2024}
& $\eta_ks_k=\mathcal{O}(1)$; equivalently $a\geq1$
& $s_k\asymp k^{1/2}$
& $r_{\rm c}(\delta)/2$ \\
ULD (frozen-force)\newline
\cite[Theorem~2.8]{HeTanWu2024}
& $\eta_k\sqrt{s_k}=\mathcal{O}(1)$; equivalently $a\geq1/2$
& $s_k\asymp k^{2/3}$
& $2r_{\rm c}(\delta)/3$ \\
ULD (UBU)\newline
(Corollary~\ref{cor:ubuiteration})
& $a\geq(1+\vartheta)/3$
& $s_k\asymp k^{3/(4+\vartheta)}$
& $3r_{\rm c}(\delta)/(4+\vartheta)$ \\
Third-order LD (midpoint)\newline
(Corollary~\ref{cor:iterationgain})
& $a\geq a_*=(1+\vartheta)/4$
& $s_k\asymp k^{4/(5+\vartheta)}$
& $4r_{\rm c}(\delta)/(5+\vartheta)$ \\
\bottomrule
\end{tabular}
\caption{Comparison of available and presently proved sufficient grid conditions
for discrete annealing. The last two columns use the smallest
rate-preserving step exponent in each row. Arbitrarily small losses in the
iteration exponents are suppressed.}
\label{tab:introdiscretecomparison}
\end{table}

Since $\vartheta<1$, the third-order Langevin midpoint exponent is strictly larger
than each of the other three exponents in
Table~\ref{tab:introdiscretecomparison}. The two underdamped bounds are not
uniformly ordered: the UBU exponent is larger than the frozen-force exponent
when $\vartheta<1/2$, equal to it when $\vartheta=1/2$, and smaller when
$\vartheta>1/2$. This comparison concerns proved sufficient conditions rather
than lower bounds or an intrinsic ordering of the integrators.

\paragraph{Organization of the paper.}
Section~\ref{sec:model} introduces the model and assumptions.
Section~\ref{sec:continuous} proves the continuous-time result, and
Section~\ref{sec:discrete} treats the numerical schemes and their convergence.
Section~\ref{sec:numerics} gives a numerical illustration, and
Section~\ref{sec:conclusion} concludes. Appendix~\ref{sec:ubuappendix} contains the UBU
comparison theorem and its proof.

\section{Model Setup and Assumptions}\label{sec:model}

For fixed $\lambda,\gamma>0$, we recall from \eqref{eqn:third:order} the \textit{third-order Langevin dynamics} 
$S_t=(X_t,V_t,Z_t)\in\R^{3d}$, that satisfies:
\begin{equation}\label{eq:model}
\begin{aligned}
 \dd X_t&=V_t\dd t,\\
 \dd V_t&=(-\nabla U(X_t)+\lambda Z_t)\dd t,\\
 \dd Z_t&=(-\lambda V_t-\gamma Z_t)\dd t
                 +\sqrt{2\gamma\eps_t}\,\dd B_t,
\end{aligned}
\end{equation}
where $\varepsilon_{t}$ is a temperature parameter
that depends on $t$ and $B_{t}$ is a standard $d$-dimensional Brownian motion.
Define the energy and invariant Gibbs law by
\begin{equation}\label{eq:gibbs}
 H(x,v,z):=U(x)+\frac12|v|^2+\frac12|z|^2,
 \qquad \pi_\eps(\dd x\,\dd v\,\dd z)
   :=\mathcal Z_\eps^{-1}e^{-H(x,v,z)/\eps}\dd x\,\dd v\,\dd z.
\end{equation}
Here
\begin{equation}\label{eq:gibbsnormalizer}
 \mathcal Z_\eps:=\int_{\R^{3d}}e^{-H(y)/\eps}\dd y
                 =Z_\eps(2\pi\eps)^d,\qquad
 Z_\eps:=\int_{\R^d}e^{-U(x)/\eps}\dd x.
\end{equation}
Under Assumption~\ref{ass:potential}, the quadratic lower bound
\eqref{eq:quadraticU} ensures that $\mathcal Z_\eps<\infty$.
For $\eps>0$, define the frozen-temperature generator
\begin{equation}\label{eq:generator}
 \mathcal L_\eps:=v\cdot\nabla_x+(-\nabla U+\lambda z)\cdot\nabla_v
       +(-\lambda v-\gamma z)\cdot\nabla_z+\gamma\eps\Delta_z.
\end{equation}
The dynamics \eqref{eq:model} is time-inhomogeneous, with generator
$\mathcal L_{\eps_t}$ at time $t$; when $\eps_t\equiv\eps$ is constant, it is generated by $\mathcal L_\eps$.
Let $b_{\mathcal T}:=(v,-\nabla U+\lambda z,-\lambda v)$ and
$\mathcal T:=b_{\mathcal T}\cdot\nabla$. By \eqref{eq:gibbs},
\begin{equation}\label{eq:transportinvariance}
 \mathcal T H=v\cdot\nabla U+(-\nabla U+\lambda z)\cdot v
                    -\lambda v\cdot z=0,
 \qquad \operatorname{div}b_{\mathcal T}=0.
\end{equation}
Writing $p_\eps=\mathcal Z_\eps^{-1}e^{-H/\eps}$ for the density in
\eqref{eq:gibbs}, we have $\nabla_zp_\eps=-zp_\eps/\eps$. Thus
\eqref{eq:generator} and \eqref{eq:transportinvariance} give
\begin{equation}\label{eq:gibbsstationary}
 \mathcal L_\eps^*p_\eps
 =-\operatorname{div}(b_{\mathcal T}p_\eps)
   +\gamma\operatorname{div}_z(zp_\eps+\eps\nabla_zp_\eps)=0,
\end{equation}
which proves invariance at every fixed $\eps>0$.

To recover the notation of Mou et al. \cite{MouEtAl2021}, write their potential as $F$ and their
smoothness parameter as $L_0$. Set $U=F/L_0$, $\lambda=\gamma_{\rm Mou}$,
$\gamma=\xi_{\rm Mou}$, and $\eps_t=T_t/L_0$. Equation \eqref{eq:model} then has
force $-\nabla F/L_0$ and noise $\sqrt{2\xi_{\rm Mou}T_t/L_0}$, with position
law proportional to $e^{-F/T_t}$ at frozen temperature. This makes the
normalization of both the potential and the cooling constant explicit.
Replacing $F$ by $F/T_t$ while keeping $L_0$ fixed would define a different
time-dependent force.
Likewise, Chak, Kantas, and Pavliotis
\cite[Theorem~2.7]{ChakKantasPavliotis2023}, applied to \eqref{eq:ckp}, assume a fixed
matrix $A$, whereas $A=\gamma\eps_t\Id_d$ varies with time under the present
cooling convention.

\begin{assumption}[Potential and cooling]\label{ass:potential}
The function $U\in C^4(\R^d)$ has minimum zero, bounded Hessian
$\|\nabla^2U\|_{\rm op}\leq K$, and satisfies
$x\cdot\nabla U(x)\geq r|x|^2-m$ for some $r>0,m\geq0$.
Its third and fourth derivatives have at most polynomial growth.
The temperature $\eps_t\in C^1([0,\infty);(0,\infty))$ is non-increasing in $t$. For
some $t_0>1$ and $E>0$,
\begin{equation}\label{eq:coolingassumption}
 \eps_t\log(t+t_0)\longrightarrow E,
 \qquad
 |\eps_t'|(1+\eps_t^{-3})
 \leq C\frac{\log(t+t_0)}{t+t_0}
 \quad\text{for all sufficiently large }t.
\end{equation}
\end{assumption}

The canonical schedule $\eps_t=E/\log(t+t_0)$ satisfies
\eqref{eq:coolingassumption}, since $\eps_t\log(t+t_0)=E$ and
\begin{equation}\label{eq:canonicalcoolingderivative}
 |\eps_t'|(1+\eps_t^{-3})
 =\frac{E}{(t+t_0)\log^2(t+t_0)}
   +\frac{\log(t+t_0)}{E^2(t+t_0)}
 \leq C_E\frac{\log(t+t_0)}{t+t_0}.
\end{equation}

Throughout, relative Fisher information is understood in the weak sense. If
$\nu\ll\pi_\eps$ and $h=\dd\nu/\dd\pi_\eps$, we set
\begin{equation}\label{eq:weakfisherdefinition}
 \cI(\nu\mid\pi_\eps):=4\int_{\R^{3d}}
 |\nabla\sqrt h|^2\dd\pi_\eps,
\end{equation}
where the gradient is a weak derivative. The value is $+\infty$ if
$\nu\not\ll\pi_\eps$, if $\sqrt h\notin H^1_{\rm loc}(\R^{3d})$, or if
the integral diverges. The score notation
$\int_{\R^{3d}}|\nabla\log h|^2\dd\nu$ denotes \eqref{eq:weakfisherdefinition};
on $\{h>0\}$ the score is $2\nabla\sqrt h/\sqrt h$.

\begin{assumption}[Continuous-time initial law]\label{ass:initial}
The initial law $\nu_0$ of \eqref{eq:model} is absolutely continuous, has a finite
second moment, and finite
relative Fisher information with respect to $\pi_{\eps_0}$:
\begin{equation}\label{eq:initialfisher}
 \cI(\nu_0\mid\pi_{\eps_0})
 :=\int_{\R^{3d}}
 \left|\nabla\log\frac{\dd\nu_0}{\dd\pi_{\eps_0}}\right|^2\dd\nu_0
 <\infty.
\end{equation}
\end{assumption}

\begin{assumption}[LSI barrier]\label{ass:lsi}
Let $\mu_\eps(\dd x)=Z_\eps^{-1}e^{-U(x)/\eps}\dd x$ be the position
Gibbs measure, with $Z_\eps$ defined in \eqref{eq:gibbsnormalizer}.
There is $D\geq0$ such that for every $\zeta>0$ there are
$C_\zeta<\infty$ and $\eps_\zeta>0$ for which
\begin{equation}\label{eq:lsi}
 \KL(\nu\mid\mu_\eps)\leq C_{\rm LS}^{\rm pos}(\eps)
       \cI(\nu\mid\mu_\eps),
 \qquad C_{\rm LS}^{\rm pos}(\eps)\leq C_\zeta e^{(D+\zeta)/\eps},
 \quad 0<\eps\leq\eps_\zeta.
\end{equation}
Here $\cI(\nu\mid\mu):=\int_{\R^d}|\nabla\log(\dd\nu/\dd\mu)|^2\dd\nu$, and we assume
$D<E$ where $E$ is the cooling
constant from Assumption~\ref{ass:potential}.
\end{assumption}

\begin{assumption}[Additional discrete regularity]\label{ass:discrete}
For the discrete-time results, we further assume that
$\|\nabla^3U\|_\infty<\infty$ and that the initial law for the numerical
scheme has a finite fourth moment. The initial law may be singular, including
a point mass corresponding to deterministic initialization.
\end{assumption}

In Assumption~\ref{ass:lsi}, $D$ denotes the LSI barrier. By
\eqref{eq:gibbs} and \eqref{eq:gibbsnormalizer},
\begin{equation}\label{eq:gibbsproduct}
 \pi_\eps=\mu_\eps\otimes\Gamma_\eps,\qquad
 \Gamma_\eps:=\mathcal N(0,\eps\Id_{2d}).
\end{equation}
The Gaussian LSI has constant $\eps/2$ in our Fisher convention.
For a normalized density $h$ with respect to $\Gamma_\eps$, the
Ornstein--Uhlenbeck semigroup
$Q_rh(y):=\E[h(e^{-r}y+\sqrt{\eps(1-e^{-2r})}\,G)]$, with
$G\sim\mathcal N(0,\Id_{2d})$, satisfies
$\nabla Q_rh=e^{-r}Q_r\nabla h$. The Cauchy--Schwarz inequality and invariance give
\begin{equation}\label{eq:gaussianlsiproof}
\begin{aligned}
 \cI((Q_rh)\Gamma_\eps\mid\Gamma_\eps)
 &\leq e^{-2r}\int_{\R^{2d}}Q_r\left(\frac{|\nabla h|^2}{h}\right)
                                      \dd\Gamma_\eps
 =e^{-2r}\cI(h\Gamma_\eps\mid\Gamma_\eps),\\
 \KL(h\Gamma_\eps\mid\Gamma_\eps)
 &=\eps\int_0^\infty\cI((Q_rh)\Gamma_\eps\mid\Gamma_\eps)\dd r
 \leq\frac\eps2\cI(h\Gamma_\eps\mid\Gamma_\eps).
\end{aligned}
\end{equation}
The entropy identity in \eqref{eq:gaussianlsiproof} follows by integration
by parts for the generator $\eps\Delta-y\cdot\nabla$ and convergence
$Q_rh\to1$. Smooth normalized densities bounded above and away from zero
give the identity directly; truncation, normalization, and lower
semicontinuity extend the inequality to the weak Fisher convention.
For $\nu=h\pi_\eps$, let
$\bar h(x):=\int_{\R^{2d}}h(x,v,z)\Gamma_\eps(\dd v\,\dd z)$.
The entropy chain rule for \eqref{eq:gibbsproduct}, followed by
\eqref{eq:lsi}, \eqref{eq:gaussianlsiproof}, and the conditional
Cauchy--Schwarz inequality for $\nabla_x\bar h$, yields
\begin{equation}\label{eq:productlsiproof}
\begin{aligned}
 \KL(\nu\mid\pi_\eps)
 &=\KL(\bar h\mu_\eps\mid\mu_\eps)
   +\int_{\R^d}\bar h(x)
       \KL\left(\frac{h(x,\cdot)}{\bar h(x)}\Gamma_\eps
                         \mid\Gamma_\eps\right)\mu_\eps(\dd x)\\
 &\leq C_{\rm LS}^{\rm pos}(\eps)
          \int_{\R^{3d}}|\nabla_x\log h|^2\dd\nu
     +\frac\eps2\int_{\R^{3d}}
          \left(|\nabla_v\log h|^2+|\nabla_z\log h|^2\right)\dd\nu.
\end{aligned}
\end{equation}
Thus, for each $\zeta>0$ and $0<\eps\leq\eps_\zeta$, an admissible
full-state LSI constant is
\begin{equation}\label{eq:tensorizedlsi}
 C_{\rm LS}(\eps)=\max\{C_{\rm LS}^{\rm pos}(\eps),\eps/2\},
\end{equation}
which has the same exponential bound as in \eqref{eq:lsi}. Under the landscape
hypotheses of Menz and Schlichting \cite{MenzSchlichting2014}, $D$ can be identified
with the corresponding critical depth. The bounded-Hessian and dissipativity conditions
support the dynamical estimates, while the landscape hypotheses determine the
low-temperature Gibbs mixing across multiple wells.

\begin{remark}[Comparison with related assumptions in the literature]\label{rem:assumptioncomparison}
He, Tan, and Wu \cite[Assumptions~2.1 and 2.3]{HeTanWu2024} assume that
$U\in C^\infty$, all derivatives of $U$ have at most polynomial growth,
$\nabla U$ is globally Lipschitz, and $U$ is dissipative. They also require
finitely many local minima with nondegenerate Hessians. These potential
conditions imply the potential regularity in Assumption~\ref{ass:potential};
their LSI estimate \cite[Proposition~3.1]{HeTanWu2024} gives
Assumption~\ref{ass:lsi} with $D$ equal to their critical depth $E_*$. Our
direct LSI formulation replaces finiteness and nondegeneracy conditions on the
local minima by the functional inequality itself. Their initial-law condition
includes a smooth density, finite raw Fisher information, and moments of every
order. The analogous condition on our three-component state implies
Assumption~\ref{ass:initial}. Indeed, for $\nu=\rho\,\dd y$, define
$\cI^{\rm raw}(\nu):=4\int_{\R^{3d}}|\nabla\sqrt\rho|^2\dd y$.
By \eqref{eq:gibbs} and Assumption~\ref{ass:potential},
\begin{equation}\label{eq:rawrelativefisherbound}
 \cI(\nu\mid\pi_\eps)
 \leq2\cI^{\rm raw}(\nu)
       +\frac{2}{\eps^2}\int_{\R^{3d}}|\nabla H|^2\dd\nu
 \leq2\cI^{\rm raw}(\nu)
       +C\eps^{-2}\int_{\R^{3d}}(1+|y|^2)\dd\nu.
\end{equation}
Assumption~\ref{ass:initial} thus allows less regular initial laws. Their
eventually exact schedule $\eps_t=E/\log t$ satisfies the asymptotic and
derivative bounds in \eqref{eq:coolingassumption}; we additionally require
global $C^1$ regularity of the schedule. For their discrete-time theorem,
He, Tan, and Wu \cite[Theorem~2.7]{HeTanWu2024} additionally assume that
$\nabla^2U$ is globally Lipschitz, which, under their smoothness assumption and
up to the choice of tensor norm, corresponds to
$\|\nabla^3U\|_\infty<\infty$ in Assumption~\ref{ass:discrete}. Their
requirement of moments of every order implies our finite-fourth-moment
condition. However, their smooth-density assumption excludes deterministic
initialization, whereas Assumption~\ref{ass:discrete} allows singular initial
laws.

Chak, Kantas, and Pavliotis
\cite[Assumptions~1--4]{ChakKantasPavliotis2023} assume a smooth potential with
bounded Hessian, dissipativity, and a quadratic gradient bound, together with
either a two-sided quadratic comparison or a nondegenerate Morse landscape with
additional saddle structure. These alternatives yield an LSI barrier
$\widehat{E}$, which corresponds to $D=\widehat{E}$ in
Assumption~\ref{ass:lsi}. The two formulations emphasize different structures:
their landscape assumptions specify Morse and saddle geometry, while ours uses
$C^4$ regularity together with a direct LSI. Their initial density is bounded and
smooth, whereas Assumption~\ref{ass:initial} accommodates nonsmooth and unbounded
initial densities. The cooling conditions likewise emphasize different controls.
They allow any $T_t\to0$ satisfying $T_t\geq E/\log t$ and
$|T_t'|\leq C/t$, while ours combines monotonicity, the asymptotic relation
$\eps_t\log(t+t_0)\to E$, and the sharper derivative control in
\eqref{eq:coolingassumption}. Their dynamics cools through
temperature-dependent friction with fixed noise, whereas \eqref{eq:model} keeps
the friction fixed and decreases the noise amplitude.
\end{remark}

\section{Continuous-Time Annealing}\label{sec:continuous}

\subsection{Dissipation and initial-density approximation}
For $h:=\dd\nu/\dd\pi_\eps$, let $w:=\nabla\log h$, and write
\begin{equation}\label{eq:functional}
 \cH_\eps(\nu):=\cI_M(\nu\mid\pi_\eps)
                  +\frac{A_0}{\eps}\KL(\nu\mid\pi_\eps),
 \qquad \cI_M:=\int_{\R^{3d}} w^\top M w\,\dd\nu,
\end{equation}
Consistently with \eqref{eq:weakfisherdefinition}, the weighted Fisher term is
$\cI_M(\nu\mid\pi_\eps)=4\int_{\R^{3d}}
(\nabla\sqrt h)^\top M\nabla\sqrt h\,\dd\pi_\eps$, with the same
weak-derivative and extended-value conventions.
Here $\KL(\nu\mid\pi_\eps)$ denotes the Kullback-Leibler (KL) divergence:
\begin{equation}\label{eq:relativeentropydefinition}
\KL(\nu\mid\pi_\eps)
:=\int_{\mathbb{R}^{3d}}\log h\,\dd\nu
=\int_{\mathbb{R}^{3d}}h\log h\,\dd\pi_\eps,
\end{equation}
and
\begin{align}\label{eq:matrix}
A_0:=\frac{2\lambda b+2(q_1^2+q_2^2)+1}{\gamma},
\qquad
M:=\begin{pmatrix}1&1&0\\1&2&b\\0&b&\varsigma\end{pmatrix}\otimes\Id_d,
\end{align}
with
\begin{align}
 b:=\frac{K+1+(K+2)^2}{\lambda},
 \quad \varsigma:=b^2+1,
 \quad q_1:=b-\lambda,
 \quad q_2:=\lambda\varsigma-2\lambda+\gamma b.
\end{align}

The coefficients in \eqref{eq:functional}--\eqref{eq:matrix} are chosen for the
frozen-temperature identities \eqref{eq:fisherentropyderivatives}. In particular,
\eqref{eq:entropyderivative} shows that KL dissipation acts only on $w_z$ and
carries a factor $\gamma\eps$. The weight $A_0/\eps$ removes this temperature
factor, while the off-diagonal entries of $M$ transfer dissipation through
$Z\to V\to X$. The choice $\varsigma=b^2+1$ ensures $M\succ0$, and $b$ leaves the
positive margin $2\lambda b-2K-(K+2)^2=2+(K+2)^2$ in the $(x,v)$ block for
every $\|Q\|_{\rm op}\leq K$. The quantities $q_1,q_2$ are the resulting
$(x,z)$ and $(v,z)$ mixed coefficients in the matrix of
\eqref{eq:certificate}; $A_0$ is chosen so that
$\gamma A_0-2\lambda b-2(q_1^2+q_2^2)=1$ after these terms are absorbed.
These are explicit sufficient choices for \eqref{eq:certificate}, rather than
optimized constants.

The next lemma provides the coercivity used in the frozen-temperature
dissipation identity \eqref{eq:frozenidentity} in the proof of
Theorem~\ref{thm:main}; it holds uniformly over all Hessians with
$\|\nabla^2U\|_{\rm op}\leq K$.

\begin{lemma}[Algebraic dissipation certificate]\label{lem:matrix}
The matrix $M$ in \eqref{eq:matrix} is positive definite. For every symmetric $Q$ with
$\|Q\|_{\rm op}\leq K$, define
\[
 J_Q:=\begin{pmatrix}0&-\Id_{d}&0\\Q&0&-\lambda\Id_{d}\\0&\lambda\Id_{d}&-\gamma\Id_{d}\end{pmatrix},
 \qquad P_z:=\operatorname{diag}(0,0,\Id_d).
\]
Then
\begin{equation}\label{eq:certificate}
 -J_QM-MJ_Q^\top+A_0\gamma P_z\succeq\frac12\Id_{3d}.
\end{equation}
\end{lemma}
\begin{proof}
The scalar matrix defining $M$ has leading principal minors
$1$, $1$, and $\varsigma-b^2=1$, so that Sylvester's criterion gives $M\succ0$.
Denote the matrix on the left of \eqref{eq:certificate} by $B_Q$. Direct
matrix multiplication gives
\[
 B_Q=\begin{pmatrix}
 2\Id_{d}&2\Id_{d}-Q&q_1\Id_{d}\\
 2\Id_{d}-Q&2\lambda b\Id_{d}-2Q&q_2\Id_{d}\\
 q_1\Id_{d}&q_2\Id_d&(2\gamma\varsigma-2\lambda b+A_0\gamma)\Id_{d}
 \end{pmatrix}.
\]
For $u,v\in\R^d$, the quadratic form of its top-left block satisfies
\begin{align*}
 2|u|^2+2u^\top(2\Id_{d}-Q)v+v^\top(2\lambda b\Id_d-2Q)v
 &\geq 2|u|^2-2(K+2)|u||v|+(2\lambda b-2K)|v|^2\\
 &\geq |u|^2+\left(2\lambda b-2K-(K+2)^2\right)|v|^2\\
 &=|u|^2+\left(2+(K+2)^2\right)|v|^2.
\end{align*}
Moreover, for $z\in\R^d$,
\[
 2(q_1u+q_2v)\cdot z
 \geq-\frac12\left(|u|^2+|v|^2\right)
      -2\left(q_1^2+q_2^2\right)|z|^2.
\]
Since
\[
 2\gamma\varsigma-2\lambda b+A_0\gamma-2\left(q_1^2+q_2^2\right)=2\gamma\varsigma+1,
\]
we obtain, for every $\xi=(u,v,z)\in\R^{3d}$,
\[
 \xi^\top B_Q\xi
 \geq\frac12\left(|u|^2+|v|^2\right)+(2\gamma\varsigma+1)|z|^2
 \geq\frac12|\xi|^2.
\]
This proves \eqref{eq:certificate}.
\end{proof}

For nonsmooth initial densities, the next lemma justifies the integrated
entropy identities by approximation.

\begin{lemma}[Justification by initial-density approximation]\label{lem:closure}
Under Assumption~\ref{ass:potential}, let the annealed third-order process
$S_t=(X_t,V_t,Z_t)$ in \eqref{eq:model} be started, or restarted
at a deterministic time $s\geq0$, from
an absolutely continuous law $\nu_s$ with finite relative Fisher information with
respect to $\pi_{\eps_s}$ and finite second moment. Let $\nu_t$ denote the
law of $S_t$ for $t\geq s$, and let $F(t):=\cH_{\eps_t}(\nu_t)$. Then
$F(t)<\infty$ for every finite $t\geq s$.
Moreover, there exist smooth, strictly positive probability laws
$\nu_s^{(n)}$, Gaussian outside compact sets, such that
\[
 \mathcal W_2\left(\nu_s^{(n)},\nu_s\right)\to0,\qquad
 \cH_{\eps_s}\left(\nu_s^{(n)}\right)\to F(s),\qquad
 \sup_{n\geq1}\int_{\R^{3d}}|y|^2\nu_s^{(n)}(\dd y)<\infty.
\]
Let $\nu_u^{(n)}$ be the laws of \eqref{eq:model} started from $\nu_s^{(n)}$
at time $s$, and set $F_n(u):=\cH_{\eps_u}\left(\nu_u^{(n)}\right)$.
Fix $T>s$.
Suppose that there exist $a,b\in L^1([s,T])$, with $b\geq0$, such that
\begin{equation}\label{eq:abstractcomparison}
 F_n'(u)+a(u)F_n(u)\leq b(u)
 \quad\text{for a.e. }u\in[s,T]\text{ and every }n\geq1.
\end{equation}
Then $F$ satisfies
\begin{equation}\label{eq:abstractintegrated}
 F(t)\leq e^{-\int_s^t a(u)\dd u}F(s)
 +\int_s^t e^{-\int_v^t a(u)\dd u}b(v)\dd v,
 \qquad s\leq t\leq T.
\end{equation}
\end{lemma}

\begin{remark}[Use of the approximation lemma]
Lemma~\ref{lem:closure} is used in the proof of
Theorem~\ref{thm:main} to pass the comparison inequality
\eqref{eq:ode} from regularized initial laws to laws satisfying
Assumption~\ref{ass:initial}. Its time-shifted form is used in the proof
of Proposition~\ref{prop:restart} to extend the regular-law estimate
\eqref{eq:restartode} to general restart laws covered by that proposition.
\end{remark}

\begin{proof}[Proof of Lemma~\ref{lem:closure}]
Let $T>s$ be as in the statement. By the deterministic time shift $r=t-s$, with
$\widetilde\eps_r=\eps_{s+r}$,
$\widetilde B_r=B_{s+r}-B_s$,
$\widetilde a(r)=a(s+r)$, and $\widetilde b(r)=b(s+r)$ when the comparison
hypothesis is used, we may take $s=0$ and relabel $T-s$ as $T$ throughout
the proof. For each $y\in\R^{3d}$, let $\Phi_t(y)$ denote the solution of
\eqref{eq:model} at time $t$ started from $y$ at the shifted time $0$, using
the same Brownian path for every $y$. By \eqref{eq:model}, the drift is
$\mathfrak b(x,v,z):=(v,-\nabla U(x)+\lambda z,-\lambda v-\gamma z)$.
Set
\[
 N_T:=\sup_{0\leq t\leq T}\left|\int_0^t\sqrt{2\gamma\eps_r}\,\dd B_r\right|.
\]
Write $Y_t=\Phi_t(y)$. Subtracting the additive noise path gives a
pathwise nonautonomous ODE.
Indeed, with
\[
 G_t:=\left(0,0,\int_0^t\sqrt{2\gamma\eps_u}\,\dd B_u\right),
 \qquad \overline Y_t:=Y_t-G_t,
\]
the integral equation for \eqref{eq:model} becomes
\begin{equation}\label{eq:pathwiseforwardintegral}
 Y_t=y+\int_0^t\mathfrak b(Y_u)\dd u+G_t,
\end{equation}
so that
\begin{equation}\label{eq:pathwiseode}
 \dot{\overline Y}_t=\mathfrak b\left(\overline Y_t+G_t\right),
 \qquad \overline Y_0=y.
\end{equation}
Fix a sample path on which $G$ is continuous. Then $\overline Y$ solves the
nonautonomous ODE \eqref{eq:pathwiseode}, whose vector field
$q\mapsto\mathfrak b(q+G_t)$ depends continuously on $t$.
Boundedness of $D\mathfrak b$
gives a deterministic constant $C>0$ such that, for all $q,q'\in\R^{3d}$,
\[
 |\mathfrak b(q)-\mathfrak b(q')|\leq C|q-q'|,
 \qquad |\mathfrak b(q)|\leq C(1+|q|).
\]
In particular, \eqref{eq:pathwiseode} has a unique global solution. Set
$m(r):=\sup_{0\leq u\leq r}|Y_u|$. Taking the supremum in the forward
integral equation \eqref{eq:pathwiseforwardintegral} and using
$\sup_{u\leq T}|G_u|=N_T$, we obtain, for
$0\leq r\leq T$,
\[
 m(r)\leq |y|+N_T+Cr+C\int_0^r m(u)\dd u
       \leq |y|+N_T+CT+C\int_0^r m(u)\dd u.
\]
Gr\"{o}nwall's inequality therefore gives
\begin{equation}\label{eq:forwardflowgrowth}
 \sup_{0\leq r\leq T}|\Phi_r(y)|
 =m(T)\leq (|y|+N_T+CT)e^{CT}
 \leq C_T(1+|y|+N_T).
\end{equation}

To construct and estimate the inverse flow, fix $t\in[0,T]$ and a terminal
point $y$. The backward integral equation is
\begin{equation}\label{eq:pathwisebackwardintegral}
 Q_r=y-\int_r^t\mathfrak b(Q_u)\dd u+G_r-G_t,
 \qquad 0\leq r\leq t.
\end{equation}
With $\widehat Q_\tau:=Q_{t-\tau}$ and
$\widehat G_\tau:=G_{t-\tau}-G_t$, this becomes
\begin{equation}\label{eq:pathwisereversedintegral}
 \widehat Q_\tau
 =y-\int_0^\tau\mathfrak b\left(\widehat Q_u\right)\dd u+\widehat G_\tau,
 \qquad 0\leq\tau\leq t.
\end{equation}
Here $\widehat G$ is continuous, $\widehat G_0=0$, and
$\sup_{\tau\leq t}|\widehat G_\tau|\leq2N_T$. Define
$\widehat R_\tau:=\widehat Q_\tau-\widehat G_\tau$. Subtracting
$\widehat G_\tau$ from both sides of \eqref{eq:pathwisereversedintegral}
gives 
\[
\widehat R_\tau=y-\int_0^\tau
\mathfrak b\left(\widehat R_u+\widehat G_u\right)\dd u.
\]
This integral is continuously
differentiable in $\tau$, so $\widehat R$ satisfies
\begin{equation}\label{eq:reversepathwiseode}
 \dot{\widehat R}_\tau
 =-\mathfrak b\left(\widehat R_\tau+\widehat G_\tau\right),
 \qquad \widehat R_0=y.
\end{equation}
The vector field in \eqref{eq:reversepathwiseode} is continuous in time and
globally Lipschitz in $\widehat R$. Thus \eqref{eq:reversepathwiseode}, and
hence \eqref{eq:pathwisebackwardintegral}, has a unique solution on $[0,t]$.
Evaluating \eqref{eq:pathwisebackwardintegral} at $r=0$ gives
$Q_0=y-\int_0^t\mathfrak b(Q_u)\dd u-G_t$, since $G_0=0$. For a general
$r\in[0,t]$, the difference between \eqref{eq:pathwisebackwardintegral}
and this identity is
\begin{align*}
 Q_r-Q_0
 &=-\int_r^t\mathfrak b(Q_u)\dd u
   +\int_0^t\mathfrak b(Q_u)\dd u+(G_r-G_t)-(-G_t)\\
 &=\int_0^r\mathfrak b(Q_u)\dd u+G_r.
\end{align*}
Consequently, $Q$ satisfies \eqref{eq:pathwiseforwardintegral} with initial
point $Q_0$. By uniqueness of the solution of
\eqref{eq:pathwiseforwardintegral}, $Q_r=\Phi_r(Q_0)$, and hence
$\Phi_t(Q_0)=y$.
Conversely, any forward trajectory ending at $y$ satisfies the backward
equation \eqref{eq:pathwisebackwardintegral}, whose uniqueness makes its
initial point equal to $Q_0$.
Thus $\Phi_t$ is bijective and
$\Phi_t^{-1}(y)=Q_0=\widehat Q_t$.

Set $\widehat m(\tau):=\sup_{0\leq u\leq\tau}|\widehat Q_u|$.
The reversed integral equation \eqref{eq:pathwisereversedintegral} gives
\[
 \widehat m(\tau)\leq |y|+2N_T+CT
                     +C\int_0^\tau\widehat m(u)\dd u.
\]
Another application of Gr\"{o}nwall's inequality yields
\begin{equation}\label{eq:inverseflowgrowth}
 |\Phi_t^{-1}(y)|\leq\widehat m(t)
 \leq (|y|+2N_T+CT)e^{Ct}
 \leq C_T(1+|y|+N_T).
\end{equation}
The constant $C_T$ in \eqref{eq:inverseflowgrowth} depends only on $T$ and
the drift growth constant $C$ and is uniform over $t\in[0,T]$. Combining
\eqref{eq:forwardflowgrowth} and \eqref{eq:inverseflowgrowth}, and
increasing $C_T$ if necessary, proves
\begin{equation}\label{eq:pathwise-flow-growth}
 \sup_{0\leq t\leq T}|\Phi_t(y)|
 +\sup_{0\leq t\leq T}\left|\Phi_t^{-1}(y)\right|
 \leq C_T(1+|y|+N_T).
\end{equation}
Moreover, $U\in C^4$ implies $\mathfrak b\in C^3$, so differentiable
dependence on initial data for the nonautonomous ODE
\eqref{eq:pathwiseode} gives
$\Phi_t\in C^3$ with respect to $y$. Since the fixed noise path $G_t$ is
independent of $y$, differentiating \eqref{eq:pathwiseforwardintegral} with respect to
the initial point gives the following variational equations for
$J_t=D_y\Phi_t(y)$,
$\mathcal K_t=D_y^2\Phi_t(y)$,
and $L_t=D_y^3\Phi_t(y)$:
\begin{equation}\label{eq:flowvariations}
\begin{aligned}
 \dot J_t&=D\mathfrak b(Y_t)J_t,\qquad J_0=I,\\
 \dot{\mathcal K}_t[\xi_1,\xi_2]
 &=D\mathfrak b(Y_t)\mathcal K_t[\xi_1,\xi_2]
   +D^2\mathfrak b(Y_t)[J_t\xi_1,J_t\xi_2],\\
 \dot L_t[\xi_1,\xi_2,\xi_3]
 &=D\mathfrak b(Y_t)L_t[\xi_1,\xi_2,\xi_3]
   +D^3\mathfrak b(Y_t)[J_t\xi_1,J_t\xi_2,J_t\xi_3]+\sum_{\rm cyc}D^2\mathfrak b(Y_t)
       [J_t\xi_1,\mathcal K_t[\xi_2,\xi_3]],
\end{aligned}
\end{equation}
where $\sum_{\rm cyc}$ denotes the sum over the three cyclic permutations of
$(\xi_1,\xi_2,\xi_3)$; explicitly, 
\[
\begin{aligned}
 \sum_{\rm cyc}D^2\mathfrak b(Y_t)
       [J_t\xi_1,\mathcal K_t[\xi_2,\xi_3]]&=D^2\mathfrak b(Y_t)[J_t\xi_1,\mathcal K_t[\xi_2,\xi_3]]
 +D^2\mathfrak b(Y_t)[J_t\xi_2,\mathcal K_t[\xi_3,\xi_1]]\\
 &\qquad
 +D^2\mathfrak b(Y_t)[J_t\xi_3,\mathcal K_t[\xi_1,\xi_2]].
\end{aligned}
\]
Since $\Phi_0(y)=y$, we have $J_0=I$ and $\mathcal K_0=L_0=0$.
Assumption~\ref{ass:potential} and \eqref{eq:flowvariations} give $\|J_t\|\leq e^{Ct}$ and, for finite exponents $p_2,p_3$,
\begin{equation}\label{eq:flowvariationintegrals}
\begin{aligned}
 \|\mathcal K_t\|&\leq C\int_0^t\|\mathcal K_u\|\dd u
 +C_T\int_0^t(1+|Y_u|^{p_2})\|J_u\|^2\dd u,\\
 \|L_t\|&\leq C\int_0^t\|L_u\|\dd u
 +C_T\int_0^t(1+|Y_u|^{p_3})
 \left(\|J_u\|^3+\|J_u\|\|\mathcal K_u\|\right)\dd u.
\end{aligned}
\end{equation}
Applying \eqref{eq:pathwise-flow-growth} to the first inequality in
\eqref{eq:flowvariationintegrals} bounds $\mathcal K_t$; inserting that
bound into the second inequality and applying
Gr\"{o}nwall's inequality then gives, for some finite $p_2',p_3'$,
\begin{equation}\label{eq:forwardflowderivativebounds}
\begin{aligned}
 \sup_{0\leq t\leq T}\|\mathcal K_t\|
 &\leq C_T(1+|y|+N_T)^{p_2'},\\
 \sup_{0\leq t\leq T}\|L_t\|
 &\leq C_T(1+|y|+N_T)^{p_3'}.
\end{aligned}
\end{equation}
Differentiating $\Phi_t^{-1}\circ\Phi_t=\Id$ gives the following formulas
for the derivatives of the inverse flow.
Indeed, if $\Psi_t=\Phi_t^{-1}$, then, with all derivatives of $\Phi_t$
evaluated at $\Psi_t(y)$,
\begin{equation}\label{eq:inverseflowderivatives}
\begin{aligned}
 D\Psi_t(y)&=\left[D\Phi_t(\Psi_t(y))\right]^{-1},\\
 D^2\Psi_t(y)[u,v]
 &=-D\Psi_t(y)D^2\Phi_t(\Psi_t(y))
       [D\Psi_t(y)u,D\Psi_t(y)v],\\
 D^3\Psi_t
 &=-D\Psi_t\left(
 D^3\Phi_t\left[D\Psi_t,D\Psi_t,D\Psi_t\right]
 +\sum_{\rm cyc}D^2\Phi_t\left[D^2\Psi_t,D\Psi_t\right]\right).
\end{aligned}
\end{equation}
For each fixed initial point, the inverse Jacobian satisfies
$\frac{\dd}{\dd t}J_t^{-1}=-J_t^{-1}D\mathfrak b(Y_t)$ and
$J_0^{-1}=I$. Boundedness of $D\mathfrak b$ and Gr\"{o}nwall's inequality
give $\|J_t^{-1}\|\leq e^{Ct}$, uniformly in the initial point.
Substituting this estimate, \eqref{eq:forwardflowderivativebounds}, and
the bound on $|\Psi_t(y)|$ from \eqref{eq:pathwise-flow-growth} into
\eqref{eq:inverseflowderivatives} gives
\begin{equation}\label{eq:inverseflowbounds}
\begin{aligned}
 \sup_{0\leq t\leq T}\|D\Psi_t(y)\|&\leq e^{CT},\\
 \sup_{0\leq t\leq T}\|D^2\Psi_t(y)\|
 &\leq C_T(1+|y|+N_T)^{p_2'},\\
 \sup_{0\leq t\leq T}\|D^3\Psi_t(y)\|
 &\leq C_T(1+|y|+N_T)^{\widetilde p_3},
 \qquad \widetilde p_3:=\max\{p_3',2p_2'\}.
\end{aligned}
\end{equation}
The exponent $2p_2'$ accounts for the product of $D^2\Phi_t$ and
$D^2\Psi_t$ in the third line of \eqref{eq:inverseflowderivatives}.
Finally, $\operatorname{div}\mathfrak b=-\gamma d$, so the pathwise
Jacobi--Liouville formula \cite{Kunita1990}
and
the change of variables under the random flow give
\begin{equation}\label{eq:flowdensity}
 \det D\Phi_t=e^{-\gamma dt},\qquad
 \rho_t(y)=e^{\gamma dt}\E\left[\rho_0\left(\Phi_t^{-1}(y)\right)\right].
\end{equation}
In particular, for every multi-index
$\alpha$ with $2\leq|\alpha|\leq3$,
and for any $0\leq t\leq T$,
\begin{equation}\label{eq:flowbounds}
\begin{aligned}
 &\left|\Phi_t^{\pm1}(y)\right|\leq C_T(1+|y|+N_T),
 \\
 &\left\|D\Phi_t^{\pm1}(y)\right\|\leq C_T,
 \\
 &\left|\partial^\alpha\Phi_t^{\pm1}(y)\right|
 \leq C_{T,\alpha}(1+|y|+N_T)^{q_\alpha},
\end{aligned}
\end{equation}
for some $q_\alpha<\infty$.

\smallskip
\noindent\emph{Step 1: regular initial densities.}
Assume first that $\rho_0$ is positive, smooth, and Gaussian outside a compact
set. Constants in this step may depend on this fixed $\rho_0$; uniformity
over the approximating sequence will be proved in Step~3. Define
\[
 a_t(y,\omega):=e^{\gamma dt}\rho_0\left(\Phi_t^{-1}(y,\omega)\right),
 \qquad A_{\alpha,t}:=\frac{\partial^\alpha a_t}{a_t}.
\]
The first three derivatives of $\rho_0$ are bounded because $\rho_0$ is
smooth and Gaussian outside a compact set. For every compact set
$K\subset\R^{3d}$, the chain rule and \eqref{eq:flowbounds} therefore give,
for $|\alpha|\leq3$ and some finite $q_\alpha$,
\begin{equation}\label{eq:compactdensityderivatives}
 \sup_{0\leq t\leq T}\sup_{y\in K}|\partial^\alpha a_t(y,\omega)|
 \leq C_{T,K,\alpha}(1+N_T)^{q_\alpha}.
\end{equation}
Doob's inequality and the Gaussian moments of $G_T$ give
\begin{equation}\label{eq:noisemaximalmoments}
 \E N_T^p\leq C_{p,d}
       \left(2\gamma\int_0^T\eps_u\dd u\right)^{p/2}<\infty,
 \qquad p\geq2.
\end{equation}
H\"older's inequality gives finite moments for $0<p<2$ as well.
The derivatives of $a_t$ are pathwise continuous in $(t,y)$.
Equations \eqref{eq:compactdensityderivatives} and \eqref{eq:noisemaximalmoments}
therefore justify differentiating
\eqref{eq:flowdensity} under the expectation and yields
\begin{equation}\label{eq:densitylocalregularity}
 \rho\in C\left([0,T];C^3_{\rm loc}(\R^{3d})\right).
\end{equation}
Positivity on compact sets and the Gaussian tails of $\rho_0$ give
\begin{equation}\label{eq:initialdensityratios}
 \frac{|\partial^\beta\rho_0(x)|}{\rho_0(x)}
 \leq C_\beta(1+|x|^{|\beta|}),\qquad 1\leq|\beta|\leq3.
\end{equation}
For example, with $\Psi_t=\Phi_t^{-1}$,
\begin{equation}\label{eq:conditionaldensitychain}
\begin{aligned}
 \frac{Da_t}{a_t}
 &=\frac{D\rho_0(\Psi_t)}{\rho_0(\Psi_t)}D\Psi_t,\\
 \frac{D^2a_t}{a_t}
 &=\frac{D^2\rho_0(\Psi_t)}{\rho_0(\Psi_t)}
                  [D\Psi_t,D\Psi_t]
   +\frac{D\rho_0(\Psi_t)}{\rho_0(\Psi_t)}D^2\Psi_t.
\end{aligned}
\end{equation}
The third derivative has the corresponding terms with $D^3\rho_0$,
$D^2\rho_0D^2\Psi_t$, and $D\rho_0D^3\Psi_t$.
Using \eqref{eq:initialdensityratios},
\eqref{eq:conditionaldensitychain}, and \eqref{eq:flowbounds} at
$\Psi_t(S_t)=S_0$ gives, for suitable $m_\alpha<\infty$,
\begin{equation}\label{eq:Aalphabound}
 |A_{\alpha,t}(S_t,\omega)|
 \leq C_{T,\alpha}(1+|S_0|+N_T)^{m_\alpha},\qquad 0\leq t\leq T,
\end{equation}
and the right-hand side of \eqref{eq:Aalphabound}
has moments of every order because the regular
initial density has Gaussian tails. Differentiating
\eqref{eq:flowdensity} under the expectation and conditioning on $S_t=y$ give
\begin{equation}\label{eq:conditionaldensityratios}
 \frac{\partial^\alpha\rho_t}{\rho_t}(y)
   =\E[A_{\alpha,t}(S_t,\omega)\mid S_t=y].
\end{equation}
To pass to logarithmic derivatives, let
$r_i=(\partial_i\rho_t)/\rho_t$,
$r_{ij}=(\partial_{ij}\rho_t)/\rho_t$, and
$r_{ijk}=(\partial_{ijk}\rho_t)/\rho_t$. Direct differentiation gives
\begin{equation}\label{eq:logderivativeidentities}
\begin{aligned}
 \partial_i\log\rho_t&=r_i,\\
 \partial_{ij}\log\rho_t&=r_{ij}-r_ir_j,\\
 \partial_{ijk}\log\rho_t
 &=r_{ijk}-r_{ij}r_k-r_{ik}r_j-r_{jk}r_i+2r_ir_jr_k.
\end{aligned}
\end{equation}
Applying conditional Jensen's inequality to \eqref{eq:conditionaldensityratios}
gives, for $m\geq0$ and $p\geq1$,
\begin{equation}\label{eq:conditionalratiomoments}
 \int_{\R^{3d}}(1+|y|^m)
       \left|\frac{\partial^\alpha\rho_t}{\rho_t}(y)\right|^p\rho_t(y)\dd y
 \leq\E\left[(1+|S_t|^m)|A_{\alpha,t}(S_t,\omega)|^p\right].
\end{equation}
The right-hand side is uniformly finite by \eqref{eq:Aalphabound} and
\eqref{eq:pathwise-flow-growth}. Applying H\"older's inequality to the
products in \eqref{eq:logderivativeidentities}, using
\eqref{eq:conditionalratiomoments}, proves
\begin{equation}\label{eq:logdensitybounds}
 \sup_{0\leq t\leq T}\int_{\R^{3d}}(1+|y|^m)
 |\partial^\alpha\log\rho_t(y)|^p\rho_t(y)\dd y<\infty,
\end{equation}
for every $m\geq0$, $p\geq1$, and $1\leq|\alpha|\leq3$. Moreover,
\begin{equation}\label{eq:densitybounds}
 c_Te^{-C_T(1+|y|^2)}\leq\rho_t(y)
 \leq e^{\gamma dT}\|\rho_0\|_\infty,
 \qquad 0\leq t\leq T.
\end{equation}
Indeed, positivity and the Gaussian tails of the regular initial density give
$\rho_0(x)\geq c e^{-C(1+|x|^2)}$. Choose $R_T<\infty$ with
$p_T:=\mathbb P(N_T\leq R_T)>0$. On this event, \eqref{eq:flowbounds} yields
\[
 \rho_0\left(\Phi_t^{-1}(y)\right)
 \geq c_Te^{-C_T(1+|y|^2+R_T^2)}.
\]
Multiplication by $e^{\gamma dt}p_T$ and absorption of $R_T$ into the constants
prove the lower bound in \eqref{eq:densitybounds}.

For every $\varphi\in C_c^\infty(\R^{3d})$ and $0\leq r\leq t\leq T$,
It\^o's formula gives
\begin{equation}\label{eq:weakfokkerplanck}
 \int_{\R^{3d}}\varphi\rho_t\dd y
 -\int_{\R^{3d}}\varphi\rho_r\dd y
 =\int_r^t\int_{\R^{3d}}\mathcal L_{\eps_u}\varphi\,\rho_u\dd y\dd u.
\end{equation}
Integrating by parts in space in \eqref{eq:weakfokkerplanck} shows that
$\rho$ satisfies the Fokker--Planck equation in the sense of distributions:
\begin{equation}\label{eq:regularfokkerplanck}
 \partial_t\rho_t=-\operatorname{div}(\mathfrak b\rho_t)
                   +\gamma\eps_t\Delta_z\rho_t.
\end{equation}
By \eqref{eq:densitylocalregularity}, the
right-hand side of \eqref{eq:regularfokkerplanck} and its first spatial
derivatives are continuous in $(t,y)$. Equation
\eqref{eq:regularfokkerplanck}, initially understood in the distributional sense,
therefore identifies both $\rho_t$ and $\nabla\rho_t$ as continuously
differentiable in time on compact sets, with one-sided derivatives at the
endpoints. In particular,
\[
 \partial_t\nabla\rho_t
 =\nabla\left[-\operatorname{div}(\mathfrak b\rho_t)
               +\gamma\eps_t\Delta_z\rho_t\right].
\]
Positivity from \eqref{eq:densitybounds} then permits the time chain rule
for the localized entropy and Fisher terms below. The same argument applies
to \eqref{eq:model} with $\eps_t\equiv\eps>0$, whose generator is
$\mathcal L_\eps$ in \eqref{eq:generator}.

Let $\chi_R\in C_c^\infty$ satisfy
\begin{equation}\label{eq:spatialcutoffs}
 \mathbf 1_{B_R}\leq\chi_R\leq\mathbf 1_{B_{2R}},
 \qquad \left|\partial^\alpha\chi_R\right|\leq C_\alpha R^{-|\alpha|}.
\end{equation}
The two bounds in \eqref{eq:densitybounds} imply
$|\log\rho_t(y)|\leq C_T(1+|y|^2)$. For
$w=\nabla\log(\rho/\pi_\eps)$, define the localized functional
\begin{equation}\label{eq:localizedfunctionaldefinition}
 \cH_{\eps,R}(\rho)
 :=\int_{\R^{3d}}\chi_Rw^\top Mw\,\rho\,\dd y
 +\frac{A_0}{\eps}\int_{\R^{3d}}
   \chi_R\log\frac{\rho}{\pi_\eps}\rho\,\dd y.
\end{equation}
For the frozen dynamics, differentiation and integration by parts give
\begin{align}\label{eq:localizedfrozen}
 \frac{\dd}{\dd t}\cH_{\eps,R}(\rho_t)
 =&-\int_{\R^{3d}}\chi_R
 w^\top\left(-J_{\nabla^2U}M-MJ_{\nabla^2U}^\top
             +A_0\gamma P_z\right)w\,\dd\nu_t\notag\\
 &\qquad-2\gamma\eps\sum_{i=1}^d\int_{\R^{3d}}\chi_R
  (\partial_{z_i}w)^\top M(\partial_{z_i}w)\,\dd\nu_t
  +\mathcal R_R^{\rm fr}.
\end{align}
The remainder is explicitly
\begin{equation}\label{eq:localizedremainder}
 \mathcal R_R^{\rm fr}
 =\int_{\R^{3d}}(\mathcal L_\eps\chi_R)
 \left(w^\top Mw+\frac{A_0}{\eps}
       \log\frac{\rho_t}{\pi_\eps}\right)\dd\nu_t.
\end{equation}
Indeed, write $h=\rho_t/\pi_\eps$ and $g=\log h$, and use
$\partial_t h=\mathcal L_\eps^\dagger h$ and
$\mathcal L_\eps+\mathcal L_\eps^\dagger=2(\gamma\eps\Delta_z-\gamma z\cdot\nabla_z)$;
see \eqref{eq:relativedensityevolution} and \eqref{eq:adjointgenerator}.
Here $\mathcal L_\eps^\dagger$ is the adjoint of $\mathcal L_\eps$ in $L^2(\pi_\eps)$.
To check the cutoff terms, let $q=w^\top Mw$ and
$\mathcal S=\gamma\eps\Delta_z-\gamma z\cdot\nabla_z$. The product rule and
integration by parts against $\dd\nu_t=h\,\dd\pi_\eps$ give
\begin{equation}\label{eq:cutoffproduct}
 \mathcal L_\eps(\chi_Rq)
 =\chi_R\mathcal L_\eps q+q\mathcal L_\eps\chi_R
   +2\gamma\eps\sum_{i=1}^d(\partial_{z_i}\chi_R)(\partial_{z_i}q),
\end{equation}
and
\begin{equation}\label{eq:cutoffouparts}
 2\int_{\R^{3d}}\chi_R\mathcal S q\,\dd\nu_t
 =-2\gamma\eps\sum_{i=1}^d\int_{\R^{3d}}
   \left(\chi_Rw_{z_i}+\partial_{z_i}\chi_R\right)
   \partial_{z_i}q\,\dd\nu_t.
\end{equation}
Adding \eqref{eq:cutoffproduct} and \eqref{eq:cutoffouparts}
cancels the integrals containing
$2\gamma\eps(\partial_{z_i}\chi_R)(\partial_{z_i}q)$.
Since
$\partial_{z_i}q=2w^\top M\partial_{z_i}w$, the terms containing
$\chi_Rw_{z_i}$ cancel the nonlinear term from
$\partial_t\nabla\log h$; see \eqref{eq:gradlogdensityevolution}.
Using \eqref{eq:gradlogdensityevolution},
\eqref{eq:cutoffproduct}, and \eqref{eq:cutoffouparts} gives
\begin{equation}\label{eq:localizedfisherderivative}
\begin{aligned}
 \frac{\dd}{\dd t}\int_{\R^{3d}}\chi_Rq\,\dd\nu_t
 &=\int_{\R^{3d}}\chi_R
       w^\top\left(J_{\nabla^2U}M+MJ_{\nabla^2U}^\top\right)w\,\dd\nu_t\\
 &\quad-2\gamma\eps\sum_{i=1}^d\int_{\R^{3d}}\chi_R
       (\partial_{z_i}w)^\top M(\partial_{z_i}w)\,\dd\nu_t
       +\int_{\R^{3d}}(\mathcal L_\eps\chi_R)q\,\dd\nu_t.
\end{aligned}
\end{equation}
Applying the product rule in \eqref{eq:cutoffproduct} with $q$ replaced
by $g$, and using \eqref{eq:logdensityevolution}, similarly gives
\begin{equation}\label{eq:localizedentropyderivative}
 \frac{\dd}{\dd t}\int_{\R^{3d}}\chi_Rg\,\dd\nu_t
 =-\gamma\eps\int_{\R^{3d}}\chi_R|w_z|^2\dd\nu_t
   +\int_{\R^{3d}}(\mathcal L_\eps\chi_R)g\,\dd\nu_t,
\end{equation}
Combining \eqref{eq:localizedfisherderivative} and
\eqref{eq:localizedentropyderivative} with \eqref{eq:localizedfunctionaldefinition} proves \eqref{eq:localizedfrozen} and
\eqref{eq:localizedremainder}.
For a fixed law $\nu=\rho\,\dd y$,
the Gibbs normalization derivative in \eqref{eq:gibbsnormalizationderivative}
and \eqref{eq:localizedfunctionaldefinition} give
\begin{align}\label{eq:localizedtemperature}
 \partial_\eps\cH_{\eps,R}(\rho)
 =&-\frac{2}{\eps^2}\int_{\R^{3d}}
       \chi_Rw^\top M\nabla H\,\dd\nu
 -\frac{A_0}{\eps^2}\int_{\R^{3d}}
       \chi_R\log\frac{\rho}{\pi_\eps}\,\dd\nu\notag\\
 &\qquad-\frac{A_0}{\eps^3}\int_{\R^{3d}}
       \chi_R\left(H-\pi_\eps(H)\right)\,\dd\nu.
\end{align}
For $R\geq1$, the linear growth of $\mathfrak b$ and the cutoff bounds give
\[
 |\mathcal L_\eps\chi_R(y)|
 \leq C_T\mathbf 1_{\{R\leq|y|\leq2R\}}
       \left(\frac{1+|y|}{R}+\frac{1}{R^2}\right)
 \leq C_T\mathbf 1_{\{R\leq|y|\leq2R\}}.
\]
On the compact temperature range $[\eps_T,\eps_0]$,
\eqref{eq:densitybounds} and the quadratic growth of $H$ imply
$|\log(\rho_t/\pi_\eps)|\leq C_T(1+|y|^2)$, while
$|w|^2\leq C_T(1+|y|^2+|\nabla\log\rho_t|^2)$. Consequently,
\begin{equation}\label{eq:localizedremainderbound}
\begin{aligned}
 |\mathcal R_R^{\rm fr}(t)|
 &\leq C_T\int_{\{|y|\geq R\}}
       \left(1+|y|^2+|\nabla\log\rho_t|^2\right)\rho_t\dd y\\
 &\leq\frac{C_T}{R^2}\int_{\R^{3d}}(1+|y|^2)
       \left(1+|y|^2+|\nabla\log\rho_t|^2\right)\rho_t\dd y
 \leq\frac{C_T}{R^2}.
\end{aligned}
\end{equation}
The last bound follows from \eqref{eq:logdensitybounds} and the finite-time
fourth-moment bound furnished by \eqref{eq:pathwise-flow-growth} for this
regular initial density. Thus the remainder vanishes uniformly in time and
in $L^1([0,T])$. The estimates are also valid with $\eps=\eps_t$.

Let
\begin{equation}\label{eq:localizationdominator}
 A_T(u,y):=1+|y|^2+|\nabla\log\rho_u(y)|^2
                         +|\nabla^2\log\rho_u(y)|^2.
\end{equation}
Equations \eqref{eq:logdensitybounds}, \eqref{eq:densitybounds}, and
\eqref{eq:pathwise-flow-growth} imply
$\sup_{0\leq u\leq T}\int_{\R^{3d}}A_T(u,y)\rho_u(y)\dd y<\infty$.
Denote by $\mathcal D_{\eps,R}(\rho)$ the
two integral terms on the right-hand side of \eqref{eq:localizedfrozen},
excluding the remainder, and by $\mathcal D_\eps(\rho)$ the right-hand side
of \eqref{eq:frozenidentity}, both evaluated at $\nu=\rho\,\dd y$.
Boundedness of $J_{\nabla^2U}$, the compact temperature range
$[\eps_T,\eps_0]$, and \eqref{eq:localizedtemperature} give
\begin{equation}\label{eq:localizationconvergence}
\begin{aligned}
 &|\mathcal D_{\eps_u,R}(\rho_u)-\mathcal D_{\eps_u}(\rho_u)|
   +\left|\left.\partial_\eps
        [\cH_{\eps,R}(\rho_u)-\cH_\eps(\nu_u)]\right|_{\eps=\eps_u}\right|\leq C_T\int_{\{|y|\geq R\}}A_T(u,y)\rho_u(y)\dd y,\\
 &\int_0^T\int_{\{|y|\geq R\}}
             A_T(u,y)\rho_u(y)\dd y\dd u\longrightarrow0.
\end{aligned}
\end{equation}
The last limit follows by dominated convergence on $[0,T]\times\R^{3d}$.
The same tail bound without the Hessian term gives
$\cH_{\eps_u,R}(\rho_u)\to\cH_{\eps_u}(\nu_u)$ for each $u$.
For the annealed law, the time chain rule applies
\eqref{eq:localizedfrozen} at $\eps=\eps_t$ and adds $\eps_t'$ times
\eqref{eq:localizedtemperature}, giving
\begin{equation}\label{eq:localizedannealedidentity}
 \frac{\dd}{\dd t}\cH_{\eps_t,R}(\rho_t)
 =\mathcal D_{\eps_t,R}(\rho_t)+\mathcal R_R^{\rm fr}(t)
 +\eps_t'\left.\partial_\eps\cH_{\eps,R}(\rho_t)\right|_{\eps=\eps_t},
\end{equation}
where the remainder in \eqref{eq:localizedremainder} is also evaluated at
$\eps=\eps_t$. Integrating \eqref{eq:localizedannealedidentity} and using
the $L^1$ convergence in \eqref{eq:localizationconvergence} and
\eqref{eq:localizedremainderbound}
to pass to the limit gives, for every $0\leq t\leq T$,
\begin{equation}\label{eq:regularintegratedidentity}
 F(t)-F(0)=\int_0^t\left[
 \mathcal D_{\eps_u}(\rho_u)
 +\eps_u'\left.\partial_\eps\cH_\eps(\nu_u)\right|_{\eps=\eps_u}
 \right]\dd u.
\end{equation}
By \eqref{eq:localizationdominator} and the same estimates,
\[
 \int_0^T\left(
       |\mathcal D_{\eps_u}(\rho_u)|
       +\left|\eps_u'\left.\partial_\eps\cH_\eps(\nu_u)
                                 \right|_{\eps=\eps_u}\right|\right)\dd u
 \leq C_T\int_0^T\int_{\R^{3d}}A_T(u,y)\rho_u(y)\dd y\dd u<\infty.
\]
Thus \eqref{eq:regularintegratedidentity} proves $F\in AC([0,T])$ for every
regular initial density; in particular, each $F_n$ in the statement is
absolutely continuous. Taking the temperature constant also proves the
integrated frozen identity \eqref{eq:frozenidentity}, while
\eqref{eq:localizedtemperature} converges to
\eqref{eq:temperatureidentity}.

\smallskip
\noindent\emph{Step 2: approximation of the initial density.}
Let $f=\sqrt{\rho_0}$. Since $\nabla H$ has at most linear growth, finite
relative Fisher information and finite second moment give the weak identity
$\nabla f=\tfrac12f(w-\eps_0^{-1}\nabla H)$, where
$w=\nabla\log(\dd\nu_0/\dd\pi_{\eps_0})$. Consequently,
\begin{equation}\label{eq:rawfisherfinite}
 \int_{\R^{3d}}|\nabla f|^2\dd y
 \leq\frac12\cI(\nu_0\mid\pi_{\eps_0})
       +\frac1{2\eps_0^2}\int_{\R^{3d}}|\nabla H|^2\dd\nu_0<\infty,
       \qquad f\in H^1(\R^{3d}).
\end{equation}
Indeed, write $p_{\eps_0}$ for the Lebesgue density of $\pi_{\eps_0}$ and
$h_0=\rho_0/p_{\eps_0}$. The weak Fisher convention
\eqref{eq:weakfisherdefinition} gives $\sqrt{h_0}\in H^1_{\rm loc}$.
The weak product rule applied to $f=\sqrt{p_{\eps_0}}\sqrt{h_0}$ yields
\begin{align*}
 \nabla f&=\sqrt{p_{\eps_0}}\nabla\sqrt{h_0}
           -\frac{f}{2\eps_0}\nabla H,\\
 \int_{\R^{3d}}|\nabla f|^2\dd y
 &\leq\frac12\cI(\nu_0\mid\pi_{\eps_0})
     +\frac{1}{2\eps_0^2}\int_{\R^{3d}}|\nabla H|^2\rho_0\dd y<\infty.
\end{align*}
Since $\|f\|_{L^2}^2=1$, this proves \eqref{eq:rawfisherfinite} with weak
derivatives throughout.
Choose $\theta_R$ equal to one on $B_R$, supported in $B_{2R}$,
with $0\leq\theta_R\leq1$ and $|\nabla\theta_R|\leq C/R$. Then
\begin{equation}\label{eq:initialcutoffapprox}
\begin{aligned}
 \nabla(\theta_Rf-f)&=(\theta_R-1)\nabla f+f\nabla\theta_R,\\
 \|\theta_Rf-f\|_{H^1}^2
 &\leq C\int_{\{|y|\geq R\}}(f^2+|\nabla f|^2)\dd y
       +CR^{-2}\int_{\R^{3d}}f^2\dd y\longrightarrow0,\\
 \int_{\R^{3d}}(1+|y|^2)|\theta_R^2f^2-f^2|\dd y
 &\leq\int_{\{|y|\geq R\}}(1+|y|^2)f^2\dd y\longrightarrow0.
\end{aligned}
\end{equation}
Choose $R_n\uparrow\infty$ so that both errors in
\eqref{eq:initialcutoffapprox} are at most $1/n$, and let
$\theta_n=\theta_{R_n}$. Convolution with a
nonnegative mollifier $\varphi_{\delta_n}$, addition of a positive Gaussian
tail, and normalization give
\[
 \widetilde f_n=(\theta_nf)*\varphi_{\delta_n}+a_ne^{-|y|^2},
 \qquad f_n=\frac{\widetilde f_n}{\|\widetilde f_n\|_{L^2}},
 \qquad \rho_0^{(n)}=f_n^2.
\]
Let $k_n=(\theta_nf)*\varphi_{\delta_n}$. Since $k_n$ is supported in a
fixed compact set for each $n$ and small $\delta_n$, convergence in $H^1$
also controls its weighted squared-density error. Choose $\delta_n$ so that
\[
 \|k_n-\theta_nf\|_{H^1}
 +\int_{\R^{3d}}(1+|y|^2)\left|k_n^2-\theta_n^2f^2\right|\dd y\leq\frac1n.
\]
Next choose $a_n>0$ so small that
\[
 a_n\left\|e^{-|y|^2}\right\|_{H^1}
 +\int_{\R^{3d}}(1+|y|^2)
           \left|\left(k_n+a_ne^{-|y|^2}\right)^2-k_n^2\right|\dd y\leq\frac1n.
\]
Normalization preserves these limits because
$\|\widetilde f_n\|_{L^2}\to\|f\|_{L^2}=1$.
Thus $f_n$ is smooth, positive, and Gaussian outside a compact
set, with
\begin{align}
 f_n\longrightarrow f \qquad\text{in }H^1(\R^{3d}),\label{eq:h1approx}
\end{align}
and
\begin{align}
 \int_{\R^{3d}}\left(1+|y|^2\right)\left|f_n^2-f^2\right|\dd y\longrightarrow0.
 \label{eq:weightedapprox}
\end{align}
Set $\rho_0^{(n)}=f_n^2$ and $\nu_0^{(n)}=\rho_0^{(n)}\dd y$.
Let $m=3d$ and $q=m/(m-2)>1$. Since $m\geq3$, the Sobolev embedding
theorem \cite[Theorem~9.9]{Brezis2011}, applied with exponent $2$, gives
\[
 \|f_n-f\|_{L^{2m/(m-2)}(\R^m)}
 \leq C_m\|\nabla(f_n-f)\|_{L^2(\R^m)}\longrightarrow0
\]
by \eqref{eq:h1approx}. Together with the $L^2$ convergence in
\eqref{eq:h1approx}, this yields
\begin{align*}
 \left\|\rho_0^{(n)}-\rho_0\right\|_{L^1}
 &\leq \|f_n-f\|_{L^2}(\|f_n\|_{L^2}+\|f\|_{L^2}),\\
 \left\|\rho_0^{(n)}-\rho_0\right\|_{L^q}
 &\leq \|f_n-f\|_{L^{2q}}
          (\|f_n\|_{L^{2q}}+\|f\|_{L^{2q}}),
\end{align*}
and both right-hand sides tend to zero because $2q=2m/(m-2)$.
Thus $\rho_0^{(n)}\to\rho_0$ in $L^1\cap L^{1+\delta}$ with
$\delta=q-1=2/(m-2)$. If $G_+(u):=(u\log u)_+$, then
\[
 |G_+(u)-G_+(v)|
 \leq C_q|u-v|\left(1+u^{q-1}+v^{q-1}\right),\qquad u,v\geq0.
\]
By H\"older's inequality,
\begin{align*}
 \left\|G_+\left(\rho_0^{(n)}\right)-G_+(\rho_0)\right\|_{L^1}
 &\leq C_q\left\|\rho_0^{(n)}-\rho_0\right\|_{L^1}\\
 &\quad+C_q\left\|\rho_0^{(n)}-\rho_0\right\|_{L^q}
       \left(\left\|\rho_0^{(n)}\right\|_{L^q}^{q-1}
                    +\|\rho_0\|_{L^q}^{q-1}\right)
 \longrightarrow0.
\end{align*}
In particular, the positive entropy parts are uniformly integrable.

For $G_-(u):=(-u\log u)_+$, use
\begin{equation}\label{eq:negativeentropypointwise}
 u\log(1/u)\leq |y|^2u+e^{-|y|^2-1},
 \qquad 0<u\leq1,
\end{equation}
together with \eqref{eq:weightedapprox}. On every ball,
$0\leq G_-\left(\rho_0^{(n)}\right)\leq e^{-1}$ and convergence in measure follows from
the $L^1$ convergence above. Dominated convergence along almost-everywhere
convergent subsequences handles the integral on the ball. Outside a ball, the
weighted convergence gives
\[
 \lim_{R\to\infty}\sup_{n\geq1}
 \int_{\{|y|>R\}}|y|^2\rho_0^{(n)}(y)\dd y=0.
\]
By \eqref{eq:negativeentropypointwise},
\[
 \sup_n\int_{\{|y|>R\}}G_-\left(\rho_0^{(n)}(y)\right)\dd y
 \leq\sup_n\int_{\{|y|>R\}}|y|^2\rho_0^{(n)}(y)\dd y
       +\int_{\{|y|>R\}}e^{-|y|^2-1}\dd y
 \longrightarrow0.
\]
Consequently, the negative entropy parts are uniformly
integrable and their integrals converge. Since
$|H(y)|\leq C(1+|y|^2)$, \eqref{eq:weightedapprox} also controls the energy
term. We have proved
\begin{equation}\label{eq:entropyapprox}
 \int_{\R^{3d}}\rho_0^{(n)}\log\rho_0^{(n)}\dd y
 \longrightarrow \int_{\R^{3d}}\rho_0\log\rho_0\dd y,
 \qquad
 \int_{\R^{3d}} H\left|\rho_0^{(n)}-\rho_0\right|\dd y\longrightarrow0.
\end{equation}

Define
\[
 \cI_M^{\rm raw}(\rho):=\int_{\R^{3d}}(\nabla\log\rho)^\top
 M(\nabla\log\rho)\rho\dd y.
\]
Equation \eqref{eq:h1approx} gives
\[
 \cI_M^{\rm raw}\left(\rho_0^{(n)}\right)
 =4\int_{\R^{3d}}(\nabla f_n)^\top M\nabla f_n\dd y
 \longrightarrow\cI_M^{\rm raw}(\rho_0).
\]
For any density $\rho$ under consideration,
\begin{align}\label{eq:relativefisherexpand}
 \cI_M(\rho\dd y\mid\pi_\eps)
 =\cI_M^{\rm raw}(\rho)
 +\frac{2}{\eps}\int_{\R^{3d}}(\nabla\rho)^\top M\nabla H\dd y
 +\frac{1}{\eps^2}\int_{\R^{3d}}(\nabla H)^\top M\nabla H\,\rho\dd y,
\end{align}
The cross term in \eqref{eq:relativefisherexpand} is integrable, since
\[
 \int_{\R^{3d}}|\nabla\rho|\,|M\nabla H|\dd y
 \leq2\|\nabla\sqrt\rho\|_{L^2}
           \left(\int_{\R^{3d}}\rho\,|M\nabla H|^2\dd y\right)^{1/2}<\infty.
\]
Using the cutoffs in \eqref{eq:spatialcutoffs},
\begin{equation}\label{eq:fishercrosscutoff}
\begin{aligned}
 \int_{\R^{3d}}\chi_R(\nabla\rho)^\top M\nabla H\dd y
 &=-\int_{\R^{3d}}\chi_R\rho\,\operatorname{tr}(M\nabla^2H)\dd y
       -\int_{\R^{3d}}\rho(\nabla\chi_R)^\top M\nabla H\dd y,\\
 \left|\int_{\R^{3d}}\rho(\nabla\chi_R)^\top M\nabla H\dd y\right|
 &\leq C\int_{\{|y|\geq R\}}\rho\dd y\longrightarrow0.
\end{aligned}
\end{equation}
Boundedness of $\nabla^2H$ and \eqref{eq:fishercrosscutoff} therefore yield
\[
 \int_{\R^{3d}}(\nabla\rho)^\top M\nabla H\dd y
 =-\int_{\R^{3d}}\rho\,\operatorname{tr}\left(M\nabla^2H\right)\dd y.
\]
The Hessian of $H$ is bounded and $|\nabla H(y)|\leq C(1+|y|)$.
Consequently, \eqref{eq:weightedapprox}--\eqref{eq:relativefisherexpand}
imply
\begin{equation}\label{eq:functionalapprox}
 \cH_{\eps_0}\left(\nu_0^{(n)}\right)\longrightarrow
 \cH_{\eps_0}(\nu_0).
\end{equation}

\smallskip
\noindent\emph{Step 3: passage to the limit.}
Let $\nu_t^{(n)}$ and $\nu_t$ be the laws starting from
$\nu_0^{(n)}$ and $\nu_0$, respectively. Couple the common part of their
densities identically, and couple the two remaining equal masses
independently. Since $|y-y'|^2\leq2|y|^2+2|y'|^2$, this gives
\begin{equation}\label{eq:weightedtowasserstein}
 \mathcal W_2^2\left(\nu_0^{(n)},\nu_0\right)
 \leq2\int_{\R^{3d}}|y|^2\left|\rho_0^{(n)}-\rho_0\right|\dd y\longrightarrow0
\end{equation}
by \eqref{eq:weightedapprox}. Combining
\eqref{eq:weightedtowasserstein} with synchronous coupling and the Lipschitz
drift gives
\begin{equation}\label{eq:floww2}
 \sup_{0\leq t\leq T}\mathcal W_2\left(\nu_t^{(n)},\nu_t\right)
 \leq C_T\mathcal W_2\left(\nu_0^{(n)},\nu_0\right)\longrightarrow0.
\end{equation}
In particular, \eqref{eq:weightedapprox} gives
\[
 C_2:=\sup_{n\geq1}\int_{\R^{3d}}
 \left(1+|y|^2\right)\dd\nu_0^{(n)}(y)<\infty.
\]
Let $S_0^{(n)}$ have law $\nu_0^{(n)}$ and let
$S_t^{(n)}=\Phi_t\left(S_0^{(n)}\right)$. Applying \eqref{eq:pathwise-flow-growth}
pathwise and then taking expectations gives
\begin{equation}\label{eq:flowsecondmoment}
 \sup_{n\geq1}\sup_{0\leq t\leq T}
 \E\left[1+\left|S_t^{(n)}\right|^2\right]
 \leq C_T\left(C_2+\E N_T^2\right).
\end{equation}
Doob's $L^2$ maximal inequality for the noise martingale and It\^o's
isometry give the explicit bound
\begin{equation}\label{eq:noiseLtwo}
 \E N_T^2\leq4\E|G_T|^2
 =8\gamma d\int_0^T\eps_u\dd u
 \leq8\gamma d\eps_0T.
\end{equation}
Since $S_t^{(n)}$ has law $\nu_t^{(n)}$,
\eqref{eq:flowsecondmoment} and \eqref{eq:noiseLtwo} yield
\begin{equation}\label{eq:regularizedmoments}
 \sup_{n\geq1}\sup_{0\leq t\leq T}
 \int_{\R^{3d}}\left(1+|y|^2\right)\dd\nu_t^{(n)}(y)\leq C_{T,C_2}.
\end{equation}

We first verify that the limiting law has a full weak score.
Conditionally on the Brownian path, \eqref{eq:flowdensity} gives the density
$a_t^{(n)}(y)=e^{\gamma dt}\rho_0^{(n)}(\Psi_t(y))$.
The chain rule, change of variables, and \eqref{eq:inverseflowbounds} imply
\begin{equation}\label{eq:flowrawfisher}
\begin{aligned}
 \nabla\log a_t^{(n)}(y)
 &=(D\Psi_t(y))^\top\nabla\log\rho_0^{(n)}(\Psi_t(y)),\\
 \cI^{\rm raw}\left(a_t^{(n)}\dd y\right)
 &\leq e^{2CT}\cI^{\rm raw}\left(\nu_0^{(n)}\right),\\
 \cI^{\rm raw}\left(\nu_t^{(n)}\right)
 &\leq\E\left[\cI^{\rm raw}\left(a_t^{(n)}\dd y\right)\right]
 \leq e^{2CT}\cI^{\rm raw}\left(\nu_0^{(n)}\right),\qquad 0\leq t\leq T.
\end{aligned}
\end{equation}
For the third line, the conditional Cauchy--Schwarz inequality and
$\rho_t^{(n)}=\E a_t^{(n)}$ give, for almost every $y$,
\[
 \frac{|\nabla\rho_t^{(n)}(y)|^2}{\rho_t^{(n)}(y)}
 =\frac{|\E[a_t^{(n)}(y)\nabla\log a_t^{(n)}(y)]|^2}
             {\E[a_t^{(n)}(y)]}
 \leq\E\left[a_t^{(n)}(y)\left|\nabla\log a_t^{(n)}(y)\right|^2\right].
\]
Integration over $\R^{3d}$ gives the third line of
\eqref{eq:flowrawfisher}. Equation \eqref{eq:h1approx} gives
$\sup_n\cI^{\rm raw}\left(\nu_0^{(n)}\right)<\infty$. Hence
\eqref{eq:flowrawfisher}, \eqref{eq:regularizedmoments}, and
\eqref{eq:rawrelativefisherbound} yield
\begin{equation}\label{eq:uniformfullfisher}
 \sup_{n\geq1}\sup_{0\leq t\leq T}
       \cI\left(\nu_t^{(n)}\mid\pi_{\eps_t}\right)<\infty.
\end{equation}
For fixed $t\in[0,T]$, let $g_n:=\sqrt{\rho_t^{(n)}}$.
By \eqref{eq:flowrawfisher} and \eqref{eq:h1approx},
\[
 \|g_n\|_{H^1(\R^{3d})}^2
 =1+\frac14\cI^{\rm raw}(\nu_t^{(n)})\leq C_T.
\]
Since $H^1(\R^{3d})$ is reflexive,
\cite[Theorem~3.18]{Brezis2011} gives a subsequence converging weakly to
some $f_t\in H^1(\R^{3d})$.
For each integer $R\geq1$, the ball $B_R:=\{y\in\R^{3d}:|y|<R\}$
is bounded with smooth boundary. The Rellich--Kondrachov compactness
theorem \cite[Theorem~9.16]{Brezis2011}, with both exponents equal to $2$,
gives the compact embedding $H^1(B_R)\hookrightarrow L^2(B_R)$.
Successively extract nested subsequences on $B_1,B_2,\ldots$; each
strong limit is the restriction of the weak limit $f_t$. Choosing an
increasing diagonal sequence $(n_j)$ gives
\[
 \|g_{n_j}-f_t\|_{L^2(B_R)}\longrightarrow0
 \qquad\text{for every fixed integer }R\geq1.
\]
Thus the same subsequence satisfies
\begin{equation}\label{eq:limitingdensityHone}
 \sqrt{\rho_t^{(n_j)}}\rightharpoonup f_t\text{ in }H^1(\R^{3d}),
 \qquad
 \sqrt{\rho_t^{(n_j)}}\to f_t\text{ in }L^2_{\rm loc}(\R^{3d}).
\end{equation}
The weak convergence in \eqref{eq:floww2} identifies
$\nu_t=f_t^2\dd y$. Weak lower semicontinuity, the second-moment bound,
and \eqref{eq:rawrelativefisherbound} show that
$\cI(\nu_t\mid\pi_{\eps_t})<\infty$.

For a law with finite full relative Fisher information, the variational
representation needed here is
\begin{equation}\label{eq:fisherdual}
 \cI_M(\nu\mid\pi_\eps)
 =\sup_{\phi\in C_c^\infty(\R^{3d};\R^{3d})}
 \int_{\R^{3d}}\left[-2\operatorname{div}(M\phi)
 +\frac{2}{\eps}\phi^\top M\nabla H-\phi^\top M\phi\right]\dd\nu.
\end{equation}
Indeed, writing $w=\nabla\log(\dd\nu/\dd\pi_\eps)\in L^2(\nu)$,
weak integration by parts and completion of the square give
\begin{align*}
 &\int_{\R^{3d}}\left[-2\operatorname{div}(M\phi)
 +\frac{2}{\eps}\phi^\top M\nabla H-\phi^\top M\phi\right]\dd\nu\\
 &\qquad=\int_{\R^{3d}}
       \left(2\phi^\top Mw-\phi^\top M\phi\right)\dd\nu\\
 &\qquad=\cI_M(\nu\mid\pi_\eps)
       -\int_{\R^{3d}}(w-\phi)^\top M(w-\phi)\dd\nu
 \leq\cI_M(\nu\mid\pi_\eps).
\end{align*}
Choose $\phi_j\in C_c^\infty$ with $\|\phi_j-w\|_{L^2(\nu)}\to0$;
such a sequence follows by compact truncation, approximation by continuous
functions in the finite measure $\nu$, and uniform smoothing on their
compact supports. Then
\[
 \int_{\R^{3d}}(w-\phi_j)^\top M(w-\phi_j)\dd\nu
 \leq\|M\|\|\phi_j-w\|_{L^2(\nu)}^2\longrightarrow0,
\]
which proves \eqref{eq:fisherdual} for this class of laws.
For every fixed compactly supported $\phi$, the integrand in
\eqref{eq:fisherdual} is bounded and continuous, because $\nabla H$ is bounded
on the support of $\phi$. Hence, if $\nu_n\Rightarrow\nu$,
\[
 \int_{\R^{3d}} G_\phi\dd\nu=\lim_{n\to\infty}\int_{\R^{3d}} G_\phi\dd\nu_n
 \leq\liminf_{n\to\infty}\cI_M(\nu_n\mid\pi_\eps),
\]
where $G_\phi$ denotes the integrand in \eqref{eq:fisherdual}.
Applying this inequality to $\nu_n=\nu_t^{(n)}$, whose limiting full score
was established in \eqref{eq:limitingdensityHone}, and taking the supremum
over $\phi$ gives
\begin{equation}\label{eq:fisherlowersemicontinuity}
 \cI_M(\nu_t\mid\pi_{\eps_t})
 \leq\liminf_{n\to\infty}\cI_M\left(\nu_t^{(n)}\mid\pi_{\eps_t}\right).
\end{equation}
Relative entropy is lower semicontinuous under the same weak convergence.

By Step~1, each $F_n$ is absolutely continuous on $[0,T]$. By hypothesis,
\eqref{eq:abstractcomparison} holds with the same $a,b$ for every $n$.
The integrating-factor formula consequently gives
\[
 \cH_{\eps_t}\left(\nu_t^{(n)}\right)
 \leq e^{-\int_0^t a(u)\dd u}\cH_{\eps_0}\left(\nu_0^{(n)}\right)
 +\int_0^t e^{-\int_v^t a(u)\dd u}b(v)\dd v.
\]
Equation \eqref{eq:functionalapprox} gives convergence of the first term.
Equations \eqref{eq:floww2} and \eqref{eq:fisherlowersemicontinuity}, together with lower
semicontinuity of relative entropy, give
\begin{equation}\label{eq:functionallowersemicontinuity}
 \cH_{\eps_t}(\nu_t)\leq
 \liminf_{n\to\infty}\cH_{\eps_t}\left(\nu_t^{(n)}\right).
\end{equation}
Taking the lower limit proves \eqref{eq:abstractintegrated} after undoing the
time shift. On a fixed finite interval, the differential inequality
\eqref{eq:finiteintervalfunctionalode} has a coefficient $C_T$ independent
of $n$, by \eqref{eq:regularizedmoments}. Applying
\eqref{eq:abstractintegrated} with $a=-C_T$ and $b=C_T$ gives, after the
time shift,
\[
 F(t)\leq e^{C_T(t-s)}F(s)+C_T\int_s^t e^{C_T(t-u)}\dd u
       =e^{C_T(t-s)}(1+F(s))-1<\infty,\qquad s\leq t\leq T.
\]
\end{proof}

\subsection{Annealing estimate}

The coercive matrix gives decay of the distorted entropy, while the
approximation lemma extends this estimate to the initial laws in
Assumption~\ref{ass:initial}. Combining this decay with concentration of the
Gibbs law yields the following optimization guarantee.

\begin{theorem}[Continuous-time annealing]\label{thm:main}
Under Assumptions~\ref{ass:potential}, \ref{ass:initial}, and \ref{ass:lsi}, for every
$\delta>0$ and $\alpha>0$ there is $C_{\delta,\alpha}<\infty$ such that, for
all sufficiently large $t$,
\begin{equation}\label{eq:mainrate}
 \mathbb P\left(U(X_t)>\delta\right)
 \leq C_{\delta,\alpha}(t+t_0)^{-\min\{\delta/E,\,(1-D/E)/2\}+\alpha}.
\end{equation}
Here $E$ is the cooling limit in \eqref{eq:coolingassumption}, and $D$ is the
LSI barrier in \eqref{eq:lsi}.
\end{theorem}

\begin{remark}[Interpretation of the continuous-time rate]\label{rem:noacceleration}
With $D$ identified with the critical depth $E_*$, the exponent in
\eqref{eq:mainrate} is the same as the kinetic annealing exponent of He, Tan,
and Wu \cite[Theorem~2.5]{HeTanWu2024}; the extra hypoelliptic level therefore
matches the kinetic polynomial exponent. For fixed $\lambda$ and $\gamma$, the
definitions in \eqref{eq:matrix} give
$b=\mathcal{O}(K^2)$, $\varsigma=\mathcal{O}(K^4)$, and $A_0=\mathcal{O}(K^8)$ as $K\to\infty$.
Thus, Theorem~\ref{thm:main} gives the same continuous-time annealing exponent
as underdamped Langevin dynamics. He, Tan, and Wu
\cite[Remark~2.6]{HeTanWu2024} identify the same barrier-controlled exponent
for overdamped annealing. The quantitative distinction established here is at
the discretization level, while fixed-temperature Eyring--Kramers theory
captures the complementary prefactor improvement.
\end{remark}

\begin{proof}[Proof of Theorem~\ref{thm:main}]
We separate moment control, frozen dissipation, and the cooling contribution.

\emph{Moment control.}
For $a>0$ small, let $b_a=2a/\lambda$, $c_a=a\lambda/\gamma$, and
\begin{equation}\label{eq:lyapunovdefinition}
 V_a=1+H+a x\cdot v+b_a v\cdot z+c_a x\cdot z
                       +\frac12c_a\lambda|x|^2+C_a.
\end{equation}
We choose $a$ and $C_a$ below. Assumption~\ref{ass:potential} gives
quadratic lower and upper bounds on $U$.
To make them explicit, set
$R_0:=\max\{1,\sqrt{2m/r}\}$. If $x=\rho\theta$ with $|\theta|=1$ and
$\rho\geq R_0$, then
\[
 \frac{\dd}{\dd\rho}U(\rho\theta)
 =\theta\cdot\nabla U(\rho\theta)
 \geq r\rho-\frac m\rho\geq\frac r2\rho.
\]
After integration from $R_0$ to $\rho$ and adjustment on the ball $B_{R_0}$,
this gives $U(x)\geq r|x|^2/4-C_-$. On the other hand, Taylor's formula and
$\|\nabla^2U\|_{\rm op}\leq K$ give
\[
 U(x)\leq U(0)+|\nabla U(0)||x|+\frac K2|x|^2
 \leq C_+(1+|x|^2).
\]
Thus,
\begin{equation}\label{eq:quadraticU}
 \frac r4|x|^2-C_-\leq U(x)\leq C_+(1+|x|^2).
\end{equation}
Let $c_H=\min\{r/4,1/2\}$ and
$C_{\rm cross}=(1+2/\lambda+\lambda/\gamma)/2$. The cross terms in
\eqref{eq:lyapunovdefinition} satisfy
\[
 |ax\cdot v+b_av\cdot z+c_ax\cdot z|
 \leq aC_{\rm cross}(|x|^2+|v|^2+|z|^2).
\]
Taking $aC_{\rm cross}\leq c_H/2$ and $C_a\geq C_-$ in
\eqref{eq:quadraticU} gives constants $c_V,C_V>0$ such that
\begin{equation}\label{eq:lyapunovcomparison}
 c_V(1+|y|^2)\leq V_a(y)\leq C_V(1+|y|^2),\qquad V_a\geq1.
\end{equation}
Applying \eqref{eq:generator} to \eqref{eq:lyapunovdefinition},
the $x\cdot v$ terms from $c_ax\cdot z$ and $c_a\lambda|x|^2/2$
cancel, and the $x\cdot z$ coefficient vanishes because
$a\lambda=\gamma c_a$. The resulting expression is
\begin{equation}\label{eq:lyapunovgenerator}
\begin{aligned}
 \mathcal L_\eps V_a={}&-a x\cdot\nabla U-a|v|^2+(2a-\gamma)|z|^2
 +(c_a-\gamma b_a)v\cdot z-b_a\nabla U\cdot z+\gamma\eps d.
\end{aligned}
\end{equation}
For clarity, let $c_0=|\nabla U(0)|$ and
$c_1=|\lambda/\gamma-2\gamma/\lambda|$. Young's inequality gives
\begin{equation}\label{eq:lyapunovyoung}
\begin{aligned}
 a c_1|v||z|&\leq\frac a4|v|^2+a c_1^2|z|^2,\\
 \frac{2a}{\lambda}(K|x|+c_0)|z|
 &\leq\frac{ar}{4}|x|^2
       +a\left(\frac{4K^2}{r\lambda^2}+1\right)|z|^2
       +\frac{a c_0^2}{\lambda^2}.
\end{aligned}
\end{equation}
In addition to the condition ensuring quadratic comparability of $V_a$, take
$a(3+c_1^2+4K^2/(r\lambda^2))\leq\gamma/2$.
Using Assumption~\ref{ass:potential} and \eqref{eq:lyapunovyoung} in
\eqref{eq:lyapunovgenerator}, followed by
\eqref{eq:lyapunovcomparison}, gives
\begin{equation}\label{eq:lyapunovdrift}
 \mathcal L_{\eps_t}V_a
 \leq-\frac{3ar}{4}|x|^2-\frac{3a}{4}|v|^2-\frac{\gamma}{2}|z|^2
       +am+\frac{ac_0^2}{\lambda^2}+\gamma d\sup_{u\geq0}\eps_u
 \leq-cV_a+C.
\end{equation}
Dynkin's formula and Gr\"{o}nwall's inequality now yield
\begin{equation}\label{eq:moments}
 \sup_{t\geq0}\E\left[|X_t|^2+|V_t|^2+|Z_t|^2\right]<\infty.
\end{equation}
By \eqref{eq:lyapunovdefinition} and \eqref{eq:lyapunovcomparison},
$\nabla_zV_a=z+b_av+c_ax$ and $|\nabla_zV_a|^2\leq CV_a$.
The diffusion chain rule and \eqref{eq:lyapunovdrift} then give
\begin{equation}\label{eq:lyapunovpowerdrift}
\begin{aligned}
 \mathcal L_{\eps_t}V_a^p
 &=pV_a^{p-1}\mathcal L_{\eps_t}V_a
   +\gamma\eps_t p(p-1)V_a^{p-2}|\nabla_zV_a|^2\\
 &\leq-pcV_a^p+\widetilde C_pV_a^{p-1}
 \leq-c_pV_a^p+C_p,\qquad p\in\mathbb N.
\end{aligned}
\end{equation}
The last inequality absorbs the lower power $V_a^{p-1}$ into
$pcV_a^p/2$. If $\E[V_a^p(S_0)]<\infty$, let
$\sigma_R=\inf\{t\geq0:|S_t|\geq R\}$. Applying It\^o's formula to
$e^{c_p(t\wedge\sigma_R)}V_a^p(S_{t\wedge\sigma_R})$, using
\eqref{eq:lyapunovpowerdrift}, gives
\[
 \E[e^{c_p(t\wedge\sigma_R)}V_a^p(S_{t\wedge\sigma_R})]
 \leq \E[V_a^p(S_0)]+C_p\int_0^te^{c_pu}\dd u.
\]
Nonexplosion and Fatou's lemma as $R\to\infty$ therefore give
\begin{equation}\label{eq:momentpowerestimate}
 \E[V_a^p(S_t)]
 \leq e^{-c_pt}\E[V_a^p(S_0)]+\frac{C_p}{c_p}(1-e^{-c_pt}).
\end{equation}
Since $V_a$ is comparable to $1+|x|^2+|v|^2+|z|^2$, this yields
\[
 \sup_{t\geq0}\E|S_t|^{2p}
 \leq C_p(1+\E|S_0|^{2p})<\infty.
\]
The case $p=1$ recovers \eqref{eq:moments}, with a constant depending on the
initial law only through its second moment; see \eqref{eq:restartmoment}
for the estimate after restarting.

For the Gibbs law \eqref{eq:gibbs}, we now bound the second moment and
mean energy uniformly for $0<\eps\leq\eps_0$.
Let $\mu_\eps$ be its position marginal, as defined in
Assumption~\ref{ass:lsi}. Integration by parts, with the boundary term
vanishing by the quadratic lower bound in \eqref{eq:quadraticU}, gives
\begin{equation}\label{eq:gibbspositionidentity}
 0=\int_{\R^d}\nabla\cdot\left(xe^{-U(x)/\eps}\right)\dd x
 =dZ_\eps-\eps^{-1}\int_{\R^d}
 x\cdot\nabla U(x)e^{-U(x)/\eps}\dd x.
\end{equation}
Using $x\cdot\nabla U(x)\geq r|x|^2-m$ from
Assumption~\ref{ass:potential} in \eqref{eq:gibbspositionidentity} yields
\begin{equation}\label{eq:gibbspositionmoment}
\begin{aligned}
 &r\int_{\R^d}|x|^2\mu_\eps(\dd x)-m
 \leq\int_{\R^d}x\cdot\nabla U(x)\mu_\eps(\dd x)=d\eps,\\
 &\int_{\R^d}|x|^2\mu_\eps(\dd x)
 \leq\frac{d\eps+m}{r}\leq\frac{d\eps_0+m}{r}.
\end{aligned}
\end{equation}
By the product law \eqref{eq:gibbsproduct}, both $v$ and $z$ have
distribution $\mathcal N(0,\eps\Id_d)$ under $\pi_\eps$.
Using \eqref{eq:gibbs}, \eqref{eq:quadraticU}, and
\eqref{eq:gibbspositionmoment}, we obtain $\pi_\eps(|v|^2)=\pi_\eps(|z|^2)=d\eps$ and
\begin{equation}\label{eq:gibbsenergybound}
\begin{aligned}
 \pi_\eps(H)
 &=\int_{\R^{3d}}H(y)\pi_\eps(\dd y)\\
 &=\int_{\R^d}U(x)\mu_\eps(\dd x)+d\eps\\
 &\leq C_+\left(1+\int_{\R^d}|x|^2\mu_\eps(\dd x)\right)+d\eps
 \leq C_+\left(1+\frac{d\eps_0+m}{r}\right)+d\eps_0.
\end{aligned}
\end{equation}
The upper bounds in \eqref{eq:gibbspositionmoment} and
\eqref{eq:gibbsenergybound} are independent of $\eps\in(0,\eps_0]$.
They therefore imply
\begin{equation}\label{eq:gibbsmoments}
 \sup_{0<\eps\leq\eps_0}
 \left\{\pi_\eps\left(|x|^2+|v|^2+|z|^2\right)+\pi_\eps(H)\right\}<\infty.
\end{equation}

\emph{Frozen dissipation.}
Recall from \eqref{eq:generator} that
\[
 \mathcal L_\eps=v\cdot\nabla_x+(-\nabla U+\lambda z)\cdot\nabla_v
 +(-\lambda v-\gamma z)\cdot\nabla_z+\gamma\eps\Delta_z.
\]
By \eqref{eq:transportinvariance}, the transport operator $\mathcal T$
is antisymmetric in $L^2(\pi_\eps)$, while
$\mathcal S=\gamma\eps\Delta_z-\gamma z\cdot\nabla_z$ is symmetric.
Thus \eqref{eq:generator} gives
\begin{equation}\label{eq:adjointgenerator}
 \mathcal L_\eps^\dagger f
 =p_\eps^{-1}\mathcal L_\eps^*(p_\eps f)
 =-\mathcal T f+\mathcal S f
 =b^\dagger\cdot\nabla f+\gamma\eps\Delta_zf,
 \qquad b^\dagger=(-v,\nabla U-\lambda z,\lambda v-\gamma z).
\end{equation}
Its drift Jacobian is $J_{\nabla^2U}=Db^\dagger$.
Let $\nu_t^\eps$ denote the
law evolving with temperature fixed at $\eps$, and let
$h:=\dd\nu_t^\eps/\dd\pi_\eps$, $g:=\log h$, $w:=\nabla g$, and
$\mathcal S:=\gamma\eps\Delta_z-\gamma z\cdot\nabla_z$.
At fixed temperature, the relative density $h$ satisfies
\begin{equation}\label{eq:relativedensityevolution}
 \partial_t h=\mathcal L_\eps^\dagger h
 =b^\dagger\cdot\nabla h+\gamma\eps\Delta_z h.
\end{equation}
Applying the diffusion chain rule to $g=\log h$ in
\eqref{eq:relativedensityevolution} gives
\begin{equation}\label{eq:logdensityevolution}
 \partial_t g=\mathcal L_\eps^\dagger g+\gamma\eps|\nabla_z g|^2.
\end{equation}
Since $Db^\dagger=J_{\nabla^2U}$, spatial differentiation of
\eqref{eq:logdensityevolution} gives
\begin{equation}\label{eq:gradlogdensityevolution}
 \partial_t w=\mathcal L_\eps^\dagger w+J_{\nabla^2U}^\top w
 +2\gamma\eps\sum_{i=1}^d w_{z_i}\partial_{z_i}w.
\end{equation}
Let $q=w^\top Mw$. By \eqref{eq:functional} and
$\dd\nu_t^\eps=h\,\dd\pi_\eps$,
\[
 \cI_M(\nu_t^\eps\mid\pi_\eps)
 =\int_{\R^{3d}}q\,\dd\nu_t^\eps
 =\int_{\R^{3d}}hq\,\dd\pi_\eps.
\]
Using \eqref{eq:relativedensityevolution} and
\eqref{eq:gradlogdensityevolution}, moving $\mathcal L_\eps^\dagger$ off $h$,
and applying the diffusion product rule with
$\mathcal L_\eps+\mathcal L_\eps^\dagger=2\mathcal S$ yields
\begin{equation}\label{eq:fisherevolutionexpanded}
\begin{aligned}
 \frac{\dd}{\dd t}\cI_M(\nu_t^\eps\mid\pi_\eps)
 &=\int_{\R^{3d}}w^\top
       \left(J_{\nabla^2U}M+MJ_{\nabla^2U}^\top\right)w\,\dd\nu_t^\eps\\
 &\qquad+4\int_{\R^{3d}}w^\top M\mathcal S w\,\dd\nu_t^\eps
   +2\gamma\eps\sum_{i=1}^d\int_{\R^{3d}}
       (\partial_{z_i}w)^\top M\partial_{z_i}w\,\dd\nu_t^\eps\\
 &\qquad\qquad+4\gamma\eps\sum_{i=1}^d\int_{\R^{3d}}
       w_{z_i}w^\top M\partial_{z_i}w\,\dd\nu_t^\eps.
\end{aligned}
\end{equation}
Integration by parts against $h\pi_\eps$ gives
\begin{equation}\label{eq:fisherintegrationbyparts}
 \int_{\R^{3d}}w^\top M\mathcal S w\,\dd\nu_t^\eps
 =-\gamma\eps\sum_{i=1}^d\int_{\R^{3d}}
   \left((\partial_{z_i}w)^\top M\partial_{z_i}w
         +w_{z_i}w^\top M\partial_{z_i}w\right)\dd\nu_t^\eps.
\end{equation}
By the definition of $\KL$ in \eqref{eq:relativeentropydefinition},
with $g=\log h$ and $\dd\nu_t^\eps=h\,\dd\pi_\eps$,
\[
 \KL(\nu_t^\eps\mid\pi_\eps)
 =\int_{\R^{3d}}g\,\dd\nu_t^\eps
 =\int_{\R^{3d}}hg\,\dd\pi_\eps.
\]
Using \eqref{eq:relativedensityevolution} and
\eqref{eq:logdensityevolution}, mass conservation and integration by parts
against $h\pi_\eps$ yield
\begin{equation}\label{eq:entropyevolutionexpanded}
\begin{aligned}
 \frac{\dd}{\dd t}\KL(\nu_t^\eps\mid\pi_\eps)
 &=\int_{\R^{3d}}h
   \left(2\mathcal Sg+\gamma\eps|w_z|^2\right)\dd\pi_\eps\\
 &=-2\gamma\eps\int_{\R^{3d}}|w_z|^2\,\dd\nu_t^\eps
   +\gamma\eps\int_{\R^{3d}}|w_z|^2\,\dd\nu_t^\eps.
\end{aligned}
\end{equation}
Substituting \eqref{eq:fisherintegrationbyparts} into
\eqref{eq:fisherevolutionexpanded} cancels the terms containing
$w_{z_i}w^\top M\partial_{z_i}w$ in the formula for
$\frac{\dd}{\dd t}\cI_M(\nu_t^\eps\mid\pi_\eps)$.
Together with \eqref{eq:entropyevolutionexpanded}, this gives
\begin{subequations}\label{eq:fisherentropyderivatives}
\begin{align}
 \frac{\dd}{\dd t}\cI_M(\nu_t^\eps\mid\pi_\eps)
 &=\int_{\R^{3d}} w^\top\left(J_{\nabla^2U}M+MJ_{\nabla^2U}^\top\right)w\dd\nu_t^\eps
   -2\gamma\eps\sum_{i=1}^{d}\int_{\R^{3d}}(\partial_{z_i}w)^\top M(\partial_{z_i}w)\dd\nu_t^\eps,
   \label{eq:fisherderivative}\\
 \frac{\dd}{\dd t}\KL(\nu_t^\eps\mid\pi_\eps)
 &=-\gamma\eps\int_{\R^{3d}}|w_z|^2\dd\nu_t^\eps.
   \label{eq:entropyderivative}
\end{align}
\end{subequations}
These operations are justified first for the regular initial laws in
Lemma~\ref{lem:closure}. Multiplying the entropy identity
\eqref{eq:entropyderivative}
by $A_0/\eps$ and adding it to the weighted Fisher-information identity
\eqref{eq:fisherderivative}
gives
\begin{align}\label{eq:frozenidentity}
 \frac{\dd}{\dd t}\cH_\eps(\nu_t^\eps)
 =&-\int_{\R^{3d}} w^\top\left(-J_{\nabla^2U}M-MJ_{\nabla^2U}^\top
                         +A_0\gamma P_z\right)w\,\dd\nu_t^\eps\notag\\
 &\qquad\qquad-2\gamma\eps\sum_{i=1}^d\int_{\R^{3d}}
           (\partial_{z_i}w)^\top M(\partial_{z_i}w)\,\dd\nu_t^\eps.
\end{align}
Applying Lemma~\ref{lem:matrix} to \eqref{eq:frozenidentity} gives
\begin{equation}\label{eq:frozenfisherdissipation}
 \frac{\dd}{\dd t}\cH_\eps(\nu_t^\eps)
 \leq-\frac12\int_{\R^{3d}}|w|^2\dd\nu_t^\eps.
\end{equation}
On the other hand, \eqref{eq:lsi} and tensorization as in
\eqref{eq:tensorizedlsi} give
\begin{equation}\label{eq:functionallsicontrol}
 \cH_\eps(\nu_t^\eps)
 \leq\left(\|M\|+\frac{A_0C_{\rm LS}(\eps)}{\eps}\right)
       \int_{\R^{3d}}|w|^2\dd\nu_t^\eps.
\end{equation}
Combining \eqref{eq:frozenfisherdissipation} and
\eqref{eq:functionallsicontrol} yields
\begin{equation}\label{eq:decay}
 \frac{\dd}{\dd t}\cH_\eps(\nu_t^\eps)
 \leq-\kappa(\eps)\cH_\eps(\nu_t^\eps),\qquad
 \kappa(\eps):=\frac{1}{2(\|M\|+A_0C_{\rm LS}(\eps)/\eps)}.
\end{equation}
To record the loss in the estimate \eqref{eq:decay} explicitly, fix
$\zeta>0$ and first choose
$\zeta_{\rm LSI}>0$ sufficiently small such that
$(D+\zeta_{\rm LSI})/E<D/E+\zeta/3$. Equations
\eqref{eq:lsi} and \eqref{eq:decay} give, for small $\eps$,
\begin{align}\label{exp:as:lower:bound}
 \kappa(\eps)\geq c\eps e^{-(D+\zeta_{\rm LSI})/\eps}.
\end{align}
By the first relation in \eqref{eq:coolingassumption},
for all sufficiently large $t$,
\[
 \frac{D+\zeta_{\rm LSI}}{\eps_t}
 \leq\left(\frac DE+\frac{2\zeta}{3}\right)\log(t+t_0).
\]
Consequently, the exponential in \eqref{exp:as:lower:bound} satisfies
\[
 e^{-(D+\zeta_{\rm LSI})/\eps_t}
 \geq(t+t_0)^{-D/E-2\zeta/3}.
\]
Finally,
$\eps_t\asymp1/\log(t+t_0)\geq c_\zeta(t+t_0)^{-\zeta/3}$, and hence
\begin{equation}\label{eq:kappalower}
 \kappa(\eps_t)\geq c_\zeta(t+t_0)^{-D/E-\zeta}.
\end{equation}

\emph{Cooling contribution.}
For a fixed law $\nu$ with density $\rho$, write
$w:=\nabla\log(\rho/p_\eps)$. Differentiating
\eqref{eq:gibbsnormalizer}, with domination provided by
\eqref{eq:quadraticU}, gives
\begin{equation}\label{eq:gibbsnormalizationderivative}
 \partial_\eps\log\mathcal Z_\eps
 =\frac{1}{\eps^2\mathcal Z_\eps}
       \int_{\R^{3d}}H e^{-H/\eps}\dd y
 =\frac{\pi_\eps(H)}{\eps^2}.
\end{equation}
By \eqref{eq:gibbs} and \eqref{eq:gibbsnormalizationderivative},
$\partial_\eps\log(\rho/\pi_\eps)=-(H-\pi_\eps(H))/\eps^2$ and
$\partial_\eps w=-\nabla H/\eps^2$. Therefore,
\begin{align}\label{eq:temperatureidentity}
 \partial_\eps\cH_\eps(\nu)
 &=-\frac{2}{\eps^2}\int_{\R^{3d}} w^\top M\nabla H\,\dd\nu
 -\frac{A_0}{\eps^2}\KL(\nu\mid\pi_\eps)
 -\frac{A_0}{\eps^3}\left(\nu(H)-\pi_\eps(H)\right).
\end{align}
For a regular annealing law $\nu_t$, apply \eqref{eq:frozenidentity}
instantaneously at $\eps=\eps_t$ and add $\eps_t'$ times
\eqref{eq:temperatureidentity}. Set $F(t)=\cH_{\eps_t}(\nu_t)$. The moment
bounds \eqref{eq:moments} for $\nu_t$ and \eqref{eq:gibbsmoments} for
$\pi_{\eps_t}$, together with $|\nabla H(y)|\leq C(1+|y|)$ and the
Cauchy--Schwarz inequality, give
\begin{align*}
 \left|\int_{\R^{3d}} w^\top M\nabla H\dd\nu_t\right|
 &\leq\sqrt{\cI_M(\nu_t\mid\pi_{\eps_t})}
      \sqrt{\int_{\R^{3d}}(\nabla H)^\top M\nabla H\,\dd\nu_t}
 \leq C\sqrt{F(t)},\\
 \KL(\nu_t\mid\pi_{\eps_t})&\leq\eps_t F(t)/A_0,
 \qquad |\nu_t(H)-\pi_{\eps_t}(H)|\leq C.
\end{align*}
Since $\sqrt F\leq1+F$, \eqref{eq:temperatureidentity} and the second
condition in \eqref{eq:coolingassumption} imply
\begin{equation}\label{eq:coolingexplicit}
\begin{aligned}
 \left|\eps_t'\left.\partial_\eps\cH_\eps(\nu_t)
                      \right|_{\eps=\eps_t}\right|
 &\leq C|\eps_t'|
       \left(\eps_t^{-2}\sqrt{F(t)}+\eps_t^{-1}F(t)+\eps_t^{-3}\right)\\
 &\leq C|\eps_t'|\left(1+\eps_t^{-3}\right)(1+F(t))
 \leq C\frac{\log(t+t_0)}{t+t_0}(1+F(t)),
\end{aligned}
\end{equation}
for all sufficiently large $t$. The constant depends on the initial law
only through its second-moment bound.

Fix $T<\infty$ and let
$F_n(t):=\cH_{\eps_t}\left(\nu_t^{(n)}\right)$ for the regularized laws in
Lemma~\ref{lem:closure}. By \eqref{eq:regularizedmoments},
\[
 M_{2,T}:=\sup_{n\geq1}\sup_{0\leq t\leq T}
       \int_{\R^{3d}}|y|^2\,\nu_t^{(n)}(\dd y)<\infty.
\]
Since $\eps$ is positive and $C^1$ on $[0,T]$,
\[
 K_T:=\sup_{0\leq t\leq T}|\eps_t'|\left(1+\eps_t^{-3}\right)<\infty.
\]
Combining \eqref{eq:frozenidentity}, Lemma~\ref{lem:matrix}, and
\eqref{eq:temperatureidentity}, and using the moment bounds
\eqref{eq:regularizedmoments} and \eqref{eq:gibbsmoments}, gives, for almost
every $t\in[0,T]$ and every $n\geq1$,
\begin{equation}\label{eq:finiteintervalfunctionalode}
\begin{aligned}
 F_n'(t)
 &\leq-\frac12\cI\left(\nu_t^{(n)}\mid\pi_{\eps_t}\right)
       +\eps_t'\left.\partial_\eps\cH_\eps\left(\nu_t^{(n)}\right)
                         \right|_{\eps=\eps_t}\\
 &\leq C(1+M_{2,T})|\eps_t'|\left(1+\eps_t^{-3}\right)(1+F_n(t))\\
 &\leq C_T(1+F_n(t)),
 \qquad C_T:=C(1+M_{2,T})(1+K_T).
\end{aligned}
\end{equation}
Here $C$ depends only on the model parameters and $\eps_0$, so $C_T$ is
independent of $n$. Multiplying \eqref{eq:finiteintervalfunctionalode} by
$e^{-C_Tt}$ and integrating gives
\[
 F_n(t)+1\leq e^{C_T(t-s)}(F_n(s)+1),\qquad 0\leq s\leq t\leq T.
\]
By \eqref{eq:functionalapprox}, $F_n(0)\to F(0)<\infty$. Hence
\[
 \sup_{n\geq1}\sup_{0\leq t\leq T}F_n(t)
 \leq e^{C_TT}\left(1+\sup_{n\geq1}F_n(0)\right)-1<\infty.
\]
Applying \eqref{eq:abstractintegrated} with $a=-C_T$ and $b=C_T$ also gives
\[
 F(t)\leq e^{C_Tt}F(0)+C_T\int_0^t e^{C_T(t-u)}\dd u
       =e^{C_Tt}(1+F(0))-1<\infty,\qquad 0\leq t\leq T.
\]

For any $0<\zeta<1-D/E$, \eqref{eq:kappalower} gives
\[
 \frac{C\log(t+t_0)/(t+t_0)}{\kappa(\eps_t)}
 \leq C_\zeta\log(t+t_0)(t+t_0)^{-1+D/E+\zeta}
 \longrightarrow0.
\]
Thus $h(t):=C\log(t+t_0)/(t+t_0)\leq\kappa(\eps_t)/2$ for all sufficiently
large $t$. By \eqref{eq:weightedapprox} and
\eqref{eq:momentpowerestimate}, the initial approximations have a common
global second-moment bound. Hence \eqref{eq:decay} and
\eqref{eq:coolingexplicit} apply with the same $h$ to every $F_n$, and give
\begin{equation}\label{eq:coolingabsorption}
 F_n'(t)\leq-\kappa(\eps_t)F_n(t)+h(t)(1+F_n(t))
       \leq-\tfrac12\kappa(\eps_t)F_n(t)+h(t).
\end{equation}
Using \eqref{eq:kappalower} and
$\log(t+t_0)\leq C_\zeta(t+t_0)^\zeta$ in
\eqref{eq:coolingabsorption} gives
\begin{equation}\label{eq:ode}
 F_n'(t)\leq-c_\zeta(t+t_0)^{-D/E-\zeta}F_n(t)
              +C_\zeta(t+t_0)^{-1+\zeta},
 \qquad F_n(t)=\cH_{\eps_t}(\nu_t^{(n)}).
\end{equation}
For any $r_0<1-D/E$, choose $\zeta$ so that
$r_0<1-D/E-2\zeta$. A sufficiently large multiple of
$(t+t_0)^{-r_0}$ is a supersolution of \eqref{eq:ode}. Indeed, if
$\overline F(t)=C(t+t_0)^{-r_0}$, then
\begin{equation}\label{eq:supersolutioncalculation}
 \overline F'(t)+c_\zeta(t+t_0)^{-D/E-\zeta}\overline F(t)
 ={C}\left[-r_0(t+t_0)^{-r_0-1}
 +c_\zeta(t+t_0)^{-r_0-D/E-\zeta}\right].
\end{equation}
Because $D/E+\zeta<1$ and $r_0+D/E+\zeta<1-\zeta$, for all large $t$,
\begin{equation}\label{eq:supersolutioninequality}
 \overline F'(t)+c_\zeta(t+t_0)^{-D/E-\zeta}\overline F(t)
 \geq\frac{Cc_\zeta}{2}(t+t_0)^{-r_0-D/E-\zeta}
 \geq C_\zeta(t+t_0)^{-1+\zeta}.
\end{equation}
Choose $t_1$ sufficiently large and increase $C$ once more so that
$\sup_nF_n(t_1)\leq\overline F(t_1)$, which is possible by
\eqref{eq:finiteintervalfunctionalode} and \eqref{eq:functionalapprox}.
For $G_n=F_n-\overline F$, subtracting
\eqref{eq:supersolutioninequality} from \eqref{eq:ode}
gives
\[
 G_n'(t)+c_\zeta(t+t_0)^{-D/E-\zeta}G_n(t)\leq0,
 \qquad t\geq t_1.
\]
Multiplication by the corresponding integrating factor yields
\[
 G_n(t)\leq G_n(t_1)
 \exp\left(-\int_{t_1}^tc_\zeta(u+t_0)^{-D/E-\zeta}\dd u\right)\leq0.
\]
It follows that $F_n(t)\leq C_{r_0}(t+t_0)^{-r_0}$ uniformly in $n$.
Applying \eqref{eq:functionallowersemicontinuity} gives
\begin{equation}\label{eq:annealedfunctionaldecay}
F(t)\leq C_{r_0}(t+t_0)^{-r_0}.
\end{equation}

Finally, \eqref{eq:functional} and \eqref{eq:annealedfunctionaldecay} give
$\KL(\nu_t\mid\pi_{\eps_t})\leq\eps_tF(t)/A_0$.
Pinsker's inequality gives
\begin{equation}\label{eq:continuouspinsker}
 \mathbb P(U(X_t)>\delta)
 \leq\pi_{\eps_t}(U>\delta)
                   +\sqrt{\frac12\KL(\nu_t\mid\pi_{\eps_t})}.
\end{equation}
The Gibbs tail estimate used here is
\begin{equation}\label{eq:gibbstail}
 \pi_\eps(U>\delta)\leq C_{\delta,\zeta}e^{-(\delta-\zeta)/\eps},
 \qquad 0<\zeta<\delta.
\end{equation}
To verify \eqref{eq:gibbstail}, choose a measurable set $A_\zeta$ of positive
Lebesgue measure on which $U\leq\zeta/2$. For $0<\zeta<\delta$,
\begin{align}\label{integral:bounds}
 \int_{\{U>\delta\}}e^{-U/\eps}\dd x
 \leq e^{-(\delta-\zeta/2)/\eps}
       \int_{\R^d}e^{-\zeta U/(2\delta\eps)}\dd x,
\end{align}
and
\begin{equation}\label{eq:positionnormalizerlower}
\begin{aligned}
 Z_\eps\geq |A_\zeta|e^{-\zeta/(2\eps)}.
\end{aligned}
\end{equation}
Indeed, fix $\eps_*>0$. Since $U\geq0$ and $U$ has a quadratic lower bound,
\[
 \sup_{0<\eps\leq\eps_*}\int_{\R^d}e^{-\zeta U/(2\delta\eps)}\dd x
 \leq\int_{\R^d}e^{-\zeta U/(2\delta\eps_*)}\dd x<\infty.
\]
Dividing \eqref{integral:bounds} by \eqref{eq:positionnormalizerlower} gives
\eqref{eq:gibbstail}.
The first condition in \eqref{eq:coolingassumption} and
\eqref{eq:gibbstail} imply
\[
 \pi_{\eps_t}(U>\delta)
 \leq C_{\delta,\zeta}(t+t_0)^{-(\delta-\zeta)/E+o(1)}.
\]
Inserting this estimate and \eqref{eq:annealedfunctionaldecay} into
\eqref{eq:continuouspinsker} gives
\begin{equation}\label{eq:continuousprobabilitycomponents}
 \mathbb P(U(X_t)>\delta)
 \leq C_{\delta,\zeta}(t+t_0)^{-(\delta-\zeta)/E+o(1)}
      +C_{r_0}(t+t_0)^{-r_0/2}.
\end{equation}
Given $\alpha>0$, choose $0<\zeta<\delta$ with $\zeta/E<\alpha/3$
and $0<r_0<1-D/E$ with $(1-D/E-r_0)/2<\alpha/3$.
For all sufficiently large $t$, the $o(1)$ in
\eqref{eq:continuousprobabilitycomponents} is at most $\alpha/3$.
The two terms then give \eqref{eq:mainrate}.
\end{proof}

The discrete analysis will restart the exact process from the law of a numerical
iterate. The following proposition isolates the continuous estimate with constants
uniform over such restarting laws.

\begin{proposition}[Uniform comparison after a restart]\label{prop:restart}
Fix a bound $M_2$ on the initial second moment and an exponent-loss parameter
$0<\zeta<1-D/E$. There are $t_*,c_\zeta,C_\zeta>0$, depending on $M_2$,
$\zeta$, and the model parameters and uniform in the restarting time and initial
Fisher information, with the following property. Start the exact annealed
process at $s\geq t_*$ from any law satisfying
$\int_{\R^{3d}}|y|^2\nu_s(\dd y)\leq M_2$ and
$F(s)=\cH_{\eps_s}(\nu_s)<\infty$. Then, for
$t\geq s$,
\begin{equation}\label{eq:restartintegral}
 F(t)\leq e^{-\int_s^t a_\zeta(u)\dd u}F(s)
 +C_\zeta\int_s^t e^{-\int_v^t a_\zeta(u)\dd u}
                         (v+t_0)^{-1+\zeta}\dd v,
 \quad a_\zeta(u):=c_\zeta(u+t_0)^{-D/E-\zeta}.
\end{equation}
\end{proposition}
\begin{proof}
We first derive the second-moment bound uniformly after restarting at $s$.
The Lyapunov function $V_a$ in \eqref{eq:lyapunovdefinition}
satisfies \eqref{eq:lyapunovcomparison} with constants $c_V,C_V>0$.
Since $0<\eps_u\leq\eps_0$, the drift estimate
\eqref{eq:lyapunovdrift} holds for every $u\geq s$ with the same constants
$c,C$, independently of $s$ and $\nu_s$. Applying Dynkin's formula on
$[s,u]$, with localization, and Gr\"{o}nwall's inequality gives
\[
 \int_{\R^{3d}}V_a(y)\nu_u(\dd y)
 \leq e^{-c(u-s)}\int_{\R^{3d}}V_a(y)\nu_s(\dd y)
       +\frac Cc\left(1-e^{-c(u-s)}\right)
 \leq C_V(1+M_2)+\frac Cc,
\]
where the last inequality uses the initial second-moment bound and
$e^{-c(u-s)}\leq1$. The lower bound on $V_a$ therefore yields
\begin{equation}\label{eq:restartmoment}
 \sup_{u\geq s}\int_{\R^{3d}}|y|^2\nu_u(\dd y)
 \leq C_{\rm mom}(1+M_2)=:\overline M_2,
 \qquad C_{\rm mom}:=\frac{C_V+C/c}{c_V}.
\end{equation}
In particular, $C_{\rm mom}$ is independent of the restarting time $s$
and of the particular initial law $\nu_s$. By \eqref{eq:restartmoment},
the constant in \eqref{eq:coolingexplicit} depends only on $M_2$ and
the model parameters. We may therefore choose $t_*$, $c_\zeta$, and $C_\zeta$
uniformly over the laws in the proposition: the frozen bound
\eqref{eq:kappalower} depends only on the model and cooling schedule,
while \eqref{eq:restartmoment} makes the coefficient $h$ in
\eqref{eq:coolingabsorption} uniform. Thus, for every regular initial
density and almost every $u\geq s\geq t_*$,
\begin{equation}\label{eq:restartode}
 F'(u)+a_\zeta(u)F(u)
 \leq C_\zeta(u+t_0)^{-1+\zeta}.
\end{equation}
Set $A_s(t):=\int_s^t a_\zeta(u)\dd u$. Multiplication of
\eqref{eq:restartode} by $e^{A_s(u)}$ yields
\[
 \frac{\dd}{\dd u}\left(e^{A_s(u)}F(u)\right)
 \leq C_\zeta e^{A_s(u)}(u+t_0)^{-1+\zeta}.
\]
Integration over $[s,t]$ and division by $e^{A_s(t)}$ give
\[
 F(t)\leq e^{-A_s(t)}F(s)
 +C_\zeta\int_s^t e^{-[A_s(t)-A_s(v)]}
                   (v+t_0)^{-1+\zeta}\dd v,
\]
which is \eqref{eq:restartintegral}.

For a general initial law, positivity of $M$ in Lemma~\ref{lem:matrix}
and \eqref{eq:functional} imply
\[
 \cI(\nu_s\mid\pi_{\eps_s})
 \leq\lambda_{\min}(M)^{-1}\cI_M(\nu_s\mid\pi_{\eps_s})
 \leq\lambda_{\min}(M)^{-1}F(s)<\infty.
\]
Finite relative entropy also ensures $\nu_s\ll\pi_{\eps_s}$.
Let $\nu_s^{(n)}$ be the regular approximations
from Lemma~\ref{lem:closure}, chosen so that their second moments are bounded
by $M_2+1$. If
$F_n(u)=\cH_{\eps_u}\left(\nu_u^{(n)}\right)$, then
$F_n(s)\to F(s)$ by Lemma~\ref{lem:closure}, while
\eqref{eq:floww2} and \eqref{eq:fisherlowersemicontinuity} give
$F(t)\leq\liminf_{n\to\infty}F_n(t)$. Applying to $F_n$ the
integrating-factor estimate obtained from \eqref{eq:restartode}, namely
\eqref{eq:restartintegral} for regular initial laws,
and passing to the lower limit proves \eqref{eq:restartintegral}. The
constants remain independent of $F(s)$.
\end{proof}

\begin{remark}[Temperature-dependent degeneracy]\label{rem:degeneracy}
The noise covariance in \eqref{eq:model} is $2\gamma\eps P_z$, which has
rank $d$ and nonzero eigenvalues $2\gamma\eps$ in the $3d$-dimensional
state space. These eigenvalues vanish as $\eps\downarrow0$.
At fixed temperature, \eqref{eq:entropyderivative} gives
\[
 \frac{A_0}{\eps}\frac{\dd}{\dd t}\KL(\nu_t^\eps\mid\pi_\eps)
 =-A_0\gamma\int_{\R^{3d}}|w_z|^2\dd\nu_t^\eps.
\]
Thus the weight removes the direct temperature factor, while
\eqref{eq:decay} and Proposition~\ref{prop:restart} retain the remaining
temperature dependence needed in the discrete comparison.
\end{remark}

\section{Discrete-Time Annealing}\label{sec:discrete}
\subsection{Discretization schemes and convergence result}

We first define the annealed scheme $(X_{k},V_{k},Z_{k})_{k\geq 0}$ 
based on the numerical scheme of Mou et al. \cite{MouEtAl2021}. 
Let $\eta_k>0$ denote the step size of the $k$-th iteration. Write
$\tau_0=0$, $\tau_{k+1}=\tau_k+\eta_k$, and $s_k=\tau_k+t_0$.
Given $S_k=(X_{k},V_{k},Z_{k})=(x,v,z)$, define $g_u:=\nabla U(x+uv)$ and
$\bar g:=\eta_k^{-1}\int_0^{\eta_k}g_u\dd u$. On $0\leq u\leq \eta_k$ construct
\begin{equation}\label{eq:stages}
\begin{aligned}
 \dd\widehat Z_u&=\left(-\lambda v-\gamma\widehat Z_u\right)\dd u
                         +\sqrt{2\gamma\eps_{\tau_k}}\dd B_u,\\
 \dd\widetilde V_u&=\left(-g_u+\lambda\widehat Z_u\right)\dd u,
 &\dd\widehat V_u&=\left(-\bar g+\lambda\widehat Z_u\right)\dd u,\\
 \dd\widetilde X_u&=\widehat V_u\dd u,
 &\dd\widetilde Z_u&=\left(-\lambda\widehat V_u-\gamma\widetilde Z_u\right)\dd u
                         +\sqrt{2\gamma\eps_{\tau_k}}\dd B_u.
\end{aligned}
\end{equation}
All variables start from the corresponding component of $(x,v,z)$. The two
$Z$ equations use the same Brownian motion. Then
$S_{k+1}=\left(\widetilde X_{\eta_k},\widetilde V_{\eta_k},\widetilde Z_{\eta_k}\right)$ in distribution, 
where $S_{k+1}=\left(X_{k+1},V_{k+1},Z_{k+1}\right)$ is 
generated through the Gaussian transition
\begin{equation}\label{eq:moustep}
 S_{k+1}\mid S_k=s\sim
 \mathcal N\left(m_{\eta_k}(s),\eps_{\tau_k}\Sigma_{\eta_k}\right),
\end{equation}
whose mean and covariance we now give explicitly. For a generic step length
$\eta>0$, set
\begin{align*}
 a_0(\eta)&:=\frac{1-e^{-\gamma \eta}}{\gamma},&
 a_1(\eta)&:=\frac{\eta}{\gamma}
          -\frac{1-e^{-\gamma \eta}}{\gamma^2},\\
 a_2(\eta)&:=\frac{\eta^2}{2\gamma}-\frac{\eta}{\gamma^2}
          +\frac{1-e^{-\gamma \eta}}{\gamma^3},&
 b_0(\eta)&:=\frac{1-e^{-\gamma \eta}-\gamma \eta e^{-\gamma \eta}}{\gamma^2},\\
 b_1(\eta)&:=\frac{\eta(1+e^{-\gamma \eta})}{\gamma^2}
          -\frac{2(1-e^{-\gamma \eta})}{\gamma^3}.
\end{align*}
For $s=(x,v,z)$ and
$\Delta U=\int_0^\eta\nabla U(x+uv)\dd u$, write
$m_\eta:=\left(m_\eta^x,m_\eta^v,m_\eta^z\right)$. Direct integration of
\eqref{eq:stages} yields
\begin{align}\label{eq:explicitmean}
 m_\eta^x(s)
 &:=x-\frac \eta2\Delta U+\left(\eta-\lambda^2a_2(\eta)\right)v
      +\lambda a_1(\eta)z,\notag\\
 m_\eta^v(s)
 &:=-\Delta U+\left(1-\lambda^2a_1(\eta)\right)v
      +\lambda a_0(\eta)z,\notag\\
 m_\eta^z(s)
 &:=\left(\frac{\lambda}{\gamma}
   -\frac{\lambda(1-e^{-\gamma \eta})}{\gamma^2\eta}\right)\Delta U
   +\left(-\lambda a_0(\eta)+\lambda^3b_1(\eta)\right)v
      +\left(e^{-\gamma \eta}-\lambda^2b_0(\eta)\right)z.
\end{align}
To describe the covariance, define the column vector
\begin{equation}\label{eq:noisekernels}
\mathbf k(r):=
\begin{pmatrix}
 \displaystyle\frac{\lambda}{\gamma}
 \left(r-\frac{1-e^{-\gamma r}}{\gamma}\right)\\[5pt]
 \displaystyle\frac{\lambda}{\gamma}(1-e^{-\gamma r})\\[5pt]
 \displaystyle e^{-\gamma r}
 -\frac{\lambda^2}{\gamma^2}
  \left(1-e^{-\gamma r}-\gamma r e^{-\gamma r}\right)
\end{pmatrix}.
\end{equation}
Equivalently, $\mathbf k(r)=\mathbf u_0+r\mathbf u_1
+e^{-\gamma r}(\mathbf v_0+r\mathbf v_1)$, where
\begin{align*}
 \mathbf u_0:=\left(-\frac{\lambda}{\gamma^2},
                       \frac{\lambda}{\gamma},
                      -\frac{\lambda^2}{\gamma^2}\right)^\top,
 \quad\mathbf u_1:=\left(\frac{\lambda}{\gamma},0,0\right)^\top,
 \quad\mathbf v_0:=\left(\frac{\lambda}{\gamma^2},
                      -\frac{\lambda}{\gamma},
                       1+\frac{\lambda^2}{\gamma^2}\right)^\top,
 \quad\mathbf v_1:=\left(0,0,\frac{\lambda^2}{\gamma}\right)^\top.
\end{align*}
The stochastic terms in \eqref{eq:stages} give the endpoint representation
\begin{equation}\label{eq:endpointnoise}
 S_\eta-m_\eta(s)
 =\sqrt{2\gamma\eps_t}\int_0^\eta
       (\mathbf k(\eta-u)\otimes\Id_d)\dd B_u.
\end{equation}
Thus It\^o's isometry applied to \eqref{eq:endpointnoise} gives the covariance
in \eqref{eq:moustep}, which can be evaluated without numerical quadrature:
\begin{equation}\label{eq:explicitcovariance}
 \Sigma_\eta
 :=2\gamma\int_0^\eta\mathbf k(r)\mathbf k(r)^\top\dd r
   \otimes\Id_d
 =2\gamma K_\eta\otimes\Id_d,
\end{equation}
where
\begin{align*}
 K_\eta:=&\eta\mathbf u_0\mathbf u_0^\top
 +\frac{\eta^2}{2}\left(\mathbf u_0\mathbf u_1^\top
                      +\mathbf u_1\mathbf u_0^\top\right)
 +\frac{\eta^3}{3}\mathbf u_1\mathbf u_1^\top\\
 &+J_0(\gamma,\eta)\left(\mathbf u_0\mathbf v_0^\top
                      +\mathbf v_0\mathbf u_0^\top\right)+J_1(\gamma,\eta)\left(\mathbf u_0\mathbf v_1^\top
                      +\mathbf v_1\mathbf u_0^\top
                      +\mathbf u_1\mathbf v_0^\top
                      +\mathbf v_0\mathbf u_1^\top\right)\\
 &\qquad+J_2(\gamma,\eta)\left(\mathbf u_1\mathbf v_1^\top
                      +\mathbf v_1\mathbf u_1^\top\right)
 +J_0(2\gamma,\eta)\mathbf v_0\mathbf v_0^\top\\
 &\qquad\qquad+J_1(2\gamma,\eta)\left(\mathbf v_0\mathbf v_1^\top
                       +\mathbf v_1\mathbf v_0^\top\right)
 +J_2(2\gamma,\eta)\mathbf v_1\mathbf v_1^\top,
\end{align*}
with
\[
 J_n(q,\eta):=\int_0^\eta r^ne^{-qr}\dd r
 =\frac{n!}{q^{n+1}}
 \left(1-e^{-q\eta}\sum_{j=0}^n\frac{(q\eta)^j}{j!}\right),
\]
for $n=0,1,2$ and $q>0$.
The force enters the endpoint only through $\Delta U$. Besides this exact-integral
version, we consider the implementable midpoint version obtained by replacing
$\Delta U$ in the endpoint mean \eqref{eq:explicitmean} by
\begin{equation}\label{eq:midpoint}
 \Delta U_{\rm mid}:=\eta_k\nabla U\left(x+\frac12\eta_kv\right).
\end{equation}
These two force treatments have the same covariance and are members of the
integration family of Mou et al. \cite{MouEtAl2021}. The next theorem gives
their annealing rates and the step decay needed to preserve the
continuous-time exponent.

\begin{theorem}[Discrete annealing for two schemes of Mou et al.]\label{thm:discrete}
Assume Assumptions~\ref{ass:potential}, \ref{ass:lsi}, and \ref{ass:discrete}.
For either \eqref{eq:stages} or its midpoint endpoint version, take
\[
 \eta_k=c_\eta s_k^{-a},\qquad c_\eta>0,\quad a>0,
\]
with $c_\eta$ small enough that every step is in a fixed sufficiently small
interval $(0,\eta_*]$. Let $\vartheta=D/E$, where $D$ and $E$ are defined in
\eqref{eq:lsi} and \eqref{eq:coolingassumption}, respectively.
If $a>\vartheta/2$, then for every
$\delta,\alpha>0$ and all large $k$,
\begin{equation}\label{eq:discreterate}
 \mathbb P\left(U(X_k)>\delta\right)
 \leq C_{\delta,\alpha}s_k^{-r(\delta,a)+\alpha},\qquad
 r(\delta,a):=\min\left\{\frac\delta E,
                   \frac{1-\vartheta}{2},2a-\vartheta\right\}.
\end{equation}
In particular, $a\geq(1+\vartheta)/4$ retains the continuous-time exponent.
Moreover $s_k\asymp k^{1/(1+a)}$, so the corresponding iteration exponent is
$r(\delta,a)/(1+a)$, again up to an arbitrarily small loss.
\end{theorem}

\begin{proof}
    We provide the proof in Section~\ref{sec:proof_of_theorem_discrete}.
\end{proof}

The next corollary chooses the coarsest polynomial grid that preserves the
continuous-time rate and records the resulting iteration exponent.

\begin{corollary}[Rate-preserving iteration exponent]
\label{cor:iterationgain}
Under the assumptions of Theorem~\ref{thm:discrete}, let
\[
 r_{\rm c}(\delta):=\min\left\{\frac\delta E,
                    \frac{1-\vartheta}{2}\right\},
 \qquad a_*:=\frac{1+\vartheta}{4}.
\]
For either scheme in that theorem, the choice $a=a_*$ gives, for every
$\delta,\alpha>0$ and all sufficiently large $k$,
\[
 \mathbb P\left(U(X_k)>\delta\right)
 \leq C_{\delta,\alpha}
 k^{-4r_{\rm c}(\delta)/(5+\vartheta)+\alpha}.
\]
\end{corollary}
\begin{proof}
At $a=a_*$, the discretization exponent satisfies
$2a-\vartheta=(1-\vartheta)/2$, so \eqref{eq:discreterate} has exponent
$r_{\rm c}(\delta)$. By \eqref{eq:gridasymptotics},
$s_k\asymp k^{1/(1+a_*)}=k^{4/(5+\vartheta)}$;
choose the loss $\alpha'=(1+a_*)\alpha$ in \eqref{eq:discreterate}. Then
\begin{align*}
 \mathbb P\left(U(X_k)>\delta\right)
 &\leq C_{\delta,\alpha'}s_k^{-r_{\rm c}(\delta)+\alpha'}\\
 &\leq C_{\delta,\alpha}k^{-r_{\rm c}(\delta)/(1+a_*)+\alpha}
 =C_{\delta,\alpha}k^{-4r_{\rm c}(\delta)/(5+\vartheta)+\alpha}.
\end{align*}
\end{proof}

For comparison, Appendix~\ref{sec:ubuappendix} analyzes the one-gradient
underdamped UBU splitting in the same recent-window framework. Its leading
local stochastic error is $\mathcal{O}(\sqrt\eps\,\eta^{5/2})$ and yields the
rate-preserving sufficient condition $a\geq(1+\vartheta)/3$ and iteration
exponent $3r_{\rm c}(\delta)/(4+\vartheta)$; see
Theorem~\ref{thm:ubu} and Corollary~\ref{cor:ubuiteration}.

\begin{remark}[Step-size comparison with kinetic annealing]
\label{rem:stepsizecomparison}
Theorem~2.8 of He, Tan, and Wu \cite{HeTanWu2024} for the (kinetic) underdamped Langevin dynamics 
assumes the sufficient grid
condition
\[
 \limsup_{k\to\infty}\eta_k\sqrt{T_k}<\infty.
\]
For polynomial steps $\eta_k\asymp T_k^{-a}$, this corresponds to
$a\geq1/2$. For UBU, Theorem~\ref{thm:ubu} gives the rate-preserving
sufficient condition $a\geq(1+\vartheta)/3$ and iteration exponent
$3r_{\rm c}(\delta)/(4+\vartheta)$. For the third-order midpoint scheme
analyzed here,
Theorem~\ref{thm:discrete} requires only
$a>\vartheta/2$ for convergence and retains the continuous-time exponent once
\[
 a\geq a_*:=\frac{1+\vartheta}{4}<\frac12.
\]
Indeed, the $\mathcal{O}(\eta^3)$ local coupling bound for this scheme accumulates over the
recent comparison window as $\sum_{j=m}^{k-1}\eta_j^3$; see
\eqref{eq:telescoping}.
This produces the term $2a-\vartheta$ in
\eqref{eq:discreterate}. If
\[
 r_{\rm c}(\delta)=\min\left\{\frac\delta E,
                    \frac{1-\vartheta}{2}\right\},
\]
then the choice $a=a_*$ gives iteration exponent
$4r_{\rm c}(\delta)/(5+\vartheta)$, whereas the kinetic sufficient condition
$a=1/2$ gives $2r_{\rm c}(\delta)/3$. The former is larger because
$\vartheta<1$. He, Tan, and Wu \cite[Remark~2.9(iv)]{HeTanWu2024} also record
the sufficient overdamped Euler condition
$\limsup_k\eta_kT_k<\infty$, citing an earlier preprint by Tang and Zhou.
The published analysis of Tang, Wu, and Zhou \cite{TangWuZhou2024} uses the
same condition, corresponding to $a\geq1$ and iteration exponent
$r_{\rm c}(\delta)/2$. Thus the presently proved
sufficient-condition comparisons give
\[
 \frac{4r_{\rm c}(\delta)}{5+\vartheta}
 >\max\left\{
       \frac{3r_{\rm c}(\delta)}{4+\vartheta},
       \frac{2r_{\rm c}(\delta)}3
      \right\},
 \qquad
 \min\left\{
       \frac{3r_{\rm c}(\delta)}{4+\vartheta},
       \frac{2r_{\rm c}(\delta)}3
      \right\}
 >\frac{r_{\rm c}(\delta)}2.
\]
The UBU and frozen-force exponents cross at $\vartheta=1/2$: the UBU
exponent is larger for $\vartheta<1/2$ and smaller for $\vartheta>1/2$.
For the numerical landscape in Section~\ref{sec:numerics}, the LSI bound
\eqref{eq:toybarrierlsi} permits $D=D_{\rm crit}$ and gives
$\vartheta=D_{\rm crit}/E=0.250836/0.28\approx0.896$. Hence the UBU
coefficient $3/(4+\vartheta)\approx0.613$ is smaller than the frozen-force
coefficient $2/3$. This crossover separates the regimes favored by the two
sufficient bounds: UBU for $\vartheta<1/2$ and frozen force for
$\vartheta>1/2$. The third-order result improves both discrete sufficient
exponents, while the continuous-time exponents agree as explained in
Remark~\ref{rem:noacceleration}. UBU and the third-order midpoint
scheme each use one interior force evaluation and exact linear Gaussian
subflows. Their different bounds arise because UBU has a centered
$\mathcal{O}(\sqrt\eps\,\eta^{5/2})$ stochastic local error, while noise reaches the position
component one integration later in the third-order chain and the local endpoint
bound is $\mathcal{O}(\eta^3)$. The kinetic frozen-force scheme also uses one force evaluation
per step, whereas the cost of the exact-integral scheme depends on the chosen
quadrature.
\end{remark}

Throughout the remainder of this section, constants may depend on the
dimension, the potential, $\lambda,\gamma$, and $c_\eta$.
Assumption~\ref{ass:discrete} enters through the control estimate and midpoint
quadrature. We also record the grid asymptotics used below. The grid covers an
unbounded time horizon, since finite accumulation would keep $\eta_k$ bounded away
from zero. Thus $s_k\to\infty$, $\eta_k\to0$, and $s_{k+1}/s_k\to1$. Applying the
mean value theorem to $s_k^{1+a}$ gives, for some $\xi_k\in[s_k,s_{k+1}]$,
\begin{equation}\label{eq:gridincrements}
 s_{k+1}^{1+a}-s_k^{1+a}
 =(1+a)\xi_k^a\eta_k
 =(1+a)c_\eta(\xi_k/s_k)^a
 \longrightarrow(1+a)c_\eta.
\end{equation}
Taking the average of the increments in \eqref{eq:gridincrements} yields
\begin{equation}\label{eq:gridasymptotics}
 \frac{s_k^{1+a}}{k}
 =\frac{s_0^{1+a}}{k}
   +\frac1k\sum_{j=0}^{k-1}(s_{j+1}^{1+a}-s_j^{1+a})
 \longrightarrow(1+a)c_\eta.
\end{equation}

\subsection{Local comparison and smoothing}

For $t\geq0$ and $\eta>0$, let $Q_{t,\eta}^{\rm int}$ and $Q_{t,\eta}^{\rm mid}$
denote the one-step kernels obtained from \eqref{eq:moustep} and
\eqref{eq:midpoint}, respectively, after replacing $(\tau_k,\eta_k)$ by $(t,\eta)$.
Let
\begin{align}\label{P:t:t:h}
 P_{t,t+\eta}f(s):=\E\left[f\left(S_{t+\eta}^{t,s}\right)\right]
\end{align}
denote the transition operator of the exact time-inhomogeneous process
\eqref{eq:model} started from $s$ at time $t$.

The interpolation residual controls the one-step coupling error. A coarser
version of the same bound will preserve the Lyapunov drift and give uniform
moments for both force treatments.
\begin{lemma}[Local coupling and moments]\label{lem:local}
Let $Q_{t,\eta}\in\left\{Q_{t,\eta}^{\rm int},Q_{t,\eta}^{\rm mid}\right\}$, and let
$P_{t,t+\eta}$ be the exact transition in \eqref{P:t:t:h}. Couple the two using the same
Brownian path and start from $s\in\R^{3d}$. For $\eta\leq \eta_*$ and large $t$,
\begin{equation}\label{eq:local}
 \left\|S_\eta^{\rm num}-S_{t+\eta}^{\rm exact}\right\|_{L^2(\mathbb P_s)}
 \leq C \eta^3\left(1+|s|^2\right)
       +\frac{C \eta^{3/2}}{(t+t_0)\log^{3/2}(t+t_0)}.
\end{equation}
For the exact-integral scheme the first factor $1+|s|^2$ can be replaced by
$1+|s|$. Both schemes satisfy $\sup_{k\geq0}\E|S_k|^p<\infty$ for
$1\leq p\leq4$.
\end{lemma}

\begin{proof}
Throughout this proof,
$\|Y\|_{L^p(\mathbb P_s)}:=(\E_s|Y|^p)^{1/p}$ denotes the norm under the
synchronous coupling conditional on the initial state at $s$.
First freeze the exact temperature at $\eps_t$. For the exact-integral
interpolation \eqref{eq:stages}, elementary integrations, the Lipschitz bound
on $\nabla U$, and the Gaussian moment bounds for $\widehat Z$ give, uniformly
for $0\leq u\leq \eta$ and any fixed $p\geq2$,
\begin{equation}\label{eq:stagezsolution}
 \widehat Z_u=e^{-\gamma u}z
 -\frac{\lambda(1-e^{-\gamma u})}{\gamma}v
 +\sqrt{2\gamma\eps_t}\int_0^u e^{-\gamma(u-r)}\dd B_r,
\end{equation}
The Gaussian moments in \eqref{eq:stagezsolution} give
\begin{equation}\label{eq:stagezmoment}
\sup_{u\leq \eta}\left\|\widehat Z_u\right\|_{L^p(\mathbb P_s)}\leq C(1+|s|).
\end{equation}
Moreover,
\begin{equation}\label{eq:stagevidentity}
 |g_u|\leq |\nabla U(0)|+K(|x|+\eta|v|),\qquad
 \widehat V_u-v=-u\bar g+\lambda\int_0^u\widehat Z_r\dd r.
\end{equation}
Combining \eqref{eq:stagezmoment} and \eqref{eq:stagevidentity} gives
\begin{equation}\label{eq:stagevincrement}
 \left\|\widehat V_u-v\right\|_{L^p(\mathbb P_s)}
 \leq u|\bar g|+\lambda\int_0^u
             \|\widehat Z_r\|_{L^p(\mathbb P_s)}\dd r
 \leq Cu(1+|s|).
\end{equation}
The identities
\begin{equation}\label{eq:stagexidentity}
\widetilde X_u-x-uv=\int_0^u\left(\widehat V_r-v\right)\dd r, 
\end{equation}
and
\begin{equation}\label{eq:stagezidentity}
\widehat Z_u-\widetilde Z_u
=\lambda\int_0^u e^{-\gamma(u-r)}\left(\widehat V_r-v\right)\dd r.
\end{equation}
together with \eqref{eq:stagevincrement} give, by
\eqref{eq:stagexidentity} and \eqref{eq:stagezidentity},
\begin{align}\label{eq:stagebounds}
 \left\|\widetilde X_u-x-uv\right\|_{L^p(\mathbb P_s)}
 +\left\|\widehat Z_u-\widetilde Z_u\right\|_{L^p(\mathbb P_s)}
 &\leq(1+\lambda)\int_0^u
          \|\widehat V_r-v\|_{L^p(\mathbb P_s)}\dd r
 \leq C\eta^2(1+|s|).
\end{align}
Also,
\begin{equation}\label{eq:stageforceincrement}
 |g_r-\bar g|
 \leq\frac{K|v|}{\eta}\int_0^\eta|r-a|\dd a
 \leq K\eta|v|,
 \qquad
 \left|\widehat V_u-\widetilde V_u\right|
 =\left|\int_0^u(g_r-\bar g)\dd r\right|
 \leq K\eta^2|v|.
\end{equation}
Relative to the exact drift at
$\widetilde S_u=\left(\widetilde X_u,\widetilde V_u,\widetilde Z_u\right)$, the three
interpolation residuals form the vector
\begin{equation}\label{eq:residuals}
 \mathfrak r_u:=
 \begin{pmatrix}
 \widehat V_u-\widetilde V_u\\
 \nabla U(\widetilde X_u)-g_u+\lambda(\widehat Z_u-\widetilde Z_u)\\
 -\lambda(\widehat V_u-\widetilde V_u)
 \end{pmatrix}.
\end{equation}
Substituting $g_u=\nabla U(x+uv)$ in \eqref{eq:residuals}, and using
\eqref{eq:stagebounds} and \eqref{eq:stageforceincrement}, gives
\begin{equation}\label{eq:residualnorm}
\begin{aligned}
 \|\mathfrak r_u\|_{L^p(\mathbb P_s)}
 &\leq(1+\lambda)\|\widehat V_u-\widetilde V_u\|_{L^p(\mathbb P_s)}
       +K\|\widetilde X_u-x-uv\|_{L^p(\mathbb P_s)}
       +\lambda\|\widehat Z_u-\widetilde Z_u\|_{L^p(\mathbb P_s)}\\
 &\leq C\eta^2(1+|s|).
\end{aligned}
\end{equation}
Recall the drift of \eqref{eq:model},
$\mathfrak b(x,v,z)=(v,-\nabla U(x)+\lambda z,-\lambda v-\gamma z)$.
By the Hessian bound in Assumption~\ref{ass:potential}, for
$y=(x,v,z)$ and $y'=(x',v',z')$,
\begin{equation}\label{eq:localdriftlipschitz}
\begin{aligned}
 |\mathfrak b(y)-\mathfrak b(y')|
 &\leq K|x-x'|+(1+\lambda)|v-v'|+(\lambda+\gamma)|z-z'|\\
 &\leq C_{\mathfrak b}|y-y'|,
 \qquad C_{\mathfrak b}:=1+K+2\lambda+\gamma.
\end{aligned}
\end{equation}
Let $S_u^{\rm frozen}=\left(X_u^{\rm frozen},V_u^{\rm frozen},
Z_u^{\rm frozen}\right)$ be the solution, started from $s$ and driven by the same
Brownian path, of
\begin{equation}\label{eq:frozenprocess}
\begin{aligned}
 \dd X_u^{\rm frozen}&=V_u^{\rm frozen}\dd u,\\
 \dd V_u^{\rm frozen}&=\left(-\nabla U\left(X_u^{\rm frozen}\right)
                      +\lambda Z_u^{\rm frozen}\right)\dd u,\\
 \dd Z_u^{\rm frozen}&=\left(-\lambda V_u^{\rm frozen}
                      -\gamma Z_u^{\rm frozen}\right)\dd u
                      +\sqrt{2\gamma\eps_t}\dd B_u.
\end{aligned}
\end{equation}
Let $E_u=\widetilde S_u-S_u^{\rm frozen}$. Subtracting
\eqref{eq:frozenprocess} from the interpolation \eqref{eq:stages},
and using \eqref{eq:residuals}, gives
\begin{equation}\label{eq:frozenerrordifference}
 E_u=\int_0^u
 \left[\mathfrak b(\widetilde S_r)-\mathfrak b(S_r^{\rm frozen})
                         +\mathfrak r_r\right]\dd r.
\end{equation}
Applying Minkowski's integral inequality to
\eqref{eq:frozenerrordifference}, with \eqref{eq:localdriftlipschitz}
and \eqref{eq:residualnorm}, yields
\begin{equation}\label{eq:frozenerrorintegral}
 \|E_u\|_{L^p(\mathbb P_s)}
 \leq C\int_0^u\|E_r\|_{L^p(\mathbb P_s)}\dd r
      +C\int_0^u \eta^2(1+|s|)\dd r,
 \qquad 0\leq u\leq \eta.
\end{equation}
Applying Gr\"{o}nwall's inequality to \eqref{eq:frozenerrorintegral} gives
\begin{equation}\label{eq:frozenendpointerror}
\|E_\eta\|_{L^p(\mathbb P_s)}\leq C\eta^3(1+|s|).
\end{equation}

For midpoint quadrature,
\begin{equation}\label{eq:midpointcancellation}
 \Delta U_{\rm mid}-\Delta U
 =-\int_0^\eta\left[
       \nabla U(x+uv)-\nabla U(x+\eta v/2)
       -(u-\eta/2)\nabla^2U(x+\eta v/2)v\right]\dd u.
\end{equation}
The linear term in \eqref{eq:midpointcancellation} integrates to zero.
Taylor's formula and Assumption~\ref{ass:discrete} therefore give
\begin{equation}\label{eq:midpointquadratureerror}
 \left|\Delta U_{\rm mid}-\Delta U\right|
 \leq\frac{\|\nabla^3U\|_\infty}{2}|v|^2
       \int_0^\eta|u-\eta/2|^2\dd u
 =\frac{\|\nabla^3U\|_\infty}{24}\eta^3|v|^2.
\end{equation}
The coefficients multiplying this difference in the $(X,V,Z)$ endpoint mean
are, respectively, 
\begin{equation}\label{eq:forcecoefficients}
 -\frac \eta2,\qquad -1,\qquad
 \frac\lambda\gamma-\frac{\lambda(1-e^{-\gamma \eta})}{\gamma^2\eta}.
\end{equation}
Indeed, writing $q=\Delta U_{\rm mid}-\Delta U$, the two force treatments
give $\delta\widehat V_u=-uq/\eta$. Consequently,
\begin{align}\label{provide:1}
 \delta V_\eta=-q,\qquad
 \delta X_\eta=\int_0^\eta\delta\widehat V_u\dd u=-\frac \eta2q,
\end{align}
and
\begin{align}\label{provide:2}
 \delta Z_\eta
 =-\lambda\int_0^\eta e^{-\gamma(\eta-u)}\delta\widehat V_u\dd u
 =\frac{\lambda q}{\eta}\int_0^\eta u e^{-\gamma(\eta-u)}\dd u
 =\left(\frac\lambda\gamma
 -\frac{\lambda(1-e^{-\gamma \eta})}{\gamma^2\eta}\right)q,
\end{align}
Therefore, \eqref{provide:1} and \eqref{provide:2} provide the coefficients in \eqref{eq:forcecoefficients}.
Using \eqref{eq:midpointquadratureerror} and
\eqref{eq:forcecoefficients}, for $\eta\leq\eta_*$ we obtain
\begin{equation}\label{eq:midpointendpointerror}
 \left\|S_\eta^{\rm mid}-S_\eta^{\rm int}\right\|_{L^p(\mathbb P_s)}
 \leq C|\Delta U_{\rm mid}-\Delta U|
 \leq C\eta^3|v|^2.
\end{equation}
Combining \eqref{eq:midpointendpointerror} with
\eqref{eq:frozenendpointerror} proves the first term in
\eqref{eq:local}. For moment stability we will instead use the
Lipschitz-gradient bound derived below.

The actual temperature changes within a step. By
\eqref{eq:coolingassumption}, $\eps_t\asymp1/\log(t+t_0)$ and
\begin{equation}\label{eq:sqrttemperaturederivative}
 |(\sqrt{\eps_t})'|
 \leq C\eps_t^{5/2}\frac{\log(t+t_0)}{t+t_0}
 \leq\frac{C}{(t+t_0)\log^{3/2}(t+t_0)}.
\end{equation}
Hence the difference between the frozen and exact noise coefficients is
\[
 \sigma_{t+u}-\sigma_t
 =\sqrt{2\gamma}\left(\sqrt{\eps_{t+u}}-\sqrt{\eps_t}\right),
\]
Integrating \eqref{eq:sqrttemperaturederivative} gives
\begin{equation}\label{eq:sqrttemperatureincrement}
 \left|\sqrt{\eps_{t+u}}-\sqrt{\eps_t}\right|
 \leq\int_t^{t+u}|(\sqrt{\eps_r})'|\dd r
 \leq\frac{Cu}{(t+t_0)\log^{3/2}(t+t_0)}.
\end{equation}
The It\^o isometry and \eqref{eq:sqrttemperatureincrement} give
\begin{equation}\label{eq:temperatureito}
 \left\|\int_0^\eta
   \left(\sqrt{\eps_{t+u}}-\sqrt{\eps_t}\right)\dd B_u\right\|_{L^2}
 \leq \frac{C \eta^{3/2}}{(t+t_0)\log^{3/2}(t+t_0)}.
\end{equation}
Let $D_u=S_{t+u}^{\rm exact}-S_u^{\rm frozen}$ and, for $0\leq u\leq\eta$,
\[
 N_u:=\left(0,0,\sqrt{2\gamma}\int_0^u
        (\sqrt{\eps_{t+r}}-\sqrt{\eps_t})\dd B_r\right).
\]
Both processes start from $s$ and use the same Brownian path.
Subtracting \eqref{eq:frozenprocess} from \eqref{eq:model} on this step gives
\begin{equation}\label{eq:unfreezingdifference}
 D_u=\int_0^u
     \left[\mathfrak b\left(S_{t+r}^{\rm exact}\right)
                      -\mathfrak b\left(S_r^{\rm frozen}\right)\right]\dd r+N_u.
\end{equation}
Minkowski's integral inequality, the bound on $\mathfrak b$ in
\eqref{eq:localdriftlipschitz}, and the It\^o isometry estimate
\eqref{eq:temperatureito} applied with $u\leq\eta$ give
\begin{equation}\label{eq:unfreezingcoupling}
\begin{aligned}
 \|D_u\|_{L^2(\mathbb P_s)}
 &\leq C\int_0^u\|D_r\|_{L^2(\mathbb P_s)}\dd r+\|N_u\|_{L^2}\\
 &\leq C\int_0^u\|D_r\|_{L^2(\mathbb P_s)}\dd r
       +\frac{C\eta^{3/2}}{(t+t_0)\log^{3/2}(t+t_0)}.
\end{aligned}
\end{equation}
Applying Gr\"{o}nwall's inequality to \eqref{eq:unfreezingcoupling} gives
\[
 \|D_\eta\|_{L^2(\mathbb P_s)}
 \leq\frac{C\eta^{3/2}}{(t+t_0)\log^{3/2}(t+t_0)},
\]
which is the second term of \eqref{eq:local}.

We next establish the uniform moment bounds
$\sup_{k\geq0}\E|S_k|^p<\infty$, $1\leq p\leq4$, for both numerical schemes.
Write $\mathcal V=V_a$ for the fixed Lyapunov function
\eqref{eq:lyapunovdefinition}. Its quadratic bounds are
\eqref{eq:lyapunovcomparison}, and \eqref{eq:lyapunovpowerdrift} gives
\[
\mathcal L_\eps\mathcal V^p\leq-c_p\mathcal V^p+C_p, 
\]
uniformly for
$0<\eps\leq\eps_0$, and
\[
|\nabla\mathcal V^p(s)|\leq C_p\left(1+|s|^{2p-1}\right).
\]
Applying \eqref{eq:momentpowerestimate} to the frozen exact process
\eqref{eq:frozenprocess} gives
\begin{equation}\label{eq:frozenmomentestimate}
 \E_s\left[\mathcal V^p\left(S_\eta^{\rm frozen}\right)\right]
 \leq e^{-c_p\eta}\mathcal V^p(s)+C_p\eta.
\end{equation}
The estimate \eqref{eq:frozenendpointerror} holds for every fixed
$q\geq2$, so the exact-integral endpoint satisfies
\begin{equation}\label{eq:integralmomentlocal}
 \left\|S_\eta^{\rm int}-S_\eta^{\rm frozen}\right\|_{L^q(\mathbb P_s)}
 \leq C_q\eta^3(1+|s|),\qquad q\geq2.
\end{equation}
For moment stability of the midpoint endpoint we use the coarser estimate
\begin{equation}\label{eq:coarsemidpointquadrature}
 |\Delta U_{\rm mid}-\Delta U|
 \leq K|v|\int_0^\eta|u-\eta/2|\dd u
 \leq \frac14K\eta^2|v|.
\end{equation}
The three endpoint coefficients multiplying this force increment are bounded
by $C\eta$, $1$, and $C\eta$, respectively, for small $\eta$.
Indeed, by \eqref{eq:forcecoefficients},
$\left|\eta/2\right|\leq C\eta$, 
while Taylor's formula gives, uniformly for $0<\eta\leq \eta_*$,
\[
 \left|\frac\lambda\gamma
 -\frac{\lambda(1-e^{-\gamma \eta})}{\gamma^2\eta}\right|
 =\frac\lambda\gamma
 \left|1-\frac{1-e^{-\gamma \eta}}{\gamma \eta}\right|
 \leq C\eta.
\]
Combining \eqref{eq:integralmomentlocal},
\eqref{eq:coarsemidpointquadrature}, and \eqref{eq:forcecoefficients}
by the triangle inequality gives
\begin{equation}\label{eq:coarsemomentlocal}
 \left\|S_\eta^{\rm num}-S_\eta^{\rm frozen}\right\|_{L^q(\mathbb P_s)}
 \leq C_q\eta^2(1+|s|),
\end{equation}
for either scheme. Write the numerical endpoint as
$S_\eta^{\rm num}=m_\eta(s)+\sqrt{\eps_t}\,\Sigma_\eta^{1/2}G$,
where $G$ is a standard Gaussian vector. The mean \eqref{eq:explicitmean} and covariance \eqref{eq:explicitcovariance} give
$|m_\eta(s)|\leq C(1+|s|)$ and $\operatorname{tr}\Sigma_\eta\leq C$.
Consequently, \eqref{eq:coarsemomentlocal} implies
\begin{equation}\label{eq:thirdendpointmoments}
 \left\|S_\eta^{\rm num}\right\|_{L^q(\mathbb P_s)}
 +\left\|S_\eta^{\rm frozen}\right\|_{L^q(\mathbb P_s)}
 \leq C_q(1+|s|).
\end{equation}
We use this estimate with $q=2p$ below.
The mean value inequality gives
\[
 \left|\mathcal V^p(y)-\mathcal V^p(y')\right|
 \leq C_p\left(1+|y|^{2p-1}+|y'|^{2p-1}\right)|y-y'|.
\]
Let $A_p=1+|S_\eta^{\rm num}|^{2p-1}
                    +|S_\eta^{\rm frozen}|^{2p-1}$.
H\"older's inequality with conjugate exponents $2p/(2p-1)$ and $2p$,
followed by \eqref{eq:thirdendpointmoments} and
\eqref{eq:coarsemomentlocal}, yields, for $p=1,2$,
\begin{equation}\label{eq:thirdmomentperturbation}
\begin{aligned}
 \left|Q_{t,\eta}\mathcal V^p(s)
       -\E_s\left[\mathcal V^p\left(S_\eta^{\rm frozen}\right)\right]\right|
 &\leq C_p\|A_p\|_{L^{2p/(2p-1)}(\mathbb P_s)}
       \left\|S_\eta^{\rm num}-S_\eta^{\rm frozen}\right\|_{L^{2p}(\mathbb P_s)}\\
 &\leq C_p\left(1+|s|^{2p-1}\right)\eta^2(1+|s|)
 \leq C_p\eta^2\mathcal V^p(s).
\end{aligned}
\end{equation}
Adding \eqref{eq:thirdmomentperturbation} to
\eqref{eq:frozenmomentestimate} gives
\begin{equation}\label{eq:thirdnumericaldrift}
 Q_{t,\eta}\mathcal V^p(s)
 \leq \left(e^{-c_p\eta}+C_p\eta^2\right)\mathcal V^p(s)+C_p\eta.
\end{equation}
Since $\eta_k\to0$, the coefficient in
\eqref{eq:thirdnumericaldrift} is at most $1-c_p\eta_k/2$ for all sufficiently
large $k$. Thus, for $p=1,2$,
\begin{equation}\label{eq:discrete-lyapunov-recursion}
\E\left[\mathcal V^p(S_{k+1})\right]
 \leq\left(1-\frac{c_p}{2}\eta_k\right)\E\left[\mathcal V^p(S_k)\right]+C_p\eta_k.
\end{equation}
The initial fourth moment, the linear growth of the numerical mean, and the
Gaussian noise give, for $p=1,2$,
\[
 \E\left[\mathcal V^p(S_{k+1})\mid S_k=s\right]
 \leq C_p(1+\mathcal V^p(s)).
\]
The conditional moment bound and the finite fourth moment of the initial law
imply by induction that $Y_k:=\E[\mathcal V^p(S_k)]$ is finite for every $k$.
Choose $k_0$ so that $a_k:=c_p\eta_k/2\in[0,1]$ and
\eqref{eq:discrete-lyapunov-recursion} holds for $k\geq k_0$. Then
\[
 Y_{k+1}\leq(1-a_k)Y_k+a_k\frac{2C_p}{c_p}
 \leq\max\left\{Y_k,\frac{2C_p}{c_p}\right\},
 \qquad
 \sup_{k\geq0}Y_k
 \leq\max\left\{Y_0,\ldots,Y_{k_0},\frac{2C_p}{c_p}\right\}<\infty.
\]
Since
$\mathcal V(s)\asymp1+|s|^2$, the case $p=2$ yields
$\sup_{k\geq0}\E|S_k|^4<\infty$,
and, for $1\leq q\leq4$,
\[
 \E|S_k|^q\leq\left(\E|S_k|^4\right)^{q/4}.
\]
This proves all the remaining moment bounds.
\end{proof}

Grid-point initialization of the distorted entropy additionally requires
regularity of the numerical law. The Gaussian covariance of one step provides
this regularity and a quantitative entropy bound.

\begin{lemma}[Gaussian smoothing at numerical grid points]\label{lem:gaussian}
For the exact-integral scheme \eqref{eq:stages}--\eqref{eq:moustep} and its
midpoint version \eqref{eq:midpoint}, the numerical law at every positive grid
point has a smooth
positive density. For $\eta\leq \eta_*$,
\begin{equation}\label{eq:covbound}
 \lambda_{\min}(\Sigma_\eta)\geq c \eta^5,\qquad
 \det\Sigma_\eta\geq c \eta^{9d}.
\end{equation}
Consequently, for all sufficiently large $k$,
\begin{equation}\label{eq:gridfisher}
 \cH_{\eps_{\tau_k}}(\mathcal L(S_k))
              \leq C s_k^{5a}(\log s_k)^2.
\end{equation}
\end{lemma}
\begin{proof}
The endpoint noise kernels and their covariance are given explicitly in
\eqref{eq:noisekernels}--\eqref{eq:explicitcovariance}; they follow by
integrating successively the $\widehat Z$, $\widehat V$, $\widetilde X$, and
$\widetilde Z$ equations with the same Brownian driver.
The kernels in \eqref{eq:noisekernels} have leading terms
$(\lambda r^2/2,\lambda r,1)$, with
remainders $\mathcal{O}(r^3),\mathcal{O}(r^2),\mathcal{O}(r)$, respectively.
Let $D_\eta=\operatorname{diag}\left(\eta^{5/2},\eta^{3/2},\eta^{1/2}\right)\otimes\Id_d$.
Define
$\widetilde{\mathbf k}_\eta(u)
=(\eta^{-2}k_x(\eta u),\eta^{-1}k_v(\eta u),k_z(\eta u))^\top$.
Changing variables in \eqref{eq:explicitcovariance} gives
\begin{equation}\label{eq:scaledthirdcovariance}
 D_\eta^{-1}\Sigma_\eta D_\eta^{-1}
 =2\gamma\int_0^1\widetilde{\mathbf k}_\eta(u)
       \widetilde{\mathbf k}_\eta(u)^\top\dd u\otimes\Id_d.
\end{equation}
Substituting the expansion from \eqref{eq:noisekernels} in
\eqref{eq:scaledthirdcovariance} gives, as $\eta\downarrow0$,
\begin{equation}\label{eq:thirdcovariancelimit}
 D_\eta^{-1}\Sigma_\eta D_\eta^{-1}\longrightarrow
 2\gamma\begin{pmatrix}
 \lambda^2/20&\lambda^2/8&\lambda/6\\
 \lambda^2/8&\lambda^2/3&\lambda/2\\
 \lambda/6&\lambda/2&1
 \end{pmatrix}\otimes\Id_d.
\end{equation}
The scalar matrix in \eqref{eq:thirdcovariancelimit} is the Gram matrix of the linearly independent polynomials
$(\lambda u^2/2,\lambda u,1)$ in $L^2(0,1)$, hence is positive definite.
The scaled covariance is therefore bounded below by a fixed positive
multiple of the identity for all sufficiently small $\eta$. Consequently
$\Sigma_\eta\succeq c\operatorname{diag}(\eta^5,\eta^3,\eta)\otimes\Id_d$;
For $\eta_*\leq1$ this gives $\lambda_{\min}(\Sigma_\eta)\geq c\eta^5$.
Also,
\[
 \det\Sigma_\eta
 =(\det D_\eta)^2
   \det(D_\eta^{-1}\Sigma_\eta D_\eta^{-1})
 \geq c\eta^{9d},
\]
which proves \eqref{eq:covbound}.
The covariance is identical for midpoint quadrature. In particular, the
actual transition covariance has lower bound $c\eps_t \eta^5\Id_{3d}$, which records
the vanishing-temperature scale explicitly.
Indeed, the midpoint replacement changes only the conditional mean, so both
schemes have transition covariance
$C_{t,\eta}=\eps_t\Sigma_\eta$.
Since $\eta\leq1$,
\[
 \Sigma_\eta\succeq c\operatorname{diag}\left(\eta^5,\eta^3,\eta\right)\otimes\Id_d
 \succeq c\eta^5\Id_{3d},
\]
and therefore
\[
 C_{t,\eta}\succeq c\eps_t \eta^5\Id_{3d},\qquad
 \det C_{t,\eta}=\eps_t^{3d}\det\Sigma_\eta>0.
\]

Let $\nu_k=\mathcal L(S_k)$ and let $\rho_k$ be its Lebesgue density.
To estimate Fisher information, write $Y=S_k=A+\xi$, where, conditionally on the random mean $A$,
$\xi\sim\mathcal N(0,C_k)$. If $\mu_A$ is the law of $A$ and
$\varphi_{C_k}$ is the density of $\mathcal N(0,C_k)$, then
\begin{equation}\label{eq:mixturedensity}
 \rho_Y(y)=\int_{\R^{3d}}\varphi_{C_k}(y-a)\,\mu_A(\dd a).
\end{equation}
Differentiation under the mixture integral \eqref{eq:mixturedensity}
gives
\begin{equation}\label{eq:gaussianmixturescore}
 \nabla\log\rho_Y(y)
 =-\E\left[C_k^{-1}(Y-A)\mid Y=y\right].
\end{equation}
Applying conditional Jensen's inequality to \eqref{eq:gaussianmixturescore} yields
\begin{equation}\label{eq:gaussianrawfisher}
\begin{aligned}
 \cI^{\rm raw}(\rho_Y)
 =\E\left|\E\left[C_k^{-1}\xi\mid Y\right]\right|^2
 \leq\E\left|C_k^{-1}\xi\right|^2
 =\operatorname{tr}\left(C_k^{-1}\right).
\end{aligned}
\end{equation}
Here
$C_k=\eps_{\tau_{k-1}}\Sigma_{\eta_{k-1}}$. Applying
\eqref{eq:rawrelativefisherbound} and \eqref{eq:covbound} to
\eqref{eq:gaussianrawfisher} gives
\begin{equation}\label{eq:gridrelativefisher}
 \cI\left(\nu_k\mid\pi_{\eps_{\tau_k}}\right)
 \leq \frac{C}{\eps_{\tau_{k-1}}\eta_{k-1}^5}
       +\frac{C}{\eps_{\tau_k}^2}\E\left(1+|S_k|^2\right).
\end{equation}
Convexity of $\Psi(u):=u\log u$ (equivalently, concavity of differential
entropy; see \cite[Section~8.6]{CoverThomas2006}) bounds the negative
differential entropy of the mixture by that of the Gaussian component.
Indeed, Jensen's inequality, Fubini's theorem, and translation invariance give
\begin{equation}\label{eq:gaussianmixtureentropy}
\begin{aligned}
 \int_{\R^{3d}}\Psi(\rho_k(y))\,\dd y
 &\leq\int_{\R^{3d}}\int_{\R^{3d}}
       \Psi\!\left(\varphi_{C_k}(y-a)\right)\,\mu_A(\dd a)\dd y\\
 &=\int_{\R^{3d}}\varphi_{C_k}(z)\log\varphi_{C_k}(z)\,\dd z\\
 &=-\frac{3d}{2}\log(2\pi e)-\frac12\log\det C_k.
\end{aligned}
\end{equation}
To check that the integral in \eqref{eq:gaussianmixtureentropy} is finite
from below, let $\sigma>0$. Nonnegativity of relative entropy with respect to
$\mathcal N(0,\sigma^2\Id_{3d})$ also gives
\begin{equation}\label{eq:differentialentropylower}
 \int_{\R^{3d}}\rho_k(y)\log\rho_k(y)\,\dd y
 \geq-\frac{3d}{2}\log\left(2\pi\sigma^2\right)
      -\frac{1}{2\sigma^2}\E|S_k|^2> -\infty.
\end{equation}
Together, \eqref{eq:gaussianmixtureentropy} and
\eqref{eq:differentialentropylower} bound the differential-entropy integral
on both sides. Moreover, $\log\mathcal Z_\eps$ is bounded above for $\eps\leq\eps_0$, because
$H\geq0$ and $\mathcal Z_{\eps_0}<\infty$.
Since, for $\nu=\rho\,\dd y$,
\begin{equation}\label{eq:klgibbsdecomposition}
 \KL(\nu\mid\pi_\eps)
 =\int_{\R^{3d}}\rho(y)\log\rho(y)\,\dd y
  +\frac{\nu(H)}{\eps}+\log\mathcal Z_\eps,
\end{equation}
the formulas \eqref{eq:klgibbsdecomposition},
\eqref{eq:gaussianmixtureentropy}, and \eqref{eq:gridrelativefisher},
with \eqref{eq:covbound}, Lemma~\ref{lem:local}, and
$\eta_{k-1}\asymp s_k^{-a}$ from \eqref{eq:gridasymptotics}, give
\begin{equation}\label{eq:gridfunctionalcomponents}
\begin{aligned}
 -\log\det C_k&\leq C\log s_k,
 &\frac{\nu_k(H)}{\eps_{\tau_k}}&\leq C\log s_k,\\
 \KL\left(\nu_k\mid\pi_{\eps_{\tau_k}}\right)&\leq C\log s_k,
 &\cI\left(\nu_k\mid\pi_{\eps_{\tau_k}}\right)
   &\leq Cs_k^{5a}(\log s_k)^2.
\end{aligned}
\end{equation}
Since $\cI_M\leq\|M\|\cI$, substituting
\eqref{eq:gridfunctionalcomponents} in \eqref{eq:functional} proves
\eqref{eq:gridfisher}. Smoothness and positivity
follow from convolution with the nondegenerate Gaussian, while the required
second and fourth moments follow from Lemma~\ref{lem:local}.
\end{proof}

The next lemma controls how a local endpoint perturbation changes the
terminal event probability. Its explicit temperature factor will be used
when summing the one-step errors over the comparison window.

\begin{lemma}[One-unit smoothing by a finite controlled coupling]\label{lem:smoothing}
Assume $\|\nabla^3U\|_\infty<\infty$. For the exact inhomogeneous transition
operator, any bounded measurable $f$, and any $s,s'\in\R^{3d}$,
\begin{equation}\label{eq:smoothing}
 |P_{t,t+1}f(s)-P_{t,t+1}f(s')|
 \leq C\eps_{t+1}^{-1/2}(1+|s|+|s'|)|s-s'|\|f\|_\infty.
\end{equation}
For $T\geq t+1$,
\[
 |P_{t,T}f(s)-P_{t,T}f(s')|
 \leq C\eps_{t+1}^{-1/2}(1+|s|+|s'|)|s-s'|\|f\|_\infty.
\]
For $0\leq T-t\leq1$ and globally Lipschitz $f$,
\[
\operatorname{Lip}(P_{t,T}f)\leq e^C\operatorname{Lip}(f).
\]
The constants are independent of $t,T,s,s'$.
\end{lemma}

\begin{proof}
We prove a finite-difference estimate directly at the level of bounded
measurable test functions. On a unit interval let $(X_u,V_u,Z_u)$ start from
$s=(x,v,z)$, and write $e=s'-s=(e_x,e_v,e_z)$. Let $q$ be the vector
polynomial of degree at most five determined by
\begin{align*}
 q(0)&=e_x,&q'(0)&=e_v,&
 q''(0)&=-\nabla U(x+e_x)+\nabla U(x)+\lambda e_z,\\
 q(1)&=0,&q'(1)&=0,&q''(1)&=0.
\end{align*}
Write $a_e=-\nabla U(x+e_x)+\nabla U(x)+\lambda e_z$. The polynomial is
\begin{equation}\label{eq:controlpolynomial}
\begin{aligned}
 q(u)&=p_0(u)e_x+p_1(u)e_v+p_2(u)a_e,\\
 p_0(u)&=1-10u^3+15u^4-6u^5,\\
 p_1(u)&=u-6u^3+8u^4-3u^5,\\
 p_2(u)&=\tfrac12u^2-\tfrac32u^3+\tfrac32u^4-\tfrac12u^5.
\end{aligned}
\end{equation}
Since $|a_e|\leq K|e_x|+\lambda|e_z|$, the fixed polynomial coefficients in
\eqref{eq:controlpolynomial} give
\[
 \max_{0\leq j\leq3}\|q^{(j)}\|_\infty
 \leq C(|e_x|+|e_v|+|a_e|)\leq C|e|.
\]
Define the adapted shifted path
\[
 X_u^e:=X_u+q(u),\qquad V_u^e:=V_u+q'(u),\qquad
 Z_u^e:=Z_u+\lambda^{-1}
       \left(q''(u)+\nabla U(X_u+q(u))-\nabla U(X_u)\right).
\]
It starts from $s'$ and has exactly the same endpoint as the unshifted path.
Its first two equations agree with \eqref{eq:model}; its last equation has
an additional drift
\begin{align}\label{eq:control}
 c_e(u)=&\lambda^{-1}\left(q'''
       +\left[\nabla^2U(X_u+q)-\nabla^2U(X_u)\right]V_u
       +\nabla^2U(X_u+q)q'\right)\notag\\
 &+\lambda q'+\frac\gamma\lambda
          \left(q''+\nabla U(X_u+q)-\nabla U(X_u)\right).
\end{align}
Since $X_u$ has finite variation, differentiating $\nabla U(X_u)$ produces
only its ordinary chain-rule term. The bound
$\|\nabla^3U\|_\infty<\infty$
gives
\begin{equation}\label{eq:finitecontrolbound}
 |c_e(u)|\leq C|e|(1+|V_u|).
\end{equation}

We check the change of measure locally in $e$, with an explicit temperature
dependence. Let
$N_u:=\int_0^u\sqrt{2\gamma\eps_{t+r}}\dd B_r$. Pathwise Gr\"{o}nwall's inequality and
the uniform upper bound on the temperature give
\begin{equation}\label{eq:smoothingpathbound}
 \sup_{u\leq1}|S_{t+u}|\leq C\left(1+|s|+\sup_{u\leq1}|N_u|\right),
\end{equation}
and
\begin{equation}\label{eq:smoothingexponentialmoment}
 \E_s \left[e^{c\sup_{u\leq1}|S_{t+u}|^2}\right]\leq C e^{C|s|^2},
\end{equation}
for some $c>0$ independent of $t,s$. To verify the exponential moment,
each coordinate of $N$ is a Brownian motion with deterministic time change
bounded by $2\gamma\eps_0$. The reflection principle and the Gaussian tail
bound therefore give
\begin{equation}\label{eq:smoothingnoisetail}
 \mathbb P\left(\sup_{u\leq1}|N_u|>R\right)
 \leq4d\exp\left(-\frac{R^2}{4\gamma d\eps_0}\right).
\end{equation}
Integrating \eqref{eq:smoothingnoisetail} gives
$\E\left[\exp\left(c_N\sup_{u\leq1}|N_u|^2\right)\right]<\infty$ for sufficiently small $c_N>0$;
substituting in \eqref{eq:smoothingpathbound} proves
\eqref{eq:smoothingexponentialmoment}.
Let $b_e(u):=c_e(u)/\sqrt{2\gamma\eps_{t+u}}$. By
\eqref{eq:finitecontrolbound},
\begin{equation}\label{eq:controlenergybound}
 \frac12\int_0^1|b_e(u)|^2\dd u
 \leq\frac{C|e|^2}{\eps_{t+1}}
       \left(1+\sup_{u\leq1}|S_{t+u}|^2\right).
\end{equation}
Choose $c_*>0$ so that $Cc_*^2<c$. The estimates
\eqref{eq:smoothingexponentialmoment} and \eqref{eq:controlenergybound}
give Novikov's condition whenever $|e|\leq c_*\sqrt{\eps_{t+1}}$.
The allowed constant $c_*$ is independent of $s,t$; the exponential
expectation is finite for each $s$ and is bounded by $Ce^{C|s|^2}$.
Consequently,
\begin{equation}\label{eq:controlleddensity}
 R_e:=\exp\left(-\int_0^1 b_e\cdot\dd B
                         -\frac12\int_0^1|b_e|^2\dd u\right)
\end{equation}
has expectation one and defines an equivalent measure
$\dd\mathbb Q_e=R_e\dd\mathbb P$. Under $\mathbb Q_e$,
$B_u+\int_0^u b_e(r)\dd r$ is a standard $d$-dimensional Brownian motion, so that the shifted path solves the
uncontrolled SDE starting at $s'$. The drift
\[
 \mathfrak b(x,v,z)=(v,-\nabla U(x)+\lambda z,-\lambda v-\gamma z)
\]
is globally Lipschitz because $\nabla^2U$ is bounded; this ensures pathwise
uniqueness and hence uniqueness in law.
Endpoint equality and the change of measure \eqref{eq:controlleddensity}
therefore yield, for every bounded
measurable $f$,
\begin{equation}\label{eq:controlledendpointidentity}
 P_{t,t+1}f(s')=\E_s[R_e f(S_{t+1})].
\end{equation}
Use the relative entropy in the direction $\mathbb P\mid\mathbb Q_e$:
the stochastic integral in $-\log R_e$ has zero mean, and hence
\begin{equation}\label{eq:pathentropy}
 \KL(\mathbb P\mid\mathbb Q_e)
 =\frac12\E_s\left[\int_0^1|b_e(u)|^2\dd u\right]
 \leq\frac{C|e|^2}{\eps_{t+1}}(1+|s|^2).
\end{equation}
Applying Pinsker's inequality to \eqref{eq:controlledendpointidentity}
and using \eqref{eq:pathentropy} gives
\begin{equation}\label{eq:localsmoothingpinsker}
 \left|P_{t,t+1}f(s')-P_{t,t+1}f(s)\right|
 \leq\|f\|_\infty\sqrt{2\KL(\mathbb P\mid\mathbb Q_e)}
 \leq C\eps_{t+1}^{-1/2}|e|(1+|s|)\|f\|_\infty
\end{equation}
for these small increments. For arbitrary $s,s'$, partition their straight
line segment into increments of length at most
$c_*\sqrt{\eps_{t+1}}$. Apply \eqref{eq:localsmoothingpinsker} at each point and sum;
all intermediate norms are at most $|s|+|s'|$.
More precisely, take
\[
 N=\max\left\{1,
 \left\lceil\frac{|s'-s|}{c_*\sqrt{\eps_{t+1}}}\right\rceil\right\},
 \qquad s_i=s+\frac{i}{N}(s'-s),\quad 0\leq i\leq N.
\]
Then $|s_{i+1}-s_i|\leq c_*\sqrt{\eps_{t+1}}$,
$|s_i|\leq |s|+|s'|$, and
\begin{align*}
 \left|P_{t,t+1}f(s')-P_{t,t+1}f(s)\right|
 &\leq\sum_{i=0}^{N-1}
   \left|P_{t,t+1}f(s_{i+1})-P_{t,t+1}f(s_i)\right|\\
 &\leq C\eps_{t+1}^{-1/2}(1+|s|+|s'|)
       \sum_{i=0}^{N-1}|s_{i+1}-s_i|\|f\|_\infty\\
 &=C\eps_{t+1}^{-1/2}(1+|s|+|s'|)|s'-s|\|f\|_\infty.
\end{align*}
This proves
\eqref{eq:smoothing} globally for all $t$, with the factor
$\eps_{t+1}^{-1/2}$ quantifying the temperature dependence.

For $T\geq t+1$, the Chapman--Kolmogorov identity gives
\[
 P_{t,T}f=P_{t,t+1}(P_{t+1,T}f).
\]
Apply \eqref{eq:smoothing} to $P_{t+1,T}f$ and use
\[
 \|P_{t+1,T}f\|_\infty\leq\|f\|_\infty.
\]
For $T-t\leq1$, synchronous coupling and the uniform Lipschitz drift bound
give 
\[
\left|S_T^s-S_T^{s'}\right|\leq e^C\left|s-s'\right|. 
\]
Consequently,
\[
 |P_{t,T}f(s)-P_{t,T}f(s')|
 \leq\E|f(S_T^s)-f(S_T^{s'})|
 \leq e^C\operatorname{Lip}(f)|s-s'|,
\]
which proves the last assertion.
\end{proof}

We now combine the three preceding lemmas with the uniform restart estimate to complete
the proof of Theorem~\ref{thm:discrete}.

\subsection{Proof of Theorem~\ref{thm:discrete}}\label{sec:proof_of_theorem_discrete}
\begin{proof}[Proof of Theorem~\ref{thm:discrete}]
Fix $T=\tau_k$, and choose a number
$\beta\in(\vartheta,\min\{1,2a\})$, to be taken arbitrarily close to
$\vartheta$. This interval is nonempty because $a>\vartheta/2$.
Choose $m$ so that $\tau_m\leq T-T^\beta<\tau_{m+1}$. The choice of $m$ gives
\begin{equation}\label{eq:comparisonwindow}
 T^\beta\leq T-\tau_m<T^\beta+\eta_m,\qquad
 s_j\asymp T,\quad \eta_j\asymp T^{-a}\quad(m\leq j\leq k),
\end{equation}
where the last two relations use \eqref{eq:gridasymptotics} and $\beta<1$.
In particular, $T-\tau_m\asymp T^\beta$.
Restart the exact annealed process at time $\tau_m$ with law $\nu_m=\mathcal L(S_m)$.
The grid Fisher bound \eqref{eq:gridfisher} and the moment conclusion
of Lemma~\ref{lem:local} give
\begin{equation}\label{eq:restartgridinitial}
 \cH_{\eps_{\tau_m}}(\nu_m)\leq CT^{5a}(\log T)^2,
 \qquad \sup_m\int_{\R^{3d}}|y|^2\,\nu_m(\dd y)\leq M_2.
\end{equation}
By \eqref{eq:restartgridinitial}, Proposition~\ref{prop:restart} applies
with a common $M_2$ and common constants for every $m$. It gives, for small
$\zeta>0$ with $\vartheta+\zeta<\beta$,
\begin{equation}\label{eq:restart}
 \cH_{\eps_T}(\nu_mP_{\tau_m,T})
 \leq C T^{5a}(\log T)^2
             e^{-cT^{\beta-\vartheta-\zeta}}
       +C_\zeta T^{-1+\vartheta+2\zeta}.
\end{equation}
Let $b_\zeta(u)=C_\zeta(u+t_0)^{-1+\zeta}$.
The coefficients in \eqref{eq:restartintegral} satisfy, on this interval,
\[
 a_\zeta(u)\geq c_\zeta T^{-\vartheta-\zeta},
 \qquad b_\zeta(u)\leq C_\zeta T^{-1+\zeta}.
\]
In particular,
\[
 \int_{\tau_m}^T a_\zeta(u)\dd u
 \geq c_\zeta T^{-\vartheta-\zeta}(T-\tau_m)
 \geq c_\zeta T^{\beta-\vartheta-\zeta}.
\]
The source integral in \eqref{eq:restartintegral} satisfies
\[
 \int_{\tau_m}^T
 e^{-\int_u^T a_\zeta(r)\dd r}b_\zeta(u)\dd u
 \leq C_\zeta T^{-1+\zeta}\int_0^\infty
          e^{-c_\zeta T^{-\vartheta-\zeta}u}\dd u
 \leq C_\zeta T^{-1+\vartheta+2\zeta}.
\]
This proves \eqref{eq:restart}, uniformly in the restarting numerical law.

Fix $0<\zeta_1<\delta$ and choose a smooth nondecreasing
$\chi_{\delta,\zeta_1}:\R\to[0,1]$ equal to zero on
$(-\infty,\delta-\zeta_1]$ and to one on $[\delta,\infty)$.
Define $f(x,v,z)=\chi_{\delta,\zeta_1}(U(x))$. Then
\begin{equation}\label{eq:eventapproximation}
 \mathbf1_{\{U(x)>\delta\}}\leq f(x,v,z)
 \leq\mathbf1_{\{U(x)>\delta-\zeta_1\}},\qquad
 \|\nabla f\|_\infty
 \leq\|\chi_{\delta,\zeta_1}'\|_\infty
       \sup_{\delta-\zeta_1\leq U(x)\leq\delta}|\nabla U(x)|<\infty.
\end{equation}
The last inequality follows from the compactness of the position sublevel
set in \eqref{eq:quadraticU}. For all $j$ in the window,
Lemma~\ref{lem:smoothing} shows that
$g_j=P_{\tau_{j+1},T}f$ satisfies
\begin{equation}\label{eq:weightedlip}
 |g_j(y)-g_j(z)|
 \leq C_{\delta,\zeta_1}\sqrt{\log T}\,
              \left(1+|y|+|z|\right)|y-z|.
\end{equation}
This also holds for the final intervals of length less than one, using the
short-time part of that lemma instead of its smoothing part.

The definitions in Lemma~\ref{lem:local} give
$\nu_{j+1}=\nu_jQ_{\tau_j,\eta_j}$, while the Chapman--Kolmogorov identity for
the exact process \eqref{eq:model} gives
$P_{\tau_j,T}=P_{\tau_j,\tau_{j+1}}P_{\tau_{j+1},T}$. Telescoping these
numerical and exact transitions gives the exact identity
\begin{equation}\label{eq:telescopingidentity}
 \nu_k f-\nu_mP_{\tau_m,T}f
 =\sum_{j=m}^{k-1}\nu_j
           \left(Q_{\tau_j,\eta_j}-P_{\tau_j,\tau_{j+1}}\right)P_{\tau_{j+1},T}f.
\end{equation}
Write $\left(S_j^{\rm num},S_j^{\rm ex}\right)$ for the conditional endpoint coupling
from Lemma~\ref{lem:local} given $S_j=s$. Its two marginals satisfy
\begin{equation}\label{eq:coupledendpointmoment}
 \E_s\left[\left(1+\left|S_j^{\rm num}\right|+\left|S_j^{\rm ex}\right|\right)^2\right]
 \leq C\left(1+|s|^2\right).
\end{equation}
Combining \eqref{eq:weightedlip}, \eqref{eq:coupledendpointmoment},
and \eqref{eq:local} by the Cauchy--Schwarz inequality gives
\begin{equation}\label{eq:thirdkernelerror}
 \left|\left(Q_{\tau_j,\eta_j}-P_{\tau_j,\tau_{j+1}}\right)g_j(s)\right|
 \leq C\sqrt{\log T}(1+|s|)
       \left(\eta_j^3\left(1+|s|^2\right)
                    +\frac{\eta_j^{3/2}}{T\log^{3/2}T}\right).
\end{equation}
Integrating \eqref{eq:thirdkernelerror} against $\nu_j$ uses the uniform
third moment from Lemma~\ref{lem:local}. Substitution in
\eqref{eq:telescopingidentity} then bounds the accumulated error.
Finally $\sum_{j=m}^{k-1}\eta_j=T-\tau_m\asymp T^\beta$ and
$\eta_j\asymp T^{-a}$ throughout the window. For every fixed $r>1$,
\begin{equation}\label{eq:windowsums}
 \sum_{j=m}^{k-1}\eta_j^r
 \leq C_rT^{-a(r-1)}\sum_{j=m}^{k-1}\eta_j
 \leq C_rT^{\beta-a(r-1)}.
\end{equation}
Applying \eqref{eq:windowsums} with $r=3$ and $r=3/2$ gives
\begin{align}\label{eq:telescoping}
 |\nu_k f-\nu_mP_{\tau_m,T}f|
 &\leq C_{\delta,\zeta_1}\sqrt{\log T}
       \sum_{j=m}^{k-1}\left(\eta_j^3+
                      \frac{\eta_j^{3/2}}{T\log^{3/2}T}\right)\notag\\
 &\leq C_{\delta,\zeta_1}\left(
       \sqrt{\log T}\cdot T^{\beta-2a}
              +\frac{T^{\beta-1-a/2}}{\log T}\right).
\end{align}
This calculation retains the temperature variation within each numerical step.

We next bound $\nu_mP_{\tau_m,T}f$. Since
$0\leq f(x,v,z)\leq\mathbf 1_{\{U(x)>\delta-\zeta_1\}}$, Pinsker's inequality and
the bound
$\KL(\nu_mP_{\tau_m,T}\mid\pi_{\eps_T})
\leq\eps_T\cH_{\eps_T}(\nu_mP_{\tau_m,T})/A_0$, which follows from
\eqref{eq:functional} and $\cI_M\geq0$, give
\begin{equation}\label{eq:restartpinsker}
 \nu_mP_{\tau_m,T}f
 \leq\pi_{\eps_T}(U>\delta-\zeta_1)
      +\sqrt{\frac{\eps_T}{2A_0}\cH_{\eps_T}(\nu_mP_{\tau_m,T})}.
\end{equation}
For $0<\zeta_2<\delta-\zeta_1$, the Gibbs tail estimate \eqref{eq:gibbstail},
applied with $(\delta-\zeta_1,\zeta_2/2)$ in place of $(\delta,\zeta)$, and the
first relation in \eqref{eq:coolingassumption} give, for all large $T$,
\begin{equation}\label{eq:restartgibbstail}
 \pi_{\eps_T}(U>\delta-\zeta_1)\leq CT^{-(\delta-\zeta_1-\zeta_2)/E}.
\end{equation}
For the second term in \eqref{eq:restartpinsker}, we insert
\eqref{eq:restart} and use
$\sqrt{x+y}\leq\sqrt x+\sqrt y$ and $\eps_T\leq\eps_0$:
\begin{equation}\label{eq:stretchedexponential}
 \sqrt{\frac{\eps_T}{2A_0}\cH_{\eps_T}(\nu_mP_{\tau_m,T})}
 \leq CT^{5a/2}\log T\,e^{-cT^{\beta-\vartheta-\zeta}/2}
      +C_\zeta T^{-(1-\vartheta)/2+\zeta}.
\end{equation}
The first term on the right-hand side of \eqref{eq:stretchedexponential} is
the stretched-exponential term; it comes from the square root of the first
term on the right-hand side of \eqref{eq:restart}.
Since $\beta-\vartheta-\zeta>0$, for
every fixed $N$ there is $C_N<\infty$ such that
\[
 CT^{5a/2}\log T\,e^{-cT^{\beta-\vartheta-\zeta}/2}\leq C_NT^{-N},
 \qquad T\geq1.
\]
Consequently,
\begin{equation}\label{eq:restartprobability}
 \nu_mP_{\tau_m,T}f
 \leq C T^{-(\delta-\zeta_1-\zeta_2)/E}
       +C_\zeta T^{-(1-\vartheta)/2+\zeta}
       +C_NT^{-N}.
\end{equation}
By \eqref{eq:eventapproximation},
$\mathbb P(U(X_k)>\delta)\leq\nu_kf$. Combining
\eqref{eq:restartprobability} and \eqref{eq:telescoping} therefore gives
\begin{align}\label{eq:fourratebound}
 \mathbb P(U(X_k)>\delta)
 \leq C\bigg(&T^{-(\delta-\zeta_1-\zeta_2)/E}
 +T^{-(1-\vartheta)/2+\zeta}
 +\sqrt{\log T}\,T^{-(2a-\beta)}\notag\\
 &+(\log T)^{-1}T^{-(1+a/2-\beta)}+T^{-N}\bigg).
\end{align}
The four polynomial exponents in \eqref{eq:fourratebound} before the
parameter limits
are
\[
 \frac{\delta-\zeta_1-\zeta_2}{E},\qquad
 \frac{1-\vartheta}{2}-\zeta,\qquad
 2a-\beta,\qquad 1+\frac a2-\beta.
\]
Letting $\beta\downarrow\vartheta$ and
$\zeta_1,\zeta_2,\zeta\downarrow0$ shows that the first three converge to the
three exponents in \eqref{eq:discreterate}. The fourth does not restrict the
rate because
\[
 1+\frac a2-\vartheta-\frac{1-\vartheta}{2}
 =\frac{1+a-\vartheta}{2}>0.
\]
More explicitly, for a given $\alpha>0$, choose
\[
 0<\beta-\vartheta<\min\{\alpha/4,(1-\vartheta)/4,
                         (2a-\vartheta)/2\},\qquad
 0<\zeta<\min\{(\beta-\vartheta)/2,\alpha/4\},
 \qquad \frac{\zeta_1+\zeta_2}{E}<\alpha/4.
\]
For large $T$, $\sqrt{\log T}\leq T^{\alpha/4}$.
If $r=r(\delta,a)$ in \eqref{eq:discreterate}, each polynomial term
in \eqref{eq:fourratebound} is then at most $CT^{-r+\alpha}$;
the fourth exponent remains above $(1-\vartheta)/2$
because $\beta-\vartheta<(1-\vartheta)/4$.
Taking $N>r$ proves \eqref{eq:discreterate}.
The condition $a>\vartheta/2$ makes every relevant exponent positive.
The iteration statement follows from \eqref{eq:gridasymptotics}.
\end{proof}

\section{Numerical Illustration}\label{sec:numerics}

In this section, we conduct numerical experiments to compare overdamped Langevin dynamics (OLD) \eqref{eqn:OLD}, 
underdamped Langevin dynamics (ULD) \eqref{eqn:ULD}, 
together with the third-order Langevin
dynamic \eqref{eq:model}. Thus every frozen process has position marginal
proportional to $e^{-U/\eps}$. We fix OLD mobility and the kinetic mass at one;
this choice defines the reference physical clock, and any other constant OLD
mobility corresponds to a rescaling of that clock. The remaining kinetic
parameters are selected on validation paths and then held fixed on independent
evaluation paths, since finite-time annealing performance can change substantially
with friction.
The fixed-temperature distinction is already provided by the sharp
Eyring--Kramers theory: Wang and Zhu \cite{WangZhu2026} show that, in an explicit
matched-parameter regime, the third-order transition-time prefactor is strictly
smaller than its underdamped counterpart. We use this fixed-temperature result as
the theoretical benchmark and focus the experiment on whether the advantage
persists under the time-inhomogeneous cooling schedule proved in this
paper. 

For OLD we use Euler--Maruyama. For ULD we use the UBU splitting
\eqref{eq:ubustep}, with friction parameter $\gamma_{\rm u}$ as in
Appendix~\ref{sec:ubuappendix} and step temperature
$\eps_k=\eps_{\tau_k}$. We choose UBU to align the numerical
treatment of the two kinetic dynamics: like the midpoint third-order scheme,
it uses one interior gradient evaluation per step and exactly resolves its
linear Gaussian subdynamics. This controls for gradient budget and force
quadrature, so the comparison more directly reflects the underlying dynamics.
For third-order Langevin dynamics, we use Algorithm~1 of Mou et al.\
\cite{MouEtAl2021}, namely the Gaussian endpoint \eqref{eq:moustep} with the
midpoint force integral \eqref{eq:midpoint}. The strong-order-two local
structure of UBU is discussed in Lemma~\ref{lem:ubulocal}; see
\cite{SanzSernaZygalakis2021,PaulinWhalley2024}.

We tune the parameters over the following validation grids:
\[
 \gamma_{\rm u}\in\{0.1,0.25,0.5,1,2,4\},\qquad
 \lambda\in\{0.25,0.5,1,1.5,2\},\qquad
 \gamma\in\{0.25,0.5,1,2,4,8\}.
\]
The score is terminal success at physical time $300$, with ever-reached success
and restricted mean hitting time as tie-breakers. Figure~\ref{fig:tuning} records
the scores on $1{,}200$ validation paths; the evaluation uses independent paths.

\begin{figure}[H]
\centering
\includegraphics[width=\textwidth]{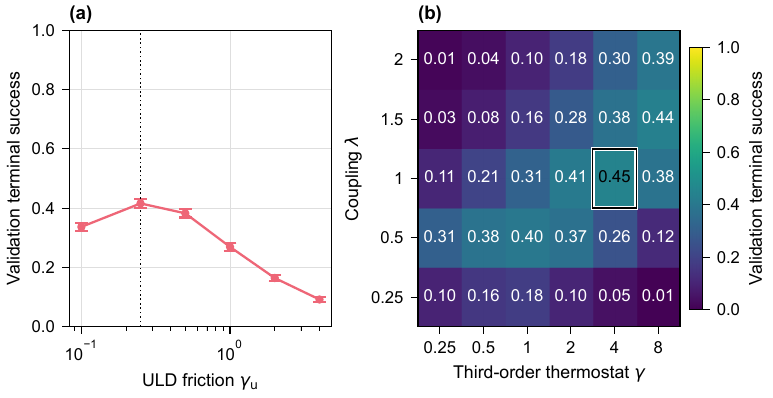}
\caption{Annealing parameter tuning on the validation set. (a) UBU terminal
success as a function of ULD friction. (b) The third-order product grid, with
terminal-success probabilities rounded to two decimals; the outlined cell marks
the selected pair. Evaluation uses independent paths.}
\label{fig:tuning}
\end{figure}

\subsection{Annealing on a double-well potential}\label{sec:toy:example}

To test annealing under decreasing temperature, we consider a toy example in this section
by using the smooth confining
tilted double well
\begin{equation}\label{eq:numericalpotential}
 U(x):=\frac{x^2}{20}-0.65e^{-(x+2)^2}-0.45e^{-(x-2)^2}.
\end{equation}
To match the normalization $\min U=0$ in
Assumption~\ref{ass:potential}, the theoretical potential for this experiment
is understood to be $U-U(x_\star)$. We retain the unshifted values below for
readability. This additive shift changes neither the gradient and the dynamics
nor the normalized Gibbs law, objective gaps, or critical depths.
Its global minimum is at $x_\star=-1.854300$, its nonglobal minimum is at
$x_{\rm loc}=1.792080$, and the intervening saddle is at $x_s=0.082433$.
Their objective values are $-0.464426$, $-0.270384$, and $-0.019547$,
respectively. Hence the one-dimensional landscape has critical depth
\begin{equation}\label{eq:numericalcriticaldepth}
 D_{\rm crit}=U(x_s)-U(x_{\rm loc})\approx0.250836.
\end{equation}
We verify that $D=D_{\rm crit}$ is admissible in Assumption~\ref{ass:lsi}.
Differentiating \eqref{eq:numericalpotential} gives
\begin{equation}\label{eq:toylandscapederivatives}
\begin{aligned}
 U'(x)&=\frac{x}{10}+1.3(x+2)e^{-(x+2)^2}
                         +0.9(x-2)e^{-(x-2)^2},\\
 U''(x)&=\frac1{10}+1.3(1-2(x+2)^2)e^{-(x+2)^2}
                         +0.9(1-2(x-2)^2)e^{-(x-2)^2}.
\end{aligned}
\end{equation}
Here $U'<0$ on $(-\infty,-2]$ and $U'>0$ on $[2,\infty)$.
Interval evaluation of \eqref{eq:toylandscapederivatives} isolates exactly
one zero in each of
\[
 I_\star=[-1.85431,-1.85429],\quad
 I_s=[0.08242,0.08244],\quad
 I_{\rm loc}=[1.79207,1.79209],
\]
and excludes zeros in $[-2,2]\setminus(I_\star\cup I_s\cup I_{\rm loc})$.
The endpoint signs of $U'$ are $(-,+)$, $(+,-)$, and $(-,+)$,
respectively, and the curvature bounds are
\[
 U''(I_\star)\subset[1.31863,1.31867],\qquad
 U''(I_s)\subset[-0.17518,-0.17514],\qquad
 U''(I_{\rm loc})\subset[0.88736,0.88740].
\]
These interval checks are reproduced by
\texttt{numerics/check\_toy\_landscape.py}. Thus the shifted potential is a
Morse double well with a unique global minimum and a unique communicating
saddle. In addition, \eqref{eq:toylandscapederivatives} gives
\[
 \lim_{|x|\to\infty}\frac{|U'(x)|^2-U''(x)}{x^2}=\frac1{100},
 \qquad \inf_{x\in\R}U''(x)>-\infty.
\]
The growth and saddle conditions in
\cite[Assumptions~1.5 and~1.7, Corollary~2.18]{MenzSchlichting2014}
therefore hold. That corollary gives a positional LSI constant satisfying,
for small $\eps>0$ and every $\zeta>0$,
\begin{equation}\label{eq:toybarrierlsi}
 C_{\rm LS}^{\rm pos}(\eps)
 \leq C(1+\eps^{-1})e^{D_{\rm crit}/\eps}
 \leq C_\zeta e^{(D_{\rm crit}+\zeta)/\eps}.
\end{equation}
The bound \eqref{eq:toybarrierlsi} verifies \eqref{eq:lsi} with the exact depth in
\eqref{eq:numericalcriticaldepth}, whose rounded value is $0.250836$.
We use the canonical logarithmic form in \eqref{eq:coolingassumption},
\[
 \eps_t=\frac{E}{\log(t+e)},\qquad E=0.28>D_{\rm crit},
\]
and start every numerical path at $x_{\rm loc}$ with zero auxiliary variables.
This deterministic initialization has a finite fourth moment and is allowed
by Assumption~\ref{ass:discrete}. The target
$G_{0.08}=\{x:U(x)-U(x_\star)\leq0.08\}$ excludes the local well, whose gap is
$0.194042$. At $T=300$ the temperature has fallen to $0.049013$.

Validation selects $\gamma_{\rm u}=0.25$ for ULD and $(\lambda,\gamma)=(1,4)$ for third
order. We then use $4{,}000$ new paths per method and record terminal success,
ever-reached success, the restricted mean first-passage time, the terminal
objective gap, and completed left--right crossings. A crossing is counted only
after the path moves from $x>1$ to $x<-1$, or conversely, to avoid rapid
recrossings near the saddle.

The first protocol holds the physical-time discretization fixed: all methods use
$\eta_k=0.05e^{1/2}(\tau_k+e)^{-1/2}$ and the same horizon $T=300$. The second
protocol holds the computational budget at $40{,}000$ new gradient evaluations.
Its grids use $\eta_k=0.05e^a(\tau_k+e)^{-a}$, so $\eta_0=0.05$ for every method.
OLD and UBU use the shared baseline exponent $1/2$, while the third-order
method uses exponent
$a_\star=(1+D_{\rm crit}/E)/4=0.473961$, the smallest exponent in
Theorem~\ref{thm:discrete} that retains the continuous-time rate. This tests the
iteration-count mechanism suggested by Remark~\ref{rem:stepsizecomparison};
because it changes both the attained physical time and terminal temperature, it
evaluates complete method--grid pairs, incorporating both the attained physical
time and terminal temperature. For these parameters, the sufficient convergence
condition in Theorem~\ref{thm:ubu} is
$a>2D_{\rm crit}/(3E)=0.597229$, and the rate-preserving threshold is
$a_{\rm u,*}=(1+D_{\rm crit}/E)/3=0.631948$. Thus the shared baseline
$a=1/2$ lies outside the asymptotic coverage of the present UBU theorem and
gives UBU a slower-decaying step schedule than its sufficient threshold requires.
At the equal gradient budget, UBU reaches $T=287.762$, while the third-order
grid reaches $T=309.029$.

\begin{table}[H]
\centering
\caption{Theory-guided annealing on \eqref{eq:numericalpotential}, with standard
errors in parentheses. ``Ever'' means that $G_{0.08}$ was reached by time $T$,
and RMST is $\E[\min\{\tau,T\}]$. Each row uses $4{,}000$ evaluation paths.}
\label{tab:annealing}
\scriptsize
\resizebox{\textwidth}{!}{%
\begin{tabular}{@{}llccccccc@{}}
\toprule
Protocol & Method & $N_\nabla$ & $T$ & Terminal & Ever
& RMST & $\E[U(X_T)-U_\star]$ & Mean crossings \\
\midrule
Common horizon & OLD & $42{,}557$ & $300.000$ & $0.324\;(0.007)$
& $0.356\;(0.008)$ & $216.9\;(1.9)$ & $0.1540\;(0.0016)$ & $0.371\;(0.008)$ \\
& ULD (UBU) & $42{,}557$ & $300.000$ & $0.402\;(0.008)$
& $0.459\;(0.008)$ & $195.1\;(2.0)$ & $0.1364\;(0.0017)$ & $0.492\;(0.009)$ \\
& Third order & $42{,}557$ & $300.000$ & $0.421\;(0.008)$
& $0.474\;(0.008)$ & $191.5\;(2.0)$ & $0.1334\;(0.0017)$ & $0.505\;(0.009)$ \\
\addlinespace
Equal budget & OLD & $40{,}000$ & $287.762$ & $0.326\;(0.007)$
& $0.356\;(0.008)$ & $209.3\;(1.8)$ & $0.1547\;(0.0016)$ & $0.379\;(0.008)$ \\
& ULD (UBU) & $40{,}000$ & $287.762$ & $0.405\;(0.008)$
& $0.457\;(0.008)$ & $186.5\;(1.9)$ & $0.1363\;(0.0017)$ & $0.489\;(0.009)$ \\
& Third order & $40{,}000$ & $309.029$ & $0.415\;(0.008)$
& $0.469\;(0.008)$ & $199.7\;(2.0)$ & $0.1342\;(0.0017)$ & $0.496\;(0.009)$ \\
\bottomrule
\end{tabular}%
}
\end{table}

\begin{figure}[H]
\centering
\includegraphics[width=0.8\textwidth]{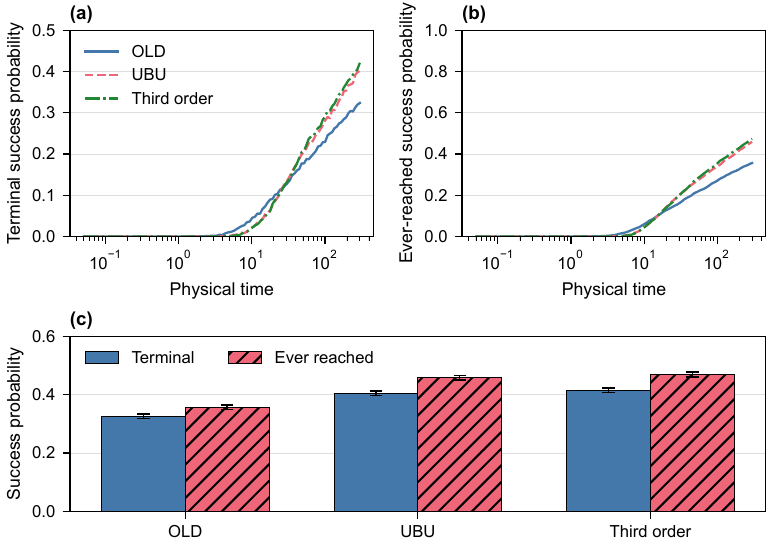}
\caption{Theory-guided annealing with tolerance $0.08$. (a) Terminal and
(b) ever-reached success under the common-horizon protocol. (c) Terminal and ever-reached
success at an equal gradient budget. Error bars in panel (c) are one
Monte Carlo standard error.}
\label{fig:annealing}
\end{figure}

At the common physical horizon, all reported point estimates favor the tuned
third-order process over tuned UBU: terminal success is $0.421$ versus $0.402$,
ever-reached success is $0.474$ versus $0.459$, restricted mean hitting time is
$191.5$ versus $195.1$, and the mean terminal gap is $0.1334$ versus $0.1364$.
The terminal-success difference (third-order Langevin method minus UBU) is $0.0198$ with an
independent-sample standard error $0.0110$, giving a 95\% normal interval
approximately $[-0.002,0.041]$. At the equal gradient budget, third order again
has the higher terminal and ever-reached point estimates, $0.415$ versus $0.405$
and $0.469$ versus $0.457$, respectively; the terminal-success difference is
$0.0098$ with standard error $0.0110$. These intervals quantify the current
Monte Carlo uncertainty in the kinetic ranking. Moreover, the equal-budget
third-order grid reaches a later physical time and lower terminal temperature,
so that protocol compares complete method--grid pairs. Both kinetic methods
outperform OLD for terminal and first-passage criteria in this parameter regime.

As a discretization check, halving the initial step and using $1{,}000$ additional
paths at the common horizon gives terminal-success estimates $0.308$, $0.423$,
and $0.412$ for OLD, UBU, and third-order Langevin method, and ever-reached estimates $0.352$,
$0.481$, and $0.459$. This smaller independent sample gives higher UBU point
estimates, with the differences covered by its larger Monte Carlo uncertainty.

\subsection{A neural-network example with synthetic data}\label{sec:NN:synthetic}

We next consider a high-dimensional nonconvex objective given by the
regularized empirical risk of a two-layer neural network. Let
$x_i\in\R^{20}$ and write $\widetilde x_i=(x_i,1)\in\R^{21}$. Define
\begin{equation}\label{eq:nnpotential}
 f_\theta(x_i):=\frac{A}{\sqrt m}\sum_{j=1}^m
 \tanh(b_j)\tanh\left(w_j^\top\widetilde x_i\right),\qquad
 U(\theta):=\frac{1}{2n}\sum_{i=1}^n(f_\theta(x_i)-y_i)^2
 +\frac{\mu}{2}\|\theta\|^2.
\end{equation}
We take $n=32$, $m=4$, $A=2$, and $\mu=10^{-3}$,
so $\theta=(w_1,\ldots,w_m,b_1,\ldots,b_m)\in\R^{88}$.  The targets are
generated once from a fixed $12$-unit teacher network, with
$x_i\sim\mathcal N(0,I_{20}/20)$ and data seed $11$.  Both layers are
trainable; the parametrization $\tanh(b_j)$ merely bounds the effective output
weights.  The risk is nonconvex. To check the regularity assumptions, let
$R(\theta)=U(\theta)-\mu|\theta|^2/2$ be the empirical-risk term in
\eqref{eq:nnpotential}. For the fixed finite data set, bounded derivatives
of $\tanh$ and the bounded output weights give
$\|D^jR\|_\infty\leq C_{\rm data}$, $1\leq j\leq4$. Consequently,
\begin{equation}\label{eq:nnregularity}
\begin{aligned}
 \|\nabla^2U\|_\infty&\leq\mu+C_{\rm data},&
 \|\nabla^3U\|_\infty+\|\nabla^4U\|_\infty&\leq2C_{\rm data},\\
 \theta\cdot\nabla U(\theta)
 &\geq\mu|\theta|^2-C_{\rm data}|\theta|
 \geq\frac{\mu}{2}|\theta|^2-\frac{C_{\rm data}^2}{2\mu}.
\end{aligned}
\end{equation}
After subtracting $\min U$, these bounds verify
Assumption~\ref{ass:potential} and the additional bounded-third-derivative
condition in Assumption~\ref{ass:discrete}. The cooling amplitude in this
experiment is selected by the validation procedure below.

A preliminary search from $96$ random starts identifies three recurrent
local-minimum levels, with
objective values $0.012778835$, $0.012810789$, and $0.012842665$, reached by
$39$, $29$, and $27$ starts, respectively.  All annealing paths begin at a
representative of the highest of these three levels.  We then apply the same
deterministic L-BFGS-B ``quench'' to every stochastic endpoint. The lowest
level found by the preliminary search defines the reference basin, and a run
is classified as successful when its quenched value is within $10^{-6}$ of
that level.

On $96$ separate validation paths per method, a common grid
$E\in\{0.004,0.008,0.012,0.016,0.024\}$ selects $E=0.024$ by the
method-symmetric rule of making pooled best-basin success closest to $0.4$.
We then tune the kinetic parameters specifically for this neural-network
problem.  A first stage evaluates
\[
 \gamma_{\rm u}\in\{0.1,0.25,0.5,1,2,4\},\qquad
 \lambda\in\{0.25,0.5,1,1.5,2\},\qquad
 \gamma\in\{0.25,0.5,1,2,4,8\}
\]
on $64$ validation paths per candidate, using common random numbers within
each method.  The three best UBU and third-order candidates are then compared
on $192$ new validation paths per candidate.  Maximizing confirmed best-basin
success, with the mean quenched objective as tie-breaker, selects
$\gamma_{\rm u}=0.1$ and $(\lambda,\gamma)=(0.5,4)$.
The held-out experiment uses $768$ paths per method, the shared schedule
$\eps_t=0.024/\log(t+e)$, the common grid
$\eta_k=0.05e^{1/2}(\tau_k+e)^{-1/2}$, and horizon $T=50$.  Hence every method
uses $3{,}059$ new gradient evaluations.

\begin{table}[H]
\centering
\caption{Held-out annealing results for the $88$-dimensional neural-network
potential \eqref{eq:nnpotential}.  Parentheses contain one Monte Carlo standard
error.  The objective gap is measured from the best level found in the
preliminary multistart search and is reported after deterministic quenching.}
\label{tab:nnannealing}
\begin{tabular}{@{}l c @{\hspace{2em}} c@{}}
\toprule
Method & Best-basin probability & Mean quenched gap $(\times10^{-3})$ \\
\midrule
OLD & $0.243\;(0.015)$ & $0.0445\;(0.0016)$ \\
ULD (UBU) & $0.388\;(0.018)$ & $0.0352\;(0.0023)$ \\
Third-order LD & $0.418\;(0.018)$ & $0.0377\;(0.0027)$ \\
\bottomrule
\end{tabular}
\end{table}

\begin{figure}[H]
\centering
\includegraphics[width=0.8\textwidth]{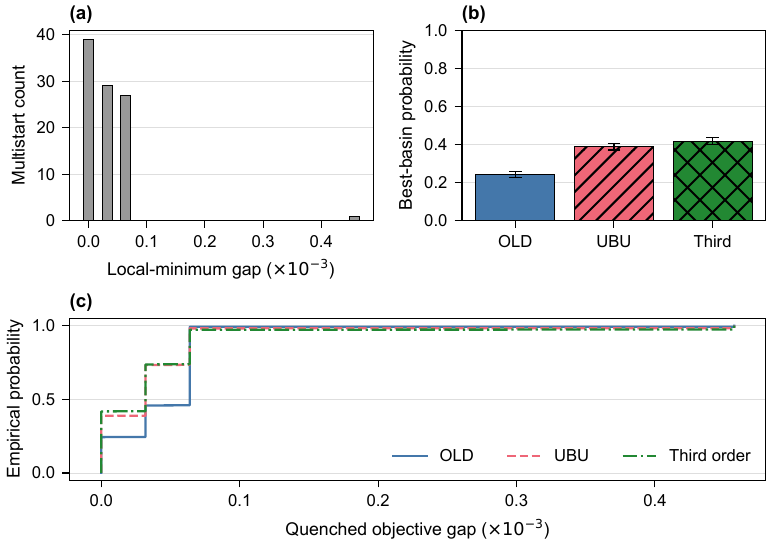}
\caption{Neural-network experiment. (a) Local
minimum levels identified by the preliminary multistart search. (b)
Held-out best-basin probabilities, with one-standard-error bars. (c)
Empirical distributions of the quenched endpoint objective gaps; curves farther
to the upper left are better.}
\label{fig:nnannealing}
\end{figure}

As a same-random-number sensitivity analysis, the parameters inherited from the
one-dimensional experiment, $\gamma_{\rm u}=0.25$ and $(\lambda,\gamma)=(1,4)$, give
best-basin probabilities $0.383$ for UBU and $0.378$ for third-order Langevin method on the
held-out noise streams.  The network-specific tuning therefore changes the
corresponding point estimates by $0.005$ and $0.040$, respectively.

Relative to OLD, the best-basin improvements are $0.145$ for UBU and $0.174$
for the third-order LD method, with independent-sample 95\% normal intervals
$[0.099,0.190]$ and $[0.128,0.221]$, respectively. The third-order best-basin
probability exceeds the UBU point estimate by $0.030$, with a 95\% normal
interval of $[-0.019,0.079]$. Halving the initial step and using $192$
new paths per method gives best-basin estimates $0.255$, $0.380$, and $0.427$
for OLD, UBU, and third-order LD method.

\subsection{A neural-network example with real data}\label{sec:NN:real}

We next retain the nonconvex neural-network potential in Section~\ref{sec:NN:synthetic} while replacing the
synthetic observations by the Connectionist Bench (Sonar, Mines versus Rocks)
data from the UCI Machine Learning Repository \cite{GormanSejnowski1988Sonar}.
The data contain $208$ observations with $60$ real-valued features and binary
rock/mine labels.  A fixed stratified split (seed $151$) assigns $160$
observations to training and $48$ to testing.  Each feature is standardized
using only the training sample and the resulting vectors are divided by
$\sqrt{60}$; labels are encoded as $-1$ and $1$.

We reuse the objective in \eqref{eq:nnpotential}, now with
$\widetilde x_i\in\R^{61}$, $n=160$, $m=4$, $A=2$, and $\mu=10^{-3}$.
Consequently, $\theta\in\R^{248}$.  The estimates \eqref{eq:nnregularity} apply to this fixed data set as well,
with a different $C_{\rm data}$, so the same regularity and dissipativity
conditions hold; the empirical risk remains nonconvex.  A deterministic search from $64$ random
starts finds $39$ distinct local-minimum levels at tolerance $10^{-6}$; the
best level occurs from $12$ starts. Every annealing path begins at a recurrent
higher level. As above, the deterministic endpoint quench assigns each path
to a local-minimum basin. The lowest value found by the multistart search is
the reference level, and success means a quenched objective within $10^{-4}$
of that value.

Temperature and kinetic parameters are selected on validation paths, with
independent paths reserved for final evaluation.
On $24$ validation paths per method, the method-symmetric rule of pooled
best-basin success closest to $0.35$ selects $E=0.24$ from
$\{0.08,0.12,0.16,0.20,0.24,0.30,0.40\}$.
A first tuning stage uses $32$ paths per candidate over
$\gamma_{\rm u}\in\{0.05,0.1,0.25,0.5,1\}$ and
$(\lambda,\gamma)\in\{0.25,0.5,1\}\times\{2,4,8\}$; the two finalists for
each kinetic method are compared on $96$ fresh paths.  This selects
$\gamma_{\rm u}=0.05$ and $(\lambda,\gamma)=(1,4)$.  The held-out evaluation uses $256$
new paths per method, $\eps_t=0.24/\log(t+e)$,
$\eta_k=0.05e^{1/2}(\tau_k+e)^{-1/2}$, and horizon $T=20$, corresponding to $839$
gradient evaluations per path.

\begin{table}[H]
\centering
\caption{Held-out results for the $248$-dimensional UCI Sonar potential.
Parentheses contain one Monte Carlo standard error across independent
annealing paths.  Test-accuracy errors are conditional on the fixed $48$-record
test split.}
\label{tab:sonarannealing}
\begin{tabular}{@{}l c c c@{}}
\toprule
Method & Best-basin probability & Mean quenched objective & Test accuracy \\
\midrule
OLD & $0.180\;(0.024)$ & $0.10671\;(0.00035)$ & $0.824\;(0.003)$ \\
ULD (UBU) & $0.277\;(0.028)$ & $0.10427\;(0.00033)$ & $0.834\;(0.003)$ \\
Third-order LD & $0.352\;(0.030)$ & $0.10357\;(0.00034)$ & $0.843\;(0.003)$ \\
\bottomrule
\end{tabular}
\end{table}

\begin{figure}[H]
\centering
\includegraphics[width=0.8\textwidth]{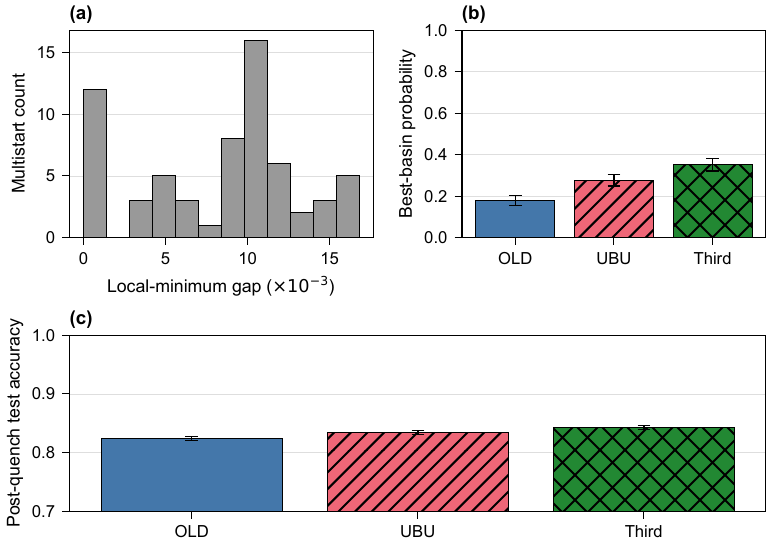}
\caption{UCI Sonar experiment. (a) Local-minimum gaps found by the
$64$-start deterministic search. (b) Held-out best-basin probabilities.
(c) Mean post-quench accuracy on the fixed test split. Error bars in panels
(b) and (c) are one Monte Carlo standard error across annealing
paths.}
\label{fig:sonarannealing}
\end{figure}

Relative to OLD, the UBU and third-order best-basin improvements are $0.098$ and
$0.172$, respectively.  The third-order point estimate is $0.074$ higher than
UBU.  The corresponding post-quench test-accuracy point estimates are
$0.824$, $0.834$, and $0.843$, again with the third-order LD method highest.  With the initial step
halved and $64$ fresh paths per method, best-basin estimates are $0.266$,
$0.281$, and $0.344$, while test accuracies are $0.835$, $0.830$, and $0.844$.
Thus, the third-order LD method remains the highest point estimate under step refinement.


\section{Conclusion}\label{sec:conclusion}

We established global convergence guarantees of third-order Langevin dynamics for non-convex optimization
via simulated annealing with fixed friction and decreasing noise.  A three-block distorted entropy transfers
dissipation through the full chain, and an initial-density approximation covers
nonsmooth initial laws.  Logarithmic cooling then gives convergence in
probability at the barrier-controlled kinetic rate.  For the exact-force and
midpoint three-stage schemes, a cubic local endpoint estimate yields a less
restrictive sufficient step-size condition and a better sufficient iteration
exponent than both the existing kinetic frozen-force result and the UBU bound
proved by the same strong-coupling method.
The experiments complement these rate comparisons.  The third-order Langevin method
has higher terminal-success point estimates than UBU in the theory-guided
double well.  On a high-dimensional synthetic neural-network objective using synthetic data,
independent two-stage tuning and held-out evaluation show that both kinetic
methods outperform OLD in best-basin recovery, while the third-order Langevin
method has a higher point estimate than UBU; both improvements over OLD persist
after step refinement.
Using the real data, the same best-basin ordering is observed,
and the third-order Langevin method also gives the highest post-quench test-accuracy
point estimate. These qualitative conclusions persist under step refinement.

\bibliographystyle{alpha}
\bibliography{bibtex}

\appendix

\section{Underdamped UBU Comparison}\label{sec:ubuappendix}

This appendix provides the UBU result summarized in
Table~\ref{tab:introdiscretecomparison} and used in the comparison in
Remark~\ref{rem:stepsizecomparison}.

\subsection{Scheme and convergence result}

We now define the underdamped comparator used in
Table~\ref{tab:introdiscretecomparison} and Section~\ref{sec:numerics}.
In particular, let us define the UBU scheme $(\bar{X}_{k},\bar{V}_{k})_{k\geq 0}$ for ULD.
Fix a friction parameter $\gamma_{\rm u}>0$ and let
$P^{\rm u}_{t,T}$ denote the transition operator of the time-inhomogeneous ULD:
\begin{equation}\label{eq:uldannealed}
 \dd X_t=V_t\dd t,\qquad
 \dd V_t=\left(-\nabla U(X_t)-\gamma_{\rm u}V_t\right)\dd t
       +\sqrt{2\gamma_{\rm u}\eps_t}\,\dd B_t.
\end{equation}
The frozen generator associated with \eqref{eq:uldannealed} is
\begin{equation}\label{eq:uldgenerator}
 \mathcal L^{\rm u}_\eps
 =v\cdot\nabla_x-(\nabla U+\gamma_{\rm u}v)\cdot\nabla_v
                         +\gamma_{\rm u}\eps\Delta_v.
\end{equation}
Let $P^{{\rm u},\eps}_r$ denote its transition operator over time $r$.
At frozen temperature $\eps$, its invariant law is
\begin{equation}\label{eq:uldgibbs}
 \pi^{\rm u}_\eps(\dd x\,\dd v)
 \propto \exp\left[-\frac{U(x)+|v|^2/2}{\eps}\right]\dd x\,\dd v.
\end{equation}
Indeed, for $H_{\rm u}(x,v)=U(x)+|v|^2/2$ and
$\mathcal T_{\rm u}=v\cdot\nabla_x-\nabla U\cdot\nabla_v$,
\[
 \mathcal T_{\rm u}H_{\rm u}=0,\qquad
 \operatorname{div}(v,-\nabla U)=0,\qquad
 (\mathcal L^{\rm u}_\eps)^*e^{-H_{\rm u}/\eps}=0.
\]
The last equality follows by combining these transport identities with
$\nabla_v e^{-H_{\rm u}/\eps}=-\eps^{-1}v e^{-H_{\rm u}/\eps}$ in
\eqref{eq:uldgenerator}.
For $r>0$, let $\mathcal U_r^\eps$ be the exact Gaussian flow over time $r$ of
the force-free part of \eqref{eq:uldannealed}. Writing
$q_r:=e^{-\gamma_{\rm u}r}$ and $F_r:=(1-q_r)/\gamma_{\rm u}$, its mean is
$(x+F_rv,q_rv)$. Its centered noise is
\begin{equation}\label{eq:ubuflownoise}
 \sqrt{2\gamma_{\rm u}\eps}\int_0^r
       \begin{pmatrix}F_{r-u}\\q_{r-u}\end{pmatrix}
           \otimes\Id_d\,\dd B_u.
\end{equation}
The covariance in \eqref{eq:ubuflownoise} is
$2\gamma_{\rm u}\eps\int_0^r
(F_u,q_u)^\top(F_u,q_u)\dd u\otimes\Id_d$.
Evaluating the three scalar integrals gives
\begin{equation}\label{eq:ubuflowcov}
 \eps\begin{pmatrix}
 \displaystyle \frac{2r}{\gamma_{\rm u}}
  -\frac{4(1-q_r)-(1-q_r^2)}{\gamma_{\rm u}^2}
 &\displaystyle \frac{(1-q_r)^2}{\gamma_{\rm u}}\\[3pt]
 \displaystyle \frac{(1-q_r)^2}{\gamma_{\rm u}}
 &1-q_r^2
 \end{pmatrix}\otimes\Id_d.
\end{equation}
Given $\bar S_k=(\bar{X}_{k},\bar{V}_{k})=(x,v)$, one frozen-temperature UBU step is
\begin{equation}\label{eq:ubustep}
 (y,w)=\mathcal U_{\eta_k/2}^{\eps_{\tau_k}}(x,v),\qquad
 w^+=w-\eta_k\nabla U(y),\qquad
 \bar S_{k+1}=\mathcal U_{\eta_k/2}^{\eps_{\tau_k}}\left(y,w^+\right),
\end{equation}
where the two half steps use independent Brownian increments. This is the
one-gradient symmetric splitting of Sanz-Serna and Zygalakis
\cite[Example~8]{SanzSernaZygalakis2021}. Its local-error analysis must retain
the centered stochastic term of order $\eta^{5/2}$; see also the correction
\cite{PaulinWhalley2024}.

The following theorem turns this centered local-error structure into a
quantitative annealing rate and step-size condition for the UBU scheme.

\begin{theorem}[Underdamped annealing with UBU]\label{thm:ubu}
Assume Assumptions~\ref{ass:potential}, \ref{ass:lsi}, and
\ref{ass:discrete}. For the UBU scheme \eqref{eq:ubustep}, take
\[
 \eta_k=c_\eta s_k^{-a},\qquad c_\eta>0,\quad a>0,
\]
with $c_\eta$ sufficiently small, and let the initial numerical law have a finite
fourth moment. Write $\vartheta=D/E$. If $a>2\vartheta/3$, then for every
$\delta,\alpha>0$ and all sufficiently large $k$,
\begin{equation}\label{eq:uburate}
 \mathbb P\left(U(\bar X_k)>\delta\right)
 \leq C_{\delta,\alpha}s_k^{-r_{\rm u}(\delta,a)+\alpha},
 \qquad
 r_{\rm u}(\delta,a)
 :=\min\left\{\frac\delta E,\frac{1-\vartheta}{2},
                         \frac{3a}{2}-\vartheta\right\}.
\end{equation}
In particular, $a\geq(1+\vartheta)/3$ retains the continuous-time exponent.
Moreover $s_k\asymp k^{1/(1+a)}$, so the corresponding iteration exponent is
$r_{\rm u}(\delta,a)/(1+a)$, up to an arbitrarily small loss.
\end{theorem}

\begin{proof}
The proof is given in Section~\ref{sec:proof_ubu}.
\end{proof}

Choosing the smallest step-decay exponent in Theorem~\ref{thm:ubu} that
preserves the continuous-time rate yields the iteration exponent in the following corollary.

\begin{corollary}[Rate-preserving UBU exponent]\label{cor:ubuiteration}
Under the assumptions of Theorem~\ref{thm:ubu}, let
\[
 r_{\rm c}(\delta):=\min\left\{\frac\delta E,
                    \frac{1-\vartheta}{2}\right\},
 \qquad a_{\rm u,*}:=\frac{1+\vartheta}{3}.
\]
Then, for every $\delta,\alpha>0$, the choice $a=a_{\rm u,*}$ gives
\[
 \mathbb P\left(U(\bar X_k)>\delta\right)
 \leq C_{\delta,\alpha}
 k^{-3r_{\rm c}(\delta)/(4+\vartheta)+\alpha}
\]
for all sufficiently large $k$.
\end{corollary}

\begin{proof}
At $a=a_{\rm u,*}$,
$3a/2-\vartheta=(1-\vartheta)/2$, and \eqref{eq:gridasymptotics} gives
$s_k\asymp k^{1/(1+a)}=k^{3/(4+\vartheta)}$.
Choose $\alpha'=(1+a_{\rm u,*})\alpha$ in \eqref{eq:uburate}. Then
\begin{align*}
 \mathbb P\left(U(\bar X_k)>\delta\right)
 \leq C_{\delta,\alpha'}s_k^{-r_{\rm c}(\delta)+\alpha'}
 \leq C_{\delta,\alpha}k^{-3r_{\rm c}(\delta)/(4+\vartheta)+\alpha}.
\end{align*}
\end{proof}

\subsection{Proof of the UBU comparison}\label{sec:proof_ubu}

We collect the underdamped analogues of the local comparison,
grid-regularization, restart, and smoothing estimates. Constants in this
subsection may depend on $d,U,\gamma_{\rm u}$, and $c_\eta$.

We begin with the one-step strong error and uniform moment control for the
UBU chain in the next lemma.

\begin{lemma}[UBU local error and moments]\label{lem:ubulocal}
Let $Q^{\rm UBU}_{t,\eta}$ be one step of \eqref{eq:ubustep}, with temperature
frozen at $\eps_t$, and couple it synchronously to \eqref{eq:uldannealed}
started from $s=(x,v)$. For $\eta\leq \eta_*$ and all sufficiently large $t$,
\begin{equation}\label{eq:ubulocal}
 \left\|\bar S_\eta^{\rm UBU}-\bar S_{t+\eta}^{\rm exact}\right\|_{L^2(\mathbb P_s)}
 \leq C\sqrt{\eps_t}\,\eta^{5/2}
      +C\eta^3\left(1+|s|^2\right)
      +\frac{C\eta^{3/2}}{(t+t_0)\log^{3/2}(t+t_0)}.
\end{equation}
The $\mathcal{O}\left(\sqrt{\eps_t}\,\eta^{5/2}\right)$ contribution is an It\^o integral with
conditional mean zero. The numerical chain satisfies
\begin{equation}\label{eq:ubumoments}
 \sup_{k\geq0}\E\left|\bar S_k\right|^p<\infty,
 \qquad 1\leq p\leq4.
\end{equation}
\end{lemma}

\begin{proof}
First freeze the temperature at $\eps_t$. We give the conditional local
calculation because the estimates in the fixed-temperature UBU literature are
usually stated after averaging over an equilibrium law. Write
$G=\nabla U$, $\mathsf H_U=\nabla^2U$, and let $(X_r,V_r)$ be the exact frozen ULD
started from $s$. Let
$N_r=\sqrt{2\gamma_{\rm u}\eps_t}\,B_r$ and $S_r=(X_r,V_r)$.
The integral equation and the linear-growth bound for the drift give
\[
 \sup_{r\leq u}|S_r|
 \leq |s|+C\int_0^u\left(1+\sup_{r\leq v}|S_r|\right)\dd v
       +\sup_{r\leq u}|N_r|.
\]
Gr\"{o}nwall's inequality and Doob's maximal inequality therefore yield
\begin{align*}
 \sup_{r\leq\eta_*}|S_r|
 &\leq C_{\eta_*}\left(1+|s|+\sup_{r\leq\eta_*}|N_r|\right),\\
 \left\|\sup_{r\leq\eta_*}|B_r|\right\|_{L^p}
 &\leq\frac{p}{p-1}\|B_{\eta_*}\|_{L^p}
 \leq C_{p,d}\sqrt{\eta_*}.
\end{align*}
Since $\eps_t\leq\eps_0$, taking $L^p$ norms gives, for every fixed $p\geq2$,
\begin{equation}\label{eq:ubuconditionalmoments}
 \sup_{0\leq r\leq \eta_*}
 \left(\|X_r\|_{L^p(\mathbb P_s)}+\|V_r\|_{L^p(\mathbb P_s)}
       +\|G(X_r)\|_{L^p(\mathbb P_s)}\right)
 \leq C_p(1+|s|).
\end{equation}
Let $\left(X_r^0,V_r^0\right)$ be the force-free process driven by the same Brownian
path, and let $Y=X_{\eta/2}^0$ be the UBU force point. Variation of constants gives
\begin{equation}\label{eq:ubuforcepoint}
\begin{aligned}
 X_{\eta/2}-Y&=-\int_0^{\eta/2}F_{\eta/2-r}G(X_r)\dd r,\\
 \|X_{\eta/2}-Y\|_{L^p(\mathbb P_s)}
 &\leq C_p\eta^2(1+|s|).
\end{aligned}
\end{equation}
Subtracting the endpoint formulas obtained from
\eqref{eq:uldannealed} and \eqref{eq:ubustep} gives the velocity error
\begin{equation}\label{eq:ubuvelocityerror}
 \bar V_\eta^{\rm UBU}-V_\eta^{\rm frozen}
 =
 \int_0^\eta q_{\eta-r}G(X_r)\dd r-\eta q_{\eta/2}G(Y).
\end{equation}
For $\chi(r)=q_{\eta-r}G(X_r)$, the symmetric midpoint identity reads
\begin{equation}\label{eq:ubumidpointidentity}
 \int_0^\eta\chi(r)\dd r-\eta\chi(\eta/2)
 =\int_{\eta/2}^\eta\dd u\int_{\eta/2}^u\dd v
       \int_{\eta-v}^{v}\dd\dot\chi(r).
\end{equation}
Indeed, $\chi$ is continuously differentiable and
\begin{align*}
 \int_0^\eta\chi(r)\dd r-\eta\chi(\eta/2)
 &=\int_{\eta/2}^\eta
       \left(\chi(u)+\chi(\eta-u)-2\chi(\eta/2)\right)\dd u\\
 &=\int_{\eta/2}^\eta\dd u\int_{\eta/2}^u
       \left(\dot\chi(v)-\dot\chi(\eta-v)\right)\dd v.
\end{align*}
Writing the last difference as an increment of $\dot\chi$ gives
\eqref{eq:ubumidpointidentity}. By \eqref{eq:uldannealed} at temperature
$\eps_t$,
\[
\dot\chi(r)=q_{\eta-r}\{\gamma_{\rm u}G(X_r)+\mathsf H_U(X_r)V_r\}
\]
is a continuous
semimartingale. Since $X$ has finite variation, It\^o's formula gives
\begin{equation}\label{eq:ubuchiito}
 \dd\dot\chi(r)=a_\chi(r)\dd r+\sigma_\chi(r)\dd B_r,
\end{equation}
with finite-variation and martingale integrands
\begin{equation}\label{eq:ubuchicoefficients}
\begin{aligned}
 a_\chi(r)&:=q_{\eta-r}\left[
 \gamma_{\rm u}^2G(X_r)+\gamma_{\rm u}\mathsf H_U(X_r)V_r
 +\nabla^3U(X_r)[V_r,V_r]-\mathsf H_U(X_r)G(X_r)\right],\\
 \sigma_\chi(r)&:=\sqrt{2\gamma_{\rm u}\eps_t}\,q_{\eta-r}\mathsf H_U(X_r).
\end{aligned}
\end{equation}
Since $0<q_{\eta-r}\leq1$, $\|\mathsf H_U\|_{\rm op}\leq K$, and
$\|\nabla^3U\|_\infty<\infty$, the coefficients
\eqref{eq:ubuchicoefficients} and the conditional moment bound
\eqref{eq:ubuconditionalmoments} with $p=2,4$ gives, for $0\leq r\leq \eta$,
\begin{equation}\label{eq:ubuchiintegrands}
\begin{aligned}
 \|a_\chi(r)\|_{L^2(\mathbb P_s)}
 &\leq C\left(\|G(X_r)\|_{L^2(\mathbb P_s)}+\|V_r\|_{L^2(\mathbb P_s)}
       +\|V_r\|_{L^4(\mathbb P_s)}^2\right)
 \leq C\left(1+|s|^2\right),\\
 \|\sigma_\chi(r)\|_{\rm F}^2
 &\leq2\gamma_{\rm u}\eps_t\,dK^2,
\end{aligned}
\end{equation}
where $\|\cdot\|_{\rm F}$ denotes the Frobenius norm. Let
\[
 \mathcal T_\eta:=\left\{(u,v,r):\ \frac \eta2\leq v\leq u\leq \eta,\
 \eta-v\leq r\leq v\right\},
 \qquad |\mathcal T_\eta|=\frac{\eta^3}{24},
\]
be the triangular integration domain in \eqref{eq:ubumidpointidentity}. By
\eqref{eq:ubuchiintegrands}, we have:
\begin{equation}\label{eq:ubufubiniconditions}
 \int_{\mathcal T_\eta}\E_s|a_\chi(r)|\,\dd(u,v,r)
 \leq\frac{\eta^3}{24}\,C\left(1+|s|^2\right),
 \qquad
 \int_{\mathcal T_\eta}\E_s\|\sigma_\chi(r)\|_{\rm F}^2\,\dd(u,v,r)
 \leq\frac{\gamma_{\rm u}\eps_t dK^2}{12}\eta^3.
\end{equation}
The first bound justifies Fubini's theorem for the finite-variation part,
and the second is the square-integrability hypothesis of the stochastic
Fubini theorem for the martingale part. For fixed $r\in[0,\eta]$, the
$(u,v)$-section of $\mathcal T_\eta$ has area
\begin{equation}\label{eq:ubufubinikernel}
 \kappa_\eta(r):=\frac12\min\{r,\eta-r\}^2.
\end{equation}
Using \eqref{eq:ubuchiito}, \eqref{eq:ubufubiniconditions}, and
\eqref{eq:ubufubinikernel}, the stochastic Fubini theorem rewrites
\eqref{eq:ubumidpointidentity} in $L^2(\mathbb P_s)$ as
\begin{equation}\label{eq:ubuvelocitysplit}
 \int_0^\eta\chi(r)\dd r-\eta\chi(\eta/2)
 =\int_{\mathcal T_\eta}a_\chi(r)\,\dd(u,v,r)
  +\int_0^\eta\kappa_\eta(r)\sigma_\chi(r)\dd B_r.
\end{equation}
By \eqref{eq:ubuvelocityerror} and \eqref{eq:ubuvelocitysplit},
the stochastic part of the velocity error is
\begin{equation}\label{eq:ubualphav}
 \alpha_\eta^v:=\int_0^\eta\kappa_\eta(r)\sigma_\chi(r)\dd B_r
 =\sqrt{2\gamma_{\rm u}\eps_t}
 \int_{\eta/2}^\eta\dd u\int_{\eta/2}^u\dd v
       \int_{\eta-v}^{v}q_{\eta-r}\mathsf H_U(X_r)\dd B_r.
\end{equation}
The It\^o integral \eqref{eq:ubualphav} has conditional mean zero.
It\^o's isometry,
\eqref{eq:ubuchiintegrands}, and $\int_0^\eta\kappa_\eta(r)^2\dd r=\eta^5/320$ yield
\begin{equation}\label{eq:ubualphavbound}
 \|\alpha_\eta^v\|_{L^2(\mathbb P_s)}^2
 =\E_s\int_0^\eta\kappa_\eta(r)^2\|\sigma_\chi(r)\|_{\rm F}^2\dd r
 \leq\frac{\gamma_{\rm u}\eps_t dK^2}{160}\eta^5,
 \qquad\text{so}\qquad
 \|\alpha_\eta^v\|_{L^2(\mathbb P_s)}\leq C\sqrt{\eps_t}\,\eta^{5/2}.
\end{equation}
Since $\chi(\eta/2)=q_{\eta/2}G(X_{\eta/2})$, the remaining part of the velocity error is
\begin{equation}\label{eq:ububetav}
 \beta_\eta^v:=\int_{\mathcal T_\eta}a_\chi(r)\,\dd(u,v,r)
 +\eta q_{\eta/2}\{G(X_{\eta/2})-G(Y)\}.
\end{equation}
Applying Minkowski's integral inequality to \eqref{eq:ububetav},
using \eqref{eq:ubuchiintegrands}, the bounded
Hessian, and the force-point estimate \eqref{eq:ubuforcepoint}, gives
\begin{equation}\label{eq:ububetavbound}
 \|\beta_\eta^v\|_{L^2(\mathbb P_s)}
 \leq\frac{\eta^3}{24}\sup_{0\leq r\leq \eta}\|a_\chi(r)\|_{L^2(\mathbb P_s)}
 +K\eta\|X_{\eta/2}-Y\|_{L^2(\mathbb P_s)}
 \leq C\eta^3\left(1+|s|^2\right).
\end{equation}

Subtracting the position endpoints from \eqref{eq:uldannealed}
and \eqref{eq:ubustep} gives
\begin{equation}\label{eq:ubupositionerror}
 \bar X_\eta^{\rm UBU}-X_\eta^{\rm frozen}
 =
 \int_0^\eta F_{\eta-r}G(X_r)\dd r-\eta F_{\eta/2}G(Y).
\end{equation}
Apply \eqref{eq:ubumidpointidentity} to $\psi(r):=F_{\eta-r}G(X_r)$. Since
$\frac{\dd}{\dd r}F_{\eta-r}=-q_{\eta-r}$,
\begin{equation}\label{eq:ubupsiderivative}
 \dot\psi(r)=-q_{\eta-r}G(X_r)+F_{\eta-r}\mathsf H_U(X_r)V_r,
\end{equation}
Applying It\^o's formula to \eqref{eq:ubupsiderivative}, using
\eqref{eq:uldannealed}, gives
$\dd\dot\psi(r)=a_\psi(r)\dd r+\sigma_\psi(r)\dd B_r$, where
\begin{equation}\label{eq:ubupsicoefficients}
\begin{aligned}
 a_\psi(r)&:=F_{\eta-r}\left[\nabla^3U(X_r)[V_r,V_r]
              -\mathsf H_U(X_r)G(X_r)\right]
   -(1+q_{\eta-r})\mathsf H_U(X_r)V_r-\gamma_{\rm u}q_{\eta-r}G(X_r),\\
 \sigma_\psi(r)&:=\sqrt{2\gamma_{\rm u}\eps_t}\,F_{\eta-r}\mathsf H_U(X_r).
\end{aligned}
\end{equation}
Using $0\leq F_{\eta-r}\leq\eta$ and
\eqref{eq:ubuconditionalmoments} in the coefficients
\eqref{eq:ubupsicoefficients} gives
\begin{equation}\label{eq:ubupsiintegrands}
 \sup_{0\leq r\leq \eta}\|a_\psi(r)\|_{L^2(\mathbb P_s)}\leq C\left(1+|s|^2\right),
 \qquad
 \|\sigma_\psi(r)\|_{\rm F}^2\leq2\gamma_{\rm u}\eps_t\,dK^2\eta^2.
\end{equation}
Hence the analogue of \eqref{eq:ubufubiniconditions} holds for
$(a_\psi,\sigma_\psi)$, and, as in \eqref{eq:ubuvelocitysplit},
\begin{equation}\label{eq:ubupositionsplit}
 \int_0^\eta\psi(r)\dd r-\eta\psi(\eta/2)
 =\int_{\mathcal T_\eta}a_\psi(r)\,\dd(u,v,r)
  +\int_0^\eta\kappa_\eta(r)\sigma_\psi(r)\dd B_r.
\end{equation}
The stochastic part
$\alpha_\eta^x:=\int_0^\eta\kappa_\eta(r)\sigma_\psi(r)\dd B_r$ is conditionally
centered, and It\^o's isometry gives
\begin{equation}\label{eq:ubualphaxvariance}
 \|\alpha_\eta^x\|_{L^2(\mathbb P_s)}^2
 \leq2\gamma_{\rm u}\eps_t\,dK^2\eta^2\int_0^\eta\kappa_\eta(r)^2\dd r
 =\frac{\gamma_{\rm u}\eps_t dK^2}{160}\eta^7,
\end{equation}
Taking the square root in \eqref{eq:ubualphaxvariance} gives
\begin{equation}\label{eq:ubualphaxbound}
 \|\alpha_\eta^x\|_{L^2(\mathbb P_s)}\leq C\sqrt{\eps_t}\,\eta^{7/2}.
\end{equation}
The residual part of the position error is
\begin{equation}\label{eq:ububetax}
 \beta_\eta^x:=\int_{\mathcal T_\eta}a_\psi(r)\,\dd(u,v,r)
 +\eta F_{\eta/2}\{G(X_{\eta/2})-G(Y)\}.
\end{equation}
Using \eqref{eq:ubupsiintegrands}, \eqref{eq:ubuforcepoint},
and $F_{\eta/2}\leq\eta/2$ in \eqref{eq:ububetax} gives
\begin{equation}\label{eq:ububetaxbound}
 \|\beta_\eta^x\|_{L^2(\mathbb P_s)}
 \leq\frac{\eta^3}{24}\sup_{0\leq r\leq \eta}\|a_\psi(r)\|_{L^2(\mathbb P_s)}
 +\frac{K\eta^2}{2}\|X_{\eta/2}-Y\|_{L^2(\mathbb P_s)}
 \leq C\eta^3\left(1+|s|^2\right).
\end{equation}
The velocity and position formulas
\eqref{eq:ubuvelocityerror} and \eqref{eq:ubupositionerror},
with \eqref{eq:ubuvelocitysplit} and \eqref{eq:ubupositionsplit},
give the full frozen one-step decomposition
\begin{equation}\label{eq:ubulocaldecomposition}
 \bar S_\eta^{\rm UBU}-S_\eta^{\rm frozen}=\alpha_\eta+\beta_\eta,\qquad
 \alpha_\eta:=\left(\alpha_\eta^x,\alpha_\eta^v\right),\quad
 \beta_\eta:=\left(\beta_\eta^x,\beta_\eta^v\right),\qquad
 \E_s\alpha_\eta=0,
\end{equation}
The bounds \eqref{eq:ubualphavbound}, \eqref{eq:ubualphaxbound},
\eqref{eq:ububetavbound}, and \eqref{eq:ububetaxbound} give
\begin{equation}\label{eq:ubufrozenlocalbound}
 \|\alpha_\eta\|_{L^2(\mathbb P_s)}
 \leq C\sqrt{\eps_t}\left(\eta^{5/2}+\eta^{7/2}\right)
 \leq C\sqrt{\eps_t}\,\eta^{5/2},\qquad
 \|\beta_\eta\|_{L^2(\mathbb P_s)}
 \leq C\eta^3(1+|s|^2).
\end{equation}
This conditional decomposition follows the local-error analysis of
Sanz-Serna and Zygalakis \cite[Section~7.6]{SanzSernaZygalakis2021}, while
the associated dimension-dependent weighted inner-product constants are
refined by Paulin and Whalley \cite{PaulinWhalley2024}.

It remains to unfreeze the temperature. The two exact linear half flows may
be realized with the restrictions of a common Brownian path to the two half
intervals. The temperature increment is bounded by
\eqref{eq:sqrttemperatureincrement}.
Let $D_u=S_{t+u}^{\rm exact}-S_u^{\rm frozen}$. The Lipschitz ULD drift and
It\^o's isometry give, as in \eqref{eq:unfreezingcoupling},
\[
 \|D_u\|_{L^2(\mathbb P_s)}
 \leq C\int_0^u\|D_r\|_{L^2(\mathbb P_s)}\dd r
       +\frac{C\eta^{3/2}}{(t+t_0)\log^{3/2}(t+t_0)}.
\]
Applying Gr\"{o}nwall's inequality to this integral bound, and adding
\eqref{eq:ubufrozenlocalbound}, proves \eqref{eq:ubulocal}.

For the moment estimate, choose $0<b<\min\{\gamma_{\rm u}/2,
\min(r/4,1/2)\}$, set
$c=b\gamma_{\rm u}/2$, and choose $C_{\rm u}$ sufficiently large that
\begin{equation}\label{eq:uldlyapunov}
 \mathcal V_{\rm u}(x,v)
 =C_{\rm u}+U(x)+\frac12|v|^2+b\,x\cdot v+c|x|^2
\asymp1+|x|^2+|v|^2,\qquad \mathcal V_{\rm u}\geq1.
\end{equation}
The comparison in \eqref{eq:uldlyapunov} follows from
\eqref{eq:quadraticU} and
$|b\,x\cdot v|\leq b(|x|^2+|v|^2)/2$.
Substituting \eqref{eq:uldlyapunov} in \eqref{eq:uldgenerator}, the choice
$2c=b\gamma_{\rm u}$ cancels the $x\cdot v$ term:
\begin{equation}\label{eq:uldlyapunovdrift}
 \mathcal L^{\rm u}_\eps\mathcal V_{\rm u}
 =-(\gamma_{\rm u}-b)|v|^2-b\,x\cdot\nabla U(x)+\gamma_{\rm u}\eps d
 \leq-c_0\mathcal V_{\rm u}+C_0.
\end{equation}
Also $|\nabla_v\mathcal V_{\rm u}|^2\leq C\mathcal V_{\rm u}$, so
\begin{equation}\label{eq:uldpowerdrift}
 \mathcal L^{\rm u}_\eps\mathcal V_{\rm u}^p
 =p\mathcal V_{\rm u}^{p-1}\mathcal L^{\rm u}_\eps\mathcal V_{\rm u}
  +p(p-1)\gamma_{\rm u}\eps\mathcal V_{\rm u}^{p-2}
       |\nabla_v\mathcal V_{\rm u}|^2
 \leq-c_p\mathcal V_{\rm u}^p+C_p.
\end{equation}
It\^o's formula and the stopped integrating-factor argument used for
\eqref{eq:momentpowerestimate}, applied to \eqref{eq:uldpowerdrift}, give
\begin{equation}\label{eq:uldfrozenmoments}
 P^{{\rm u},\eps}_\eta\mathcal V_{\rm u}^p
 \leq e^{-c_p\eta}\mathcal V_{\rm u}^p+C_p\eta,
 \qquad p=1,2.
\end{equation}
The following coarse local bound preserves linear dependence on the initial
state:
\begin{equation}\label{eq:ubucoarselocal}
 \left\|\bar S_\eta^{\rm UBU}-S_\eta^{\rm frozen}\right\|_{L^q(\mathbb P_s)}
 \leq C_q\eta^2(1+|s|),\qquad q\geq2,
\end{equation}
where $S^{\rm frozen}$ is the exact ULD with temperature fixed at $\eps_t$.
Indeed, drive its force-free version $\left(X_r^0,V_r^0\right)$ and both UBU half
steps by consecutive pieces of the same Brownian path, and let
$Y=X_{\eta/2}^0$. The exact force-free formulas for the second half-step give
\begin{equation}\label{eq:ubunumericalendpoints}
 \bar V_\eta^{\rm UBU}=V_\eta^0-\eta q_{\eta/2}\nabla U(Y),\qquad
 \bar X_\eta^{\rm UBU}=X_\eta^0-\eta F_{\eta/2}\nabla U(Y).
\end{equation}
On the other hand, variation of constants for the exact frozen process gives
\begin{equation}\label{eq:ubuexactendpoints}
 V_\eta^{\rm frozen}=V_\eta^0-\int_0^\eta q_{\eta-r}\nabla U(X_r)\dd r,
 \qquad
 X_\eta^{\rm frozen}=X_\eta^0-\int_0^\eta F_{\eta-r}\nabla U(X_r)\dd r.
\end{equation}
By \eqref{eq:ubuconditionalmoments} and $F_r\leq r$, for
$0\leq r\leq\eta\leq\eta_*$,
\begin{equation}\label{eq:ubuforcepointincrements}
\begin{aligned}
 \|X_r-X_r^0\|_{L^q(\mathbb P_s)}
 &\leq\int_0^r F_{r-u}\|G(X_u)\|_{L^q(\mathbb P_s)}\dd u
 \leq C_q r^2(1+|s|),\\
 \|X_r^0-Y\|_{L^q(\mathbb P_s)}
 &\leq\left|\int_{\eta/2}^r
                \|V_u^0\|_{L^q(\mathbb P_s)}\dd u\right|
 \leq C_q|r-\eta/2|(1+|s|).
\end{aligned}
\end{equation}
The force-free velocity has the same moment bound by its explicit Gaussian
formula. Since $|q_r-q_{r'}|\leq\gamma_{\rm u}|r-r'|$ and
$|F_r-F_{r'}|\leq|r-r'|$, subtracting
\eqref{eq:ubuexactendpoints} from \eqref{eq:ubunumericalendpoints}
and using \eqref{eq:ubuforcepointincrements} gives
\begin{equation}\label{eq:ubucoarsecomponenterrors}
\begin{aligned}
 \left\|\bar V_\eta^{\rm UBU}-V_\eta^{\rm frozen}\right\|_{L^q(\mathbb P_s)}
 &\leq K\int_0^\eta\|X_r-Y\|_{L^q(\mathbb P_s)}\dd r
       +\gamma_{\rm u}\|G(Y)\|_{L^q(\mathbb P_s)}
                       \int_0^\eta|r-\eta/2|\dd r\\
 &\leq C_q\eta^2(1+|s|),\\
 \left\|\bar X_\eta^{\rm UBU}-X_\eta^{\rm frozen}\right\|_{L^q(\mathbb P_s)}
 &\leq K\eta\int_0^\eta\|X_r-Y\|_{L^q(\mathbb P_s)}\dd r
       +\|G(Y)\|_{L^q(\mathbb P_s)}
                       \int_0^\eta|r-\eta/2|\dd r\\
 &\leq C_q\eta^2(1+|s|).
\end{aligned}
\end{equation}
The component bounds \eqref{eq:ubucoarsecomponenterrors} prove
\eqref{eq:ubucoarselocal}. Moreover, both coupled endpoints have linear
moment growth: for every $q\geq2$ and $0<\eta\leq \eta_*$,
\begin{equation}\label{eq:ubuendpointmoments}
 \left\|S_\eta^{\rm frozen}\right\|_{L^q(\mathbb P_s)}\leq C_q(1+|s|),
 \qquad
 \left\|\bar S_\eta^{\rm UBU}\right\|_{L^q(\mathbb P_s)}\leq C_q(1+|s|).
\end{equation}
The first bound in \eqref{eq:ubuendpointmoments} follows from
\eqref{eq:ubuconditionalmoments}. The second bound in \eqref{eq:ubuendpointmoments}
follows from the first bound, \eqref{eq:ubucoarselocal}, and the triangle
inequality:
\[
 \left\|\bar S_\eta^{\rm UBU}\right\|_{L^q(\mathbb P_s)}
 \leq\left\|S_\eta^{\rm frozen}\right\|_{L^q(\mathbb P_s)}
 +\left\|\bar S_\eta^{\rm UBU}-S_\eta^{\rm frozen}\right\|_{L^q(\mathbb P_s)}
 \leq C_q(1+|s|)+C_q\eta^2(1+|s|).
\]
The mean-value theorem gives
\[
 \left|\mathcal V_{\rm u}^p(y)-\mathcal V_{\rm u}^p(y')\right|
 \leq C_p\left(1+|y|^{2p-1}+|y'|^{2p-1}\right)|y-y'|.
\]
Hence, H\"older's inequality, \eqref{eq:ubuendpointmoments} with $q=4p-2$,
and \eqref{eq:ubucoarselocal} with $q=2$ imply, for $p=1,2$,
\begin{equation}\label{eq:ubumomentperturbation}
\begin{aligned}
 \left|Q^{\rm UBU}_{t,\eta}\mathcal V_{\rm u}^p(s)
       -P^{{\rm u},\eps_t}_\eta\mathcal V_{\rm u}^p(s)\right|
 &\leq C_p\left(1+\left\|\bar S_\eta^{\rm UBU}\right\|_{L^{4p-2}(\mathbb P_s)}^{2p-1}
      +\left\|S_\eta^{\rm frozen}\right\|_{L^{4p-2}(\mathbb P_s)}^{2p-1}\right)
      \left\|\bar S_\eta^{\rm UBU}-S_\eta^{\rm frozen}\right\|_{L^2(\mathbb P_s)}\\
 &\leq C_p\eta^2\left(1+|s|^{2p}\right)
 \leq C_p\eta^2\mathcal V_{\rm u}^p(s).
\end{aligned}
\end{equation}
Combining \eqref{eq:ubumomentperturbation} and
\eqref{eq:uldfrozenmoments}, and using
$e^{-c_p\eta}\leq1-c_p\eta+(c_p^2/2)\eta^2$, gives, after increasing $C_p$,
\[
 Q^{\rm UBU}_{t,\eta}\mathcal V_{\rm u}^p
 \leq \left(1-c_p\eta+C_p\eta^2\right)\mathcal V_{\rm u}^p+C_p\eta.
\]
For all sufficiently large $k$, with
$Y_k=\E[\mathcal V_{\rm u}^p(\bar S_k)]$, this yields
\[
 Y_{k+1}\leq(1-c_p\eta_k/2)Y_k+C_p\eta_k
 \leq\max\left\{Y_k,2C_p/c_p\right\},\qquad p=1,2.
\]
The endpoint bound \eqref{eq:ubuendpointmoments} and the finite initial fourth
moment make every $Y_k$ finite. The maximum bound therefore proves
$\sup_kY_k<\infty$. By \eqref{eq:uldlyapunov}, this gives
\eqref{eq:ubumoments}.
\end{proof}

Let $\cH^{\rm u}_\eps$ denote the kinetic distorted entropy inspired by
He, Tan, and Wu \cite[Eq.~(2.7)]{HeTanWu2024}. For the fixed friction
$\gamma_{\rm u}$, let
\begin{equation}\label{eq:uldmatrix}
 M_{\rm u}:=
 \begin{pmatrix}1&1\\1&1\end{pmatrix}\otimes\Id_d,\qquad
 A_{\rm u}:=
 2+\frac{2K+(1+\gamma_{\rm u}+K)^2+1}{\gamma_{\rm u}},
\end{equation}
and define
\begin{equation}\label{eq:ulddistorted}
 \cH^{\rm u}_\eps(\nu)
 :=\int_{\R^{2d}}
 \left|\nabla_x\log\frac{\dd\nu}{\dd\pi^{\rm u}_\eps}
       +\nabla_v\log\frac{\dd\nu}{\dd\pi^{\rm u}_\eps}\right|^2\dd\nu
   +\frac{A_{\rm u}}{\eps}\KL(\nu\mid\pi^{\rm u}_\eps).
\end{equation}
For $h=\dd\nu/\dd\pi^{\rm u}_\eps$, the first term in
\eqref{eq:ulddistorted} is interpreted through weak derivatives as
\begin{equation}\label{eq:uldweakdirectionalfisher}
 \cI_{M_{\rm u}}(\nu\mid\pi^{\rm u}_\eps)
 =4\int_{\R^{2d}}|(\nabla_x+\nabla_v)\sqrt h|^2
                                      \dd\pi^{\rm u}_\eps.
\end{equation}
It is $+\infty$ if the indicated weak derivative does not belong to
$L^2(\pi^{\rm u}_\eps)$; set the functional to $+\infty$ when
$\nu\not\ll\pi^{\rm u}_\eps$. For laws with finite full relative Fisher
information, \eqref{eq:uldweakdirectionalfisher} equals the score integral
in \eqref{eq:ulddistorted}. This definition requires only the indicated
directional derivatives. The constant $A_{\rm u}$ in
\eqref{eq:uldmatrix} gives the two-block dissipation certificate
\eqref{eq:uldcertificate}.

The final Gaussian half-step makes the UBU law smooth and positive.
Its covariance also bounds the full Fisher information needed to restart
the exact ULD from a numerical iterate.

\begin{lemma}[UBU grid regularization]\label{lem:ubugaussian}
At every positive UBU grid point, the numerical law has a smooth positive
density and finite full relative Fisher information. Moreover, for all
sufficiently large $k$,
\begin{equation}\label{eq:ubugridfisher}
 \cH^{\rm u}_{\eps_{\tau_k}}\left(\mathcal L(\bar S_k)\right)
 \leq Cs_k^{3a}(\log s_k)^2.
\end{equation}
\end{lemma}

\begin{proof}
Conditionally on the state after the force kick in \eqref{eq:ubustep}, the
last $\mathcal U_{\eta/2}^{\eps_t}$ substep adds an independent centered Gaussian
with covariance $C^{\rm u}_{t,\eta}$ given by \eqref{eq:ubuflowcov} with
$r=\eta/2$. Let
$D_\eta=\operatorname{diag}(\eta^{3/2},\eta^{1/2})\otimes\Id_d$.
Taylor expansion in \eqref{eq:ubuflowcov} gives
\begin{equation}\label{eq:ubucovariancelimit}
 \eps_t^{-1}D_\eta^{-1}C^{\rm u}_{t,\eta}D_\eta^{-1}
 \longrightarrow\gamma_{\rm u}
       \begin{pmatrix}1/12&1/4\\1/4&1\end{pmatrix}\otimes\Id_d
 \quad\text{as }\eta\downarrow0.
\end{equation}
The limiting $2\times2$ matrix in \eqref{eq:ubucovariancelimit}
has positive diagonal entries and determinant
$\gamma_{\rm u}^2/48>0$. It therefore has a positive smallest eigenvalue,
which implies, for all sufficiently small $\eta$,
\begin{equation}\label{eq:ubucovbound}
 C^{\rm u}_{t,\eta}\succeq
 c\eps_t\operatorname{diag}(\eta^3,\eta)\otimes\Id_d,\qquad
 \det C^{\rm u}_{t,\eta}\geq c\eps_t^{2d}\eta^{4d}.
\end{equation}
Thus the endpoint law is a mixture of Gaussians with the same nondegenerate
covariance. Applying the conditional score formula
\eqref{eq:gaussianmixturescore} and Jensen's inequality as in
\eqref{eq:gaussianrawfisher},
now in $\R^{2d}$, with \eqref{eq:ubucovbound} gives
\begin{equation}\label{eq:ubugridrawfisher}
 \cI^{\rm raw}\left(\mathcal L(\bar S_k)\right)
 \leq\operatorname{tr}\left(\left(C^{\rm u}_{\tau_{k-1},\eta_{k-1}}\right)^{-1}\right)
 \leq\frac{C}{\eps_{\tau_{k-1}}\eta_{k-1}^3}.
\end{equation}
Let $\nu_k=\mathcal L(\bar S_k)$, write its Lebesgue density as $\rho_k$, and let
$H_{\rm u}(x,v)=U(x)+|v|^2/2$. Using \eqref{eq:ubugridrawfisher},
\eqref{eq:ubumoments}, and $|\nabla H_{\rm u}(s)|\leq C(1+|s|)$ gives
\begin{equation}\label{eq:ubugridrelativefisher}
\begin{aligned}
 \cI(\nu_k\mid\pi^{\rm u}_{\eps_{\tau_k}})
 &=\int_{\R^{2d}}
      \left|\nabla\log\rho_k+\eps_{\tau_k}^{-1}\nabla H_{\rm u}\right|^2\dd\nu_k\\
 &\leq2\cI^{\rm raw}(\nu_k)
       +C\eps_{\tau_k}^{-2}\E\left(1+\left|\bar S_k\right|^2\right)
       \leq\frac{C}{\eps_{\tau_{k-1}}\eta_{k-1}^3}
       +C\eps_{\tau_k}^{-2}.
\end{aligned}
\end{equation}
Let $Z^{\rm u}_\eps=\int_{\R^{2d}}e^{-H_{\rm u}/\eps}\dd x\,\dd v$.
The entropy calculation \eqref{eq:gaussianmixtureentropy},
with state dimension $2d$, and the decomposition
\eqref{eq:klgibbsdecomposition} with $H_{\rm u}$ give
\begin{equation}\label{eq:ubugridentropy}
\begin{aligned}
 \KL(\nu_k\mid\pi^{\rm u}_{\eps_{\tau_k}})
 &\leq-d\log(2\pi e)
       -\frac12\log\det C^{\rm u}_{\tau_{k-1},\eta_{k-1}}
       +\eps_{\tau_k}^{-1}\nu_k(H_{\rm u})
       +\log Z^{\rm u}_{\eps_{\tau_k}}\\
 &\leq C\left(1+|\log\eps_{\tau_{k-1}}|
                 +|\log\eta_{k-1}|+\eps_{\tau_k}^{-1}\right).
\end{aligned}
\end{equation}
The covariance bound \eqref{eq:ubucovbound} and the uniform moments in
Lemma~\ref{lem:ubulocal} were used in the last line. Consequently,
\[
 \KL\left(\mathcal L(\bar S_k)\mid\pi^{\rm u}_{\eps_{\tau_k}}\right)
 \leq C\log s_k.
\]
Substituting \eqref{eq:ubugridrelativefisher} and
\eqref{eq:ubugridentropy} in \eqref{eq:ulddistorted}, with
$\eta_{k-1}\asymp s_k^{-a}$ from \eqref{eq:gridasymptotics} and
$\eps_{\tau_k}^{-1}=\mathcal{O}(\log s_k)$ from
\eqref{eq:coolingassumption}, proves
\eqref{eq:ubugridfisher}. Smoothness and positivity follow from the final
Gaussian convolution.
\end{proof}

We record in the next lemma the restart decay and semigroup smoothing estimates needed to
propagate the local error over the final time window.

\begin{lemma}[ULD restart and smoothing]\label{lem:uldrestart}
Let $\rho$ be a law on $\R^{2d}$ with second moment bounded by $M_2$,
$\cH^{\rm u}_{\eps_s}(\rho)<\infty$, and finite full relative Fisher
information $\cI(\rho\mid\pi^{\rm u}_{\eps_s})<\infty$. If
$T-s\asymp T^\beta$, $s\asymp T$, and
$\beta>\vartheta$, then, for every sufficiently small $\zeta>0$,
\begin{equation}\label{eq:uldrestart}
 \cH^{\rm u}_{\eps_T}(\rho P^{\rm u}_{s,T})
 \leq C\cH^{\rm u}_{\eps_s}(\rho)
          e^{-cT^{\beta-\vartheta-\zeta}}
       +C_\zeta T^{-1+\vartheta+2\zeta},
\end{equation}
with constants uniform over such $\rho$. Moreover, for every bounded
measurable $f$ and $T\geq t+1$,
\begin{equation}\label{eq:uldsmoothing}
 \left|P^{\rm u}_{t,T}f(s)-P^{\rm u}_{t,T}f(s')\right|
 \leq C\eps_{t+1}^{-1/2}|s-s'|\|f\|_\infty.
\end{equation}
For $0\leq T-t\leq1$ and globally Lipschitz $f$,
$\operatorname{Lip}(P^{\rm u}_{t,T}f)\leq C\operatorname{Lip}(f)$.
\end{lemma}

\begin{proof}
We first prove the restart estimate for the regular initial densities
constructed in Lemma~\ref{lem:closure}. For a symmetric matrix $Q$ with
$\|Q\|_{\rm op}\leq K$, set
\[
 J_Q^{\rm u}:=
 \begin{pmatrix}0&-\Id_d\\Q&-\gamma_{\rm u}\Id_d\end{pmatrix},
 \qquad P_v:=\operatorname{diag}(0,\Id_d).
\]
Direct multiplication, using $M_{\rm u}$ and $A_{\rm u}$ from \eqref{eq:uldmatrix},
gives
\begin{align}\label{eq:uldcertificate}
 B_Q^{\rm u}
 &:=-J_Q^{\rm u}M_{\rm u}
       -M_{\rm u}(J_Q^{\rm u})^\top
       +A_{\rm u}\gamma_{\rm u}P_v\notag\\
 &=
 \begin{pmatrix}
 2\Id_d&(1+\gamma_{\rm u})\Id_d-Q\\
 (1+\gamma_{\rm u})\Id_d-Q&
 (A_{\rm u}+2)\gamma_{\rm u}\Id_d-2Q
 \end{pmatrix}\succeq\Id_{2d}.
\end{align}
Indeed, the inequality $2|r||v|(1+\gamma_{\rm u}+K)
\leq |r|^2+(1+\gamma_{\rm u}+K)^2|v|^2$ gives
\begin{align*}
 (r,v)^\top B_Q^{\rm u}(r,v)
 &\geq |r|^2+
 \left((A_{\rm u}+2)\gamma_{\rm u}-2K
       -(1+\gamma_{\rm u}+K)^2\right)|v|^2
 =|r|^2+(4\gamma_{\rm u}+1)|v|^2.
\end{align*}

Let $H_{\rm u}(x,v):=U(x)+|v|^2/2$ and
$w:=\nabla\log(\dd\nu/\dd\pi^{\rm u}_\eps)$. For a regular density evolving
at frozen temperature, the same Fisher-information differentiation as in
\eqref{eq:fisherentropyderivatives} gives
\begin{align}\label{eq:uldfrozendissipation}
 \frac{\dd}{\dd t}\cH^{\rm u}_\eps(\nu_t)
 =&-\int_{\R^{2d}}w^\top B_{\nabla^2U}^{\rm u}w\,\dd\nu_t\notag\\
 &-2\gamma_{\rm u}\eps\sum_{i=1}^d
   \int_{\R^{2d}}(\partial_{v_i}w)^\top
       M_{\rm u}(\partial_{v_i}w)\,\dd\nu_t
 \leq-\int_{\R^{2d}}|w|^2\dd\nu_t .
\end{align}
The entropy decomposition and the conditional Cauchy--Schwarz inequality
in \eqref{eq:productlsiproof}, now with one velocity block and the product
law \eqref{eq:uldgibbs}, give
\begin{equation}\label{eq:uldlsicontrol}
 \cH^{\rm u}_\eps(\nu)
 \leq\left(\|M_{\rm u}\|
       +\frac{A_{\rm u}C_{\rm LS}(\eps)}{\eps}\right)
       \int_{\R^{2d}}|w|^2\dd\nu.
\end{equation}
Consequently, for every small $\zeta>0$ and all sufficiently large $u$,
combining \eqref{eq:uldfrozendissipation},
\eqref{eq:uldlsicontrol}, \eqref{eq:lsi}, and
\eqref{eq:coolingassumption} gives the decay coefficient
\begin{equation}\label{eq:uldkappalower}
 \kappa_{\rm u}(\eps_u):=
 \left(\|M_{\rm u}\|
       +\frac{A_{\rm u}C_{\rm LS}(\eps_u)}{\eps_u}\right)^{-1}
 \geq c_\zeta(u+t_0)^{-\vartheta-\zeta}.
\end{equation}

The moment constant is uniform over the restarting laws. Let
$\nu_u=\rho P^{\rm u}_{s,u}$. The Lyapunov function in \eqref{eq:uldlyapunov}
and its drift bound \eqref{eq:uldlyapunovdrift} give
\begin{equation}\label{eq:uldrestartmoments}
 \nu_u(\mathcal V_{\rm u})
 \leq e^{-c_0(u-s)}\rho(\mathcal V_{\rm u})+C_0/c_0
 \leq C(1+M_2),
 \qquad
 \sup_{u\geq s}\nu_u(|x|^2+|v|^2)\leq C(1+M_2).
\end{equation}
The cutoff integration by parts used in \eqref{eq:gibbsmoments}
gives $\pi^{\rm u}_\eps(x\cdot\nabla U)=d\eps$ and
$\pi^{\rm u}_\eps(|v|^2)=d\eps$. Assumption~\ref{ass:potential}
therefore implies
$\pi^{\rm u}_\eps(|x|^2)\leq(d\eps+m)/r$ and
$\pi^{\rm u}_\eps(H_{\rm u})\leq C$, uniformly for $\eps\leq\eps_0$.
Together with \eqref{eq:uldrestartmoments}, this gives
\begin{equation}\label{eq:uldrestartenergy}
 \sup_{u\geq s}
 \left(\nu_u(H_{\rm u})+\pi^{\rm u}_{\eps_u}(H_{\rm u})\right)
 \leq C(1+M_2).
\end{equation}

For a fixed law, \eqref{eq:uldgibbs} gives
$\partial_\eps\log Z^{\rm u}_\eps
=\eps^{-2}\pi^{\rm u}_\eps(H_{\rm u})$,
by the dominated-differentiation argument in
\eqref{eq:gibbsnormalizationderivative}. Differentiating
\eqref{eq:ulddistorted} therefore gives
\begin{align}\label{eq:uldtemperatureidentity}
 \partial_\eps\cH^{\rm u}_\eps(\nu)
 =-\frac{2}{\eps^2}\int_{\R^{2d}}
       w^\top M_{\rm u}\nabla H_{\rm u}\,\dd\nu
 -\frac{A_{\rm u}}{\eps^2}\KL(\nu\mid\pi^{\rm u}_\eps)
 -\frac{A_{\rm u}}{\eps^3}
       \left(\nu(H_{\rm u})-\pi^{\rm u}_\eps(H_{\rm u})\right).
\end{align}
Applying the Cauchy--Schwarz inequality to \eqref{eq:uldtemperatureidentity},
using \eqref{eq:uldrestartenergy} and \eqref{eq:ulddistorted}, and noting that
$\KL(\nu\mid\pi^{\rm u}_\eps)\leq\eps\cH^{\rm u}_\eps(\nu)/A_{\rm u}$, gives
\begin{equation}\label{eq:uldcoolingbound}
\begin{aligned}
 \left|\eps_u'\left.\partial_\eps\cH^{\rm u}_\eps(\nu_u)
                         \right|_{\eps=\eps_u}\right|
 &\leq C|\eps_u'|(1+\eps_u^{-3})
       \left(1+\cH^{\rm u}_{\eps_u}(\nu_u)\right)
 \leq C\frac{\log(u+t_0)}{u+t_0}
       \left(1+\cH^{\rm u}_{\eps_u}(\nu_u)\right).
\end{aligned}
\end{equation}
Let $h(u)=C\log(u+t_0)/(u+t_0)$ be the coefficient in
\eqref{eq:uldcoolingbound}. For $\vartheta+\zeta<1$,
\[
 \frac{h(u)}{\kappa_{\rm u}(\eps_u)}
 \leq C_\zeta\log(u+t_0)(u+t_0)^{-1+\vartheta+\zeta}
 \longrightarrow0.
\]
Thus $h(u)\leq\kappa_{\rm u}(\eps_u)/2$ after a starting time depending
only on the common moment bound. Combining
\eqref{eq:uldfrozendissipation}, \eqref{eq:uldlsicontrol},
\eqref{eq:uldcoolingbound}, and \eqref{eq:uldkappalower} then gives
\begin{equation}\label{eq:uldrestartode}
 \frac{\dd}{\dd u}\cH^{\rm u}_{\eps_u}(\nu_u)
 \leq-c_\zeta(u+t_0)^{-\vartheta-\zeta}
             \cH^{\rm u}_{\eps_u}(\nu_u)
       +C_\zeta(u+t_0)^{-1+\zeta}.
\end{equation}
This is the fixed-friction version of the distorted-entropy calculation of
He, Tan, and Wu \cite[Sections~3.2--3.4]{HeTanWu2024}.

Let $a_\zeta(u)=c_\zeta(u+t_0)^{-\vartheta-\zeta}$ and
$b_\zeta(u)=C_\zeta(u+t_0)^{-1+\zeta}$.
Multiplying \eqref{eq:uldrestartode} by
$\exp(\int_s^u a_\zeta(r)\dd r)$ and integrating gives
\begin{equation}\label{eq:uldrestartintegral}
 \cH^{\rm u}_{\eps_T}(\nu_T)
 \leq e^{-\int_s^T a_\zeta(r)\dd r}\cH^{\rm u}_{\eps_s}(\rho)
       +\int_s^T e^{-\int_u^T a_\zeta(r)\dd r}b_\zeta(u)\dd u.
\end{equation}
Since $s\asymp T$ and $T-s\asymp T^\beta$,
\begin{equation}\label{eq:uldwindowdissipation}
 \int_s^T a_\zeta(r)\dd r\geq cT^{\beta-\vartheta-\zeta},
\end{equation}
and
\begin{equation}\label{eq:uldwindowsource}
 \int_s^T e^{-\int_u^T a_\zeta(r)\dd r}b_\zeta(u)\dd u
 \leq C_\zeta T^{-1+\zeta}\int_0^\infty
       e^{-cT^{-\vartheta-\zeta}r}\dd r
 \leq C_\zeta T^{-1+\vartheta+2\zeta}.
\end{equation}
Substituting \eqref{eq:uldwindowdissipation} and
\eqref{eq:uldwindowsource} in \eqref{eq:uldrestartintegral}
proves \eqref{eq:uldrestart} for regular initial densities.

We now pass from regular densities to the initial law $\rho$ in the lemma.
Write $\rho=f^2\dd y$. The full relative Fisher information and second
moment imply $f\in H^1(\R^{2d})$, since
$|\nabla H_{\rm u}(y)|\leq C(1+|y|)$; this is the weak-gradient estimate
\eqref{eq:rawfisherfinite} in dimension $2d$ with $H_{\rm u}$ in place of $H$.
The cutoff and smoothing estimates \eqref{eq:initialcutoffapprox},
\eqref{eq:h1approx}, and \eqref{eq:weightedapprox}, in dimension $2d$,
give regular laws $\rho_n=f_n^2\dd y$ satisfying
\begin{equation}\label{eq:uldinitialapproximation}
 f_n\to f\ \text{in }H^1(\R^{2d}),\qquad
 \int_{\R^{2d}}(1+|y|^2)|f_n^2-f^2|\dd y\to0.
\end{equation}
For $d=1$, \cite[Corollary~9.11]{Brezis2011} gives, for $2\leq p<\infty$,
\[
 \|f_n-f\|_{L^p(\R^2)}
 \leq C_p\|f_n-f\|_{H^1(\R^2)}\longrightarrow0
\]
by \eqref{eq:uldinitialapproximation}. Choosing $p>2$ provides the
$L^{1+\delta}$ convergence of $f_n^2$, with $\delta=p/2-1>0$, needed in
\eqref{eq:entropyapprox}. The negative entropy estimate
\eqref{eq:negativeentropypointwise} and the weighted convergence in
\eqref{eq:uldinitialapproximation} control the tails in dimension $2$.
For the Fisher term,
the expansion and cutoff integration by parts in
\eqref{eq:fishercrosscutoff} give
\begin{equation}\label{eq:uldinitialfisherexpansion}
\begin{aligned}
 \cI_{M_{\rm u}}(\rho_n\mid\pi^{\rm u}_{\eps_s})
 =&4\int_{\R^{2d}}(\nabla f_n)^\top M_{\rm u}\nabla f_n\dd y\\
 &-\frac{2}{\eps_s}\int_{\R^{2d}}f_n^2
                     \operatorname{tr}\left(M_{\rm u}\nabla^2H_{\rm u}\right)\dd y+\frac1{\eps_s^2}\int_{\R^{2d}}f_n^2
                     (\nabla H_{\rm u})^\top M_{\rm u}\nabla H_{\rm u}\dd y.
\end{aligned}
\end{equation}
The three integrals converge by \eqref{eq:uldinitialapproximation},
the bounded Hessian, and the linear gradient bound. Combining
\eqref{eq:uldinitialfisherexpansion} with the entropy convergence yields
\begin{equation}\label{eq:uldinitialfunctionalconvergence}
 F_n(s):=\cH^{\rm u}_{\eps_s}(\rho_n)
       \longrightarrow\cH^{\rm u}_{\eps_s}(\rho).
\end{equation}

Let $\nu_u^n=\rho_nP^{\rm u}_{s,u}$ and
$F_n(u)=\cH^{\rm u}_{\eps_u}(\nu_u^n)$.
To justify differentiation for each regular approximation, use the spatial
cutoffs \eqref{eq:spatialcutoffs} in dimension $2d$. For
$h=\dd\nu/\dd\pi^{\rm u}_\eps$ and $w=\nabla\log h$, define
\begin{equation}\label{eq:uldlocalizedfunctional}
 \cH^{\rm u}_{\eps,R}(\nu)
 :=\int_{\R^{2d}}\chi_R
       \left(w^\top M_{\rm u}w+\frac{A_{\rm u}}{\eps}\log h\right)\dd\nu,
 \qquad F_{n,R}(u):=\cH^{\rm u}_{\eps_u,R}(\nu_u^n).
\end{equation}
The product and integration-by-parts formulas
\eqref{eq:cutoffproduct}--\eqref{eq:cutoffouparts}, applied to
\eqref{eq:uldlocalizedfunctional} with diffusion in $v$,
give the localized identity
\begin{equation}\label{eq:uldlocalizedidentity}
 \frac{\dd}{\dd u}F_{n,R}(u)
 =-\mathcal D^{\rm u}_{n,R}(u)+\mathcal R^{\rm u}_{n,R}(u)
  +\eps_u'\left.\partial_\eps
       \cH^{\rm u}_{\eps,R}(\nu_u^n)\right|_{\eps=\eps_u}.
\end{equation}
Here $\mathcal D^{\rm u}_{n,R}$ is the sum of the two nonnegative integrals
on the right of \eqref{eq:uldfrozendissipation}, with cutoff $\chi_R$.
The conditional density calculation \eqref{eq:flowdensity} and the derivative
bounds \eqref{eq:logdensitybounds} apply to the ULD flow: its drift has
bounded first derivative and divergence $-\gamma_{\rm u}d$.
The cutoff derivatives are supported on $R\leq|y|\leq2R$.
Thus the tail bounds used in \eqref{eq:localizedremainderbound} and
\eqref{eq:localizationconvergence}, applied to each fixed regular $\rho_n$,
give, on every finite interval $[s,T]$,
\begin{equation}\label{eq:uldlocalizationlimit}
\begin{aligned}
 &\sup_{s\leq u\leq T}|F_{n,R}(u)-F_n(u)|\to0,\\
 &\int_s^T\left(|\mathcal R^{\rm u}_{n,R}|
       +|\mathcal D^{\rm u}_{n,R}-\mathcal D^{\rm u}_n|
       +|\eps_u'\partial_\eps
          [\cH^{\rm u}_{\eps,R}(\nu_u^n)
                   -\cH^{\rm u}_{\eps}(\nu_u^n)]_{\eps=\eps_u}|\right)\dd u
 \leq C_{T,\rho_n}R^{-2}.
\end{aligned}
\end{equation}
Using \eqref{eq:uldlocalizationlimit} to pass $R\to\infty$ in the
integrated form of \eqref{eq:uldlocalizedidentity} justifies the absolutely continuous
identity for $F_n$, and hence \eqref{eq:uldrestartode} for each $\rho_n$.
Its constants are uniform in $n$: \eqref{eq:uldinitialapproximation}
and \eqref{eq:uldrestartmoments} give a common second-moment bound.
Before the low-temperature LSI applies, the nonpositive frozen derivative
in \eqref{eq:uldfrozendissipation} and the first bound in
\eqref{eq:uldcoolingbound} give
\[
 F_n'(u)\leq C_T(1+F_n(u)),\qquad
 \sup_n\sup_{s\leq u\leq T}F_n(u)
 \leq e^{C_T(T-s)}\left(1+\sup_nF_n(s)\right)<\infty.
\]

For lower semicontinuity, we also establish full Fisher information at
the limiting law. Let $\Psi^{\rm u}_{s,u}$ be the inverse ULD stochastic
flow. Its bounded drift Jacobian gives
$\|D\Psi^{\rm u}_{s,u}\|_{\rm op}\leq e^{C(u-s)}$, and its constant
divergence gives $\det D\Psi^{\rm u}_{s,u}
=e^{\gamma_{\rm u}d(u-s)}$. The conditional density formula and the
conditional Cauchy--Schwarz inequality used in \eqref{eq:flowrawfisher}
therefore give
\begin{equation}\label{eq:uldflowfullfisher}
 \cI^{\rm raw}(\nu_u^n)
 \leq e^{2C(u-s)}\cI^{\rm raw}(\rho_n),
 \qquad
 \sup_n\sup_{s\leq u\leq T}
       \cI(\nu_u^n\mid\pi^{\rm u}_{\eps_u})<\infty.
\end{equation}
The second bound uses \eqref{eq:uldinitialapproximation},
\eqref{eq:uldrestartmoments}, and
$\inf_{s\leq u\leq T}\eps_u>0$.
Synchronous coupling and \eqref{eq:uldinitialapproximation} give
\begin{equation}\label{eq:uldflowwasserstein}
 \mathcal W_2\left(\nu_u^n,\rho P^{\rm u}_{s,u}\right)
 \leq e^{C(u-s)}\mathcal W_2(\rho_n,\rho)\longrightarrow0.
\end{equation}
For fixed $u\in[s,T]$, let $g_n:=\sqrt{\dd\nu_u^n/\dd y}$.
By \eqref{eq:uldflowfullfisher} and \eqref{eq:uldinitialapproximation},
\[
 \|g_n\|_{H^1(\R^{2d})}^2
 =1+\frac14\cI^{\rm raw}(\nu_u^n)\leq C_T.
\]
The weak subsequence theorem \cite[Theorem~3.18]{Brezis2011} and the
Rellich--Kondrachov compactness theorem \cite[Theorem~9.16]{Brezis2011}
apply to $H^1(\R^{2d})$ and to its restrictions to the bounded smooth
balls $B_R\subset\R^{2d}$, respectively. The diagonal construction
leading to \eqref{eq:limitingdensityHone} gives
\[
 g_{n_j}\rightharpoonup g\text{ in }H^1(\R^{2d}),\qquad
 \|g_{n_j}-g\|_{L^2(B_R)}\longrightarrow0\quad(R\geq1).
\]
Equation \eqref{eq:uldflowwasserstein} identifies
$\rho P^{\rm u}_{s,u}=g^2\dd y$. Weak lower semicontinuity,
\eqref{eq:uldrestartmoments}, and \eqref{eq:rawrelativefisherbound}
with $H_{\rm u}$ in place of $H$ give
$\cI(\rho P^{\rm u}_{s,u}\mid\pi^{\rm u}_{\eps_u})<\infty$.

For this class of laws, weak integration by parts and completion of the
square give
\begin{equation}\label{eq:uldfisherdual}
\begin{aligned}
 \cI_{M_{\rm u}}(\nu\mid\pi^{\rm u}_\eps)
 &=\sup_{\phi\in C_c^\infty(\R^{2d};\R^{2d})}
       \int_{\R^{2d}}G_\phi^{\rm u}\dd\nu,\\
 G_\phi^{\rm u}
 &:=-2\operatorname{div}(M_{\rm u}\phi)
     +\frac2\eps\phi^\top M_{\rm u}\nabla H_{\rm u}
     -\phi^\top M_{\rm u}\phi.
\end{aligned}
\end{equation}
Indeed, the integral equals
$\cI_{M_{\rm u}}-\int_{\R^{2d}}(w-\phi)^\top
M_{\rm u}(w-\phi)\dd\nu$. Smooth compactly supported approximation of
$w$ in $L^2(\nu)$ makes the second integral tend to zero.
This uses only $M_{\rm u}\succeq0$ and its finite operator norm.
Each $G_\phi^{\rm u}$ is bounded and continuous, so weak convergence,
\eqref{eq:uldfisherdual}, and lower semicontinuity of relative entropy imply
\begin{equation}\label{eq:uldfunctionallowerlimit}
 \cH^{\rm u}_{\eps_T}(\rho P^{\rm u}_{s,T})
 \leq\liminf_{n\to\infty}F_n(T).
\end{equation}
Apply \eqref{eq:uldrestartintegral} to $F_n$. Its integral term is independent
of $n$, and \eqref{eq:uldinitialfunctionalconvergence} gives convergence
of its initial term. Taking the lower limit, then using
\eqref{eq:uldfunctionallowerlimit}, \eqref{eq:uldwindowdissipation},
and \eqref{eq:uldwindowsource}, proves \eqref{eq:uldrestart} for $\rho$.

For smoothing, write $s'-s=(e_x,e_v)$ and let
\[
 q(u)=(1-3u^2+2u^3)e_x+(u-2u^2+u^3)e_v,\qquad 0\leq u\leq1.
\]
Then
\[
 q(0)=e_x,\quad q'(0)=e_v,\quad q(1)=q'(1)=0,
 \qquad \max_{0\leq j\leq2}\left\|q^{(j)}\right\|_\infty\leq C|s-s'|.
\]
If $(X,V)$ starts from $s$, set $X^e=X+q$ and $V^e=V+q'$. The shifted path
starts from $s'$ and has the same endpoint, while its velocity equation has
the additional drift
\[
 c_e(u)=q''(u)+\nabla U(X_u+q(u))-\nabla U(X_u)
                   +\gamma_{\rm u}q'(u).
\]
Let $b_e(u)=c_e(u)/\sqrt{2\gamma_{\rm u}\eps_{t+u}}$.
The bounded Hessian and monotonicity of the temperature give
\[
 \int_0^1|b_e(u)|^2\dd u
 \leq C\eps_{t+1}^{-1}|s-s'|^2,
 \qquad
 \E\exp\left(\frac12\int_0^1|b_e(u)|^2\dd u\right)<\infty.
\]
Thus Novikov's condition holds for every $s,s'$. Under the measure
\[
 \dd\mathbb Q_e=
 \exp\left(-\int_0^1b_e\cdot\dd B-\frac12\int_0^1|b_e|^2\dd u\right)
 \dd\mathbb P,
\]
the shifted path solves ULD from $s'$. Its endpoint agrees pathwise with
that of the unshifted process from $s$, and
\[
 \KL(\mathbb P\mid\mathbb Q_e)
 =\frac12\E\int_0^1|b_e(u)|^2\dd u
 \leq C\eps_{t+1}^{-1}|s-s'|^2.
\]
Pinsker's inequality therefore gives
\begin{equation}\label{eq:uldunitsmoothing}
 |P^{\rm u}_{t,t+1}f(s)-P^{\rm u}_{t,t+1}f(s')|
 \leq C\eps_{t+1}^{-1/2}|s-s'|\|f\|_\infty.
\end{equation}
For $T\geq t+1$, the Chapman--Kolmogorov identity gives
$P^{\rm u}_{t,T}f=P^{\rm u}_{t,t+1}(P^{\rm u}_{t+1,T}f)$.
Apply \eqref{eq:uldunitsmoothing} to $P^{\rm u}_{t+1,T}f$ and use
$\|P^{\rm u}_{t+1,T}f\|_\infty\leq\|f\|_\infty$ to obtain
\eqref{eq:uldsmoothing}. For $T-t\leq1$, synchronous coupling gives
\[
 \left|S_T^s-S_T^{s'}\right|\leq e^C|s-s'|,
 \qquad
 \left|P^{\rm u}_{t,T}f(s)-P^{\rm u}_{t,T}f(s')\right|
 \leq e^C\operatorname{Lip}(f)|s-s'|.
\]
\end{proof}

\begin{proof}[Proof of Theorem~\ref{thm:ubu}]
Fix $T=\tau_k$ and choose
\[
 \beta\in(\vartheta,\min\{1,3a/2\}),
\]
which is possible because $a>2\vartheta/3$. Choose $m$ so that
$\tau_m\leq T-T^\beta<\tau_{m+1}$.
The grid calculation \eqref{eq:comparisonwindow} gives
$T-\tau_m\asymp T^\beta$, $s_j\asymp T$, and $\eta_j\asymp T^{-a}$
throughout the comparison window. Lemmas~\ref{lem:ubulocal} and
\ref{lem:ubugaussian}, including the full Fisher-information conclusion of
the latter, allow Lemma~\ref{lem:uldrestart} to be applied at $\tau_m$ and give
\begin{equation}\label{eq:uburestartbound}
 \cH^{\rm u}_{\eps_T}\left(\nu_mP^{\rm u}_{\tau_m,T}\right)
 \leq CT^{3a}(\log T)^2e^{-cT^{\beta-\vartheta-\zeta}}
       +C_\zeta T^{-1+\vartheta+2\zeta},
\end{equation}
where $\nu_m=\mathcal L(\bar S_m)$.

Choose $0<\zeta_1<\delta$, $0<\zeta_2<\delta-\zeta_1$, and
$0<\zeta<\min\{\beta-\vartheta,(1-\vartheta)/2\}$.
Use $f(x,v)=\chi_{\delta,\zeta_1}(U(x))$ from
\eqref{eq:eventapproximation}, with the auxiliary $z$ block omitted.
The smoothing estimate \eqref{eq:uldsmoothing} gives, for
$g_j=P^{\rm u}_{\tau_{j+1},T}f$,
\begin{equation}\label{eq:ubuwindowlip}
 \operatorname{Lip}(g_j)
 \leq C_{\delta,\zeta_1}\max\left\{1,\eps_{\tau_{j+1}+1}^{-1/2}\right\}
 \leq C_{\delta,\zeta_1}\sqrt{\log T}.
\end{equation}
The short terminal intervals use the Lipschitz regularity of $f$.
Telescoping the kernels gives
\begin{equation}\label{eq:ubutelescopingidentity}
 \nu_kf-\nu_mP^{\rm u}_{\tau_m,T}f
 =\sum_{j=m}^{k-1}\nu_j
       \left(Q^{\rm UBU}_{\tau_j,\eta_j}-P^{\rm u}_{\tau_j,\tau_{j+1}}\right)g_j.
\end{equation}
For each summand, \eqref{eq:ubuwindowlip} and the synchronous
coupling bound \eqref{eq:ubulocal} give
\begin{equation}\label{eq:ubukernelerror}
\begin{aligned}
 \left|\left(Q^{\rm UBU}_{\tau_j,\eta_j}-P^{\rm u}_{\tau_j,\tau_{j+1}}\right)g_j(s)\right|
 &\leq\operatorname{Lip}(g_j)
       \left\|\bar S_{\eta_j}^{\rm UBU}-S_{\tau_{j+1}}^{\rm exact}\right\|_{L^2(\mathbb P_s)}\\
 &\leq C_{\delta,\zeta_1}\left(
       \eta_j^{5/2}+\sqrt{\log T}\,\eta_j^3(1+|s|^2)
       +\frac{\eta_j^{3/2}}{T\log T}\right).
\end{aligned}
\end{equation}
Here $\sqrt{\log T}\sqrt{\eps_{\tau_j}}\leq C$ throughout the window.
Integrating \eqref{eq:ubukernelerror} against $\nu_j$, using
\eqref{eq:ubumoments}, and substituting in
\eqref{eq:ubutelescopingidentity}, we apply
\eqref{eq:windowsums} with $r=5/2,3,3/2$ gives
\begin{align}\label{eq:ubutelescoping}
 \left|\nu_kf-\nu_mP^{\rm u}_{\tau_m,T}f\right|
 &\leq C\sum_{j=m}^{k-1}\left(
       \eta_j^{5/2}+\sqrt{\log T}\,\eta_j^3
       +\frac{\eta_j^{3/2}}{T\log T}\right)\notag\\
 &\leq C\left(
       T^{\beta-3a/2}
       +\sqrt{\log T}\,T^{\beta-2a}
       +\frac{T^{\beta-1-a/2}}{\log T}\right).
\end{align}

By \eqref{eq:eventapproximation} and \eqref{eq:ulddistorted},
Pinsker's inequality gives
\[
 \nu_mP^{\rm u}_{\tau_m,T}f
 \leq\pi^{\rm u}_{\eps_T}(U>\delta-\zeta_1)
 +\sqrt{\frac{\eps_T}{2A_{\rm u}}\,
        \cH^{\rm u}_{\eps_T}\left(\nu_mP^{\rm u}_{\tau_m,T}\right)}.
\]
Substituting \eqref{eq:uburestartbound} and the Gibbs tail calculation
\eqref{eq:restartgibbstail} gives
\begin{equation}\label{eq:uburestartprobability}
 \nu_mP^{\rm u}_{\tau_m,T}f
 \leq CT^{-(\delta-\zeta_1-\zeta_2)/E}
 +CT^{3a/2}\log T\,e^{-cT^{\beta-\vartheta-\zeta}/2}
 +C_\zeta T^{-(1-\vartheta)/2+\zeta}.
\end{equation}
Here \eqref{eq:gibbstail} applies to $\pi^{\rm u}_{\eps_T}$ because its
position marginal is also $\mu_{\eps_T}$. The square root of the first term
in \eqref{eq:uburestartbound} gives the stretched-exponential term.
Since $\beta-\vartheta-\zeta>0$, the same argument as after
\eqref{eq:stretchedexponential} gives, for every fixed $N$,
\begin{equation}\label{eq:ubustretchedexponential}
 CT^{3a/2}\log T\,e^{-cT^{\beta-\vartheta-\zeta}/2}
 \leq C_NT^{-N},\qquad T\geq1.
\end{equation}
Combining \eqref{eq:uburestartprobability},
\eqref{eq:ubustretchedexponential}, \eqref{eq:eventapproximation},
and \eqref{eq:ubutelescoping} gives
\begin{equation}\label{eq:ubupolynomialprobability}
\begin{aligned}
 \mathbb P\left(U(\bar X_k)>\delta\right)\leq C\bigg(&
 T^{-(\delta-\zeta_1-\zeta_2)/E}
 +T^{-(1-\vartheta)/2+\zeta}
 +T^{-(3a/2-\beta)}\\
 &+\sqrt{\log T}\,T^{-(2a-\beta)}
 +(\log T)^{-1}T^{-(1+a/2-\beta)}+T^{-N}\bigg).
\end{aligned}
\end{equation}
Let $\beta\downarrow\vartheta$ and
$\zeta_1,\zeta_2,\zeta\downarrow0$. The exponent $2a-\vartheta$ is larger than
$3a/2-\vartheta$. The temperature-variation exponent is also nonbinding,
since
\[
 1+\frac a2-\vartheta-\frac{1-\vartheta}{2}
 =\frac{1+a-\vartheta}{2}>0.
\]
For a given $\alpha>0$, choose
\[
 0<\beta-\vartheta<\min\{\alpha/4,(1-\vartheta)/4,
                          (3a/2-\vartheta)/2\},\quad
 0<\zeta<\min\{(\beta-\vartheta)/2,\alpha/4\},\quad
 \frac{\zeta_1+\zeta_2}{E}<\alpha/4.
\]
Using $\sqrt{\log T}\leq T^{\alpha/4}$ for large $T$ and
$N>r_{\rm u}(\delta,a)$, each term in
\eqref{eq:ubupolynomialprobability} is bounded by
$CT^{-r_{\rm u}(\delta,a)+\alpha}$. This proves
\eqref{eq:uburate}; \eqref{eq:gridasymptotics} gives the iteration statement.
\end{proof}

\begin{remark}[Local error comparison between UBU and the third-order scheme]
\label{rem:ubunoisetiming}
Both UBU and \eqref{eq:midpoint} use one interior force evaluation.
Their linear position noises have different orders. From
\eqref{eq:ubuflownoise} and \eqref{eq:endpointnoise}, define
\[
 N^{x,{\rm u}}_u
 =\sqrt{2\gamma_{\rm u}\eps}\int_0^uF_{u-r}\dd B_r,\qquad
 N^{x,3}_u
 =\sqrt{2\gamma\eps}\int_0^uk_x(u-r)\dd B_r.
\]
Since $F_r\leq r$ and $|k_x(r)|\leq Cr^2$ by
\eqref{eq:noisekernels}, It\^o's isometry gives
\begin{equation}\label{eq:positionnoiseorders}
\begin{aligned}
 \|N^{x,{\rm u}}_u\|_{L^2}^2
 &\leq2\gamma_{\rm u}\eps d\int_0^ur^2\dd r
 =\frac{2\gamma_{\rm u}\eps d}{3}u^3,\\
 \|N^{x,3}_u\|_{L^2}^2
 &\leq C\eps\int_0^ur^4\dd r\leq C\eps u^5.
\end{aligned}
\end{equation}
The UBU force kick occurs at $u=\eta/2$, and
\eqref{eq:ubualphavbound} identifies its centered local term of order
$\sqrt\eps\,\eta^{5/2}$. For the third-order interpolation
\eqref{eq:stages}, let $\widetilde X_u^0$ be the position component obtained
by setting the noise terms to zero, with the same initial state $(x,v,z)$. Then
$\widetilde X_u-\widetilde X_u^0=N^{x,3}_u$, so
\begin{equation}\label{eq:thirdordernoiseforce}
 \left\|\int_0^\eta
       \left[\nabla U\left(\widetilde X_u\right)-\nabla U\left(\widetilde X_u^0\right)\right]\dd u\right\|_{L^2}
 \leq K\int_0^\eta\left\|N^{x,3}_u\right\|_{L^2}\dd u
 \leq C\sqrt\eps\,\eta^{7/2}.
\end{equation}
The noise orders \eqref{eq:positionnoiseorders} imply
\eqref{eq:thirdordernoiseforce}, which is smaller than the cubic endpoint bound \eqref{eq:frozenendpointerror}.
The step-size comparison uses the sufficient conditions proved by the
present strong-coupling argument. A weak-error or martingale analysis
that exploits the centering in \eqref{eq:ubulocaldecomposition} could
produce a sharper UBU bound.
\end{remark}

\end{document}